%% file: Adaptive_clustering_v8_arxiv.tex
\documentclass[11pt,journal,onecolumn,draftcls]{IEEEtran}
\input MyDefinitions.tex

\begin{document}
\title{Distributed Clustering and Learning Over Networks}

\author{Xiaochuan~Zhao,~\IEEEmembership{Student~Member,~IEEE,}
        and Ali~H.~Sayed,~\IEEEmembership{Fellow,~IEEE}
\thanks{The authors are with Department of Electrical Engineering, University of California, Los Angeles, CA 90095
Emails: xiaochuanzhao@ucla.edu and sayed@ee.ucla.edu.}% <-this % stops a space
\thanks{This work was supported by NSF grants CCF-1011918 and ECCS-1407712. A short and limited early version of this work appeared in the conference proceedings \cite{Zhao12CIP}.}}

%\markboth{IEEE~TRANSACTIONS~ON~SIGNAL~PROCESSING,~VOL.~xx,~NO.~X,~MONTH~YEAR}%
%{Zhao and Sayed: Asynchronous Adaptation and Learning over Networks --- Part I: Modeling and Stability Analysis}

\maketitle
% ------

\begin{abstract}
Distributed processing over networks relies on in-network processing and cooperation among neighboring agents. Cooperation is beneficial when agents share a common objective. However, in many applications agents may belong to different clusters that pursue different objectives. Then, indiscriminate cooperation will lead to undesired results. In this work, we propose an adaptive clustering and learning scheme that allows agents to learn which neighbors they should cooperate with and which other neighbors they should ignore. In doing so, the resulting algorithm enables the agents to identify their clusters and to attain improved learning and estimation accuracy over networks. We carry out a detailed mean-square analysis and assess the error probabilities of Types I and II, i.e., false alarm and mis-detection, for the clustering mechanism. Among other results, we establish that these probabilities decay exponentially with the step-sizes so that the probability of correct clustering can be made arbitrarily close to one.
\end{abstract}


\allowdisplaybreaks

\section{Introduction}
Distributed algorithms for learning, inference, modeling, and optimization by networked agents are prevalent in many domains and applicable to a wide range of problems \cite{Sayed14PROC, Sayed14NOW, Dimakis10PROC, Saber07Pro}. Among the various classes of algorithms, techniques that are based on first-order gradient-descent iterations are particularly useful for distributed processing due to their low complexity, low power demands, and robustness against imperfections or unmodeled effects. Three of the most studied classes are consensus algorithms \cite{Olfati04TAC, Saber07Pro, Kar09TSP, Nedic09TAC, Kar11TSP}, diffusion algorithms \cite{Lopes08TSP, Cattivelli10TSP, Chen12TSP, Zhao12TSP2, Zhao12TSP, Sayed13Chapter, Sayed13SPM, Sayed14PROC}, and incremental algorithms \cite{Bertsekas97JOP, Nedic01JOP, Rabbat05JSAC, Lopes07TSP, Helou09SIAMOPT, Johansson09SIAMOPT}. The incremental techniques rely on the determination of a Hamiltonian cycle over the topology, which is generally an NP-hard problem and is therefore a hindrance to real-time adaptation, and even more so when the topology is dynamic and changes with time. For this reason, we will consider mainly learning algorithms of the consensus and diffusion types.

In this work we focus on the case in which \emph{constant} step-sizes are employed in order to enable \emph{continuous} adaptation and learning in response to streaming data. When diminishing step-sizes are used, the algorithms would cease to adapt after the step-sizes have approached zero, which is problematic for applications that require the network to remain continually vigilant and to track possible drifts in the data and clusters. Therefore, adaptation with constant step-sizes is necessary in these scenarios. It turns out that when constant step-sizes are used, the dynamics of the distributed (consensus or diffusion) strategies are modified in a non-trivial manner: the stochastic gradient noise that is present in their update steps does not die out anymore and it seeps into the operation of the algorithms. In other words, while this noise component would be annihilated by decaying step-sizes, it will remain persistently active during constant step-size adaptation. As such, it becomes important to evaluate how well constant step-size implementations can alleviate the influence of gradient noise. It was shown in \cite{Tu12TSP, Sayed14PROC, Sayed14NOW} that consensus strategies can become problematic when constant step-sizes are employed. This is because of an asymmetry in their update relations, which can cause the state of the network to grow unbounded when these networks are used for adaptation. In comparison, diffusion networks do not suffer from this asymmetry problem and have been shown to be mean stable regardless of the topology of the network. This is a reassuring property, especially in the context of applications where the topology can undergo changes over time. These observations motivate us to focus our analysis on diffusion strategies, although the conclusions and arguments can be extended with proper adjustments to consensus strategies.

Now, most existing works on distributed learning algorithms focus on the case in which all agents in the network are interested in estimating a common parameter vector, which generally corresponds to the minimizer of some aggregate cost function (see, e.g., \cite{Sayed14PROC, Sayed14NOW, Dimakis10PROC, Saber07Pro} and the references therein). In this article, we are instead interested in scenarios where different clusters of agents within the network are interested in estimating different parameter vectors. There have been several useful works in this domain in the literature under various assumptions, including in the earlier version of this work in \cite{Zhao12CIP}. This early investigation dealt only with the case of two separate clusters in the network with each cluster interested in one parameter vector. One useful application of this formulation in the context of biological networks was considered in \cite{Tu14TSP}, where each agent was assumed to collect data arising from one of two models (e.g., the location of two separate food sources). The agents did not know which model generated their observations and, yet, they needed to reach agreement about which model to follow (i.e., which food source to move towards). Another important extension dealing with multiple (more than two) models appears in \cite{ChenJ14TSP, ChenJ14TSP2} where multi-task problems are introduced. In this formulation, different clusters of the agents are again interested in estimating different parameter vectors (called ``tasks'') and the tasks of adjacent clusters are further assumed to be related to each other so that cooperation among clusters can still be beneficial. This formulation is useful in many scenarios, as already illustrated in \cite{ChenJ14TSP}, including in multiple target tracking \cite{Liu07MSP, Zhang11TAC} and classification problems involving multiple models \cite{Duda01, Francis74AS, Li96TAC, Cherkassky05TNN, Theodoridis09, Jacob09NIPS}. Other useful variations of multi-task problems appear in \cite{Bertrand10TSP}, which assumes fully-connected networks, and in \cite{Bogdanovic14ICASSP} where the agents have two types of parameters to estimate (a local parameter and a global parameter). These various works focus on mean-square-error (MSE) design, where the parameters of interest are estimated by seeking the minimizer of an MSE cost. Moreover, with the exception of \cite{Zhao12CIP, ChenJ14TSP2}, it is generally assumed in these works that the agents know beforehand which clusters they belong to or which parameters they are interested in estimating.

In this article, we extend the approach of \cite{Zhao12CIP} and study multi-tasking adaptive networks  under three conditions that are fundamentally different from previous studies. First, we go beyond mean-square-error estimation and allow for more general convex risk functions at the agents. This level of generality allows the framework to handle broader situations both in adaptation and learning, such as logistic regression for pattern classification purposes. Second, we do not assume any relation among the different objectives pursued by the clusters. In other words, we study the important problem where different components of the network are truly interested in different objectives and would like to avoid interference among clusters. And third, the agents do not know beforehand which clusters they belong to and which other agents are interested in the same objective.

For example, in an application involving a sensor network tracking multiple moving objects from various directions, it is reasonable to assume that the trajectories of these objects are independent of each other. In this case, only information shared within clusters is beneficial for learning; the information from agents in other clusters would amount to interference. This means that agents would need to cooperate with neighbors that belong to the same cluster and would need to cut their links to neighbors with different objectives. This task would be simple to achieve if agents were aware of their cluster information. However, we will not be making that assumption. The cluster information will need to be learned as well. This point highlights one major feature of our formulation: we do not assume that agents have full knowledge about their clusters. This assumption is quite common in the context of unsupervised machine learning \cite{Duda01, Theodoridis09}, where the collected measurement data are not labeled and there are multiple candidate models. If two neighboring agents are interested in the same model and they are aware of this fact, then they should exchange data and cooperate. However, the agents may not know this fact, so they cannot be certain about whether or not they should cooperate. Accordingly, in this work, we will devise an adaptive clustering and learning strategy that allows agents to learn which neighbors they should cooperate with. In doing so, the resulting algorithm enables the agents in a network to be correctly clustered and to attain improved learning performance through enhanced intra-cluster cooperation.

\emph{Notation}: We use lowercase letters to denote vectors, uppercase letters for matrices, plain letters for deterministic variables, and boldface letters for random variables. We also use $(\cdot)^{\T}$ to denote transposition, $(\cdot)^{-1}$ for matrix inversion, $\Tr(\cdot)$ for the trace of a matrix, and $\| \cdot \|$ for the 2-norm of a matrix or the Euclidean norm of a vector. Besides, we use $A \kron B$ for matrices $A$ and $B$ to denote their Kronecker product, $A \ge B$ to demote that $A-B$ is positive semi-definite, and $A \succeq B$ to demote that all entries of $A-B$ are nonnegative.

\section{Problem Formulation}
\label{sec:problem}
We consider a network consisting of $N$ agents inter-connected via some topology. An individual cost function, $J_k(w):\mbbR^{M\times 1} \mapsto \mbbR$, of a vector parameter $w$, is associated with every agent $k$. Each cost $J_k(w)$ is assumed to be strictly-convex and is minimized at a unique point $w_k^o$. According to the minimizers $\{ w_k^o \}$, agents in the network are categorized into $Q \ge 2$ mutually-exclusive clusters, denoted by $\Ccal_q$, $q = 1, 2, \dots, Q$.

\begin{definition}[Cluster]
\label{def:cluster}
Each cluster $q$, denoted by $\Ccal_q$, consists of the collection of agents whose individual costs share the common minimizer $w_q^\star$, i.e., $w_k^o = w_q^\star$ for all $k \in \Ccal_q$. \hfill \IEEEQED
\end{definition}

Since agents from different clusters do not share common minimizers, the network then aims to solve the \emph{clustered} multi-task problem:
\be
\label{eqn:Jclusterdef}
\minimize_{\{ w_q \}_{q = 1}^Q } \quad J(w_1, \dots, w_Q) \defeq \sum_{q=1}^{Q} \sum_{k \in \Ccal_q} J_k(w_q)
\ee

\noindent If the cluster information $\{\Ccal_q\}$ is available to the agents, then problem \eqref{eqn:Jclusterdef} can be decomposed into $Q$ separate optimization problems over the sub-networks associated with the clusters:
\be
\label{eqn:Jclusterqdef}
\minimize_w \quad J_q^\clst(w) \defeq \sum_{k \in \Ccal_q} J_k(w) 
\ee
for $q = 1, 2, \dots, Q$. Assuming the cluster topologies are connected, the corresponding minimizers $\{w_q^\star \}$ can be sought by employing diffusion strategies over each cluster. In this case, collaborative learning will only occur \emph{within} each cluster without any interaction across clusters. This means that for every agent $k$ that belongs to a particular cluster $\Ccal_q$, i.e., $k \in \Ccal_q$, its neighbors, which belong to the set denoted by $\Ncal_k$, will need to be segmented into two sets: one set is denoted by $\Ncal_k^+$ and consists of neighbors that belong to the same cluster $\Ccal_q$, and the other set is denoted by $\Ncal_k^-$ and consists of neighbors that belong to other clusters. It is clear that
\be
\label{eqn:Neighborhood+and-def}
\Ncal_k^+ \defeq \Ncal_k \cap \Ccal_q, \qquad \qquad \Ncal_k^- \defeq \Ncal_k \backslash \Ncal_k^+
\ee
We illustrate a two-cluster network with a total of $N = 20$ agents in Fig. \ref{fig:illustration_init_tot}. The agents in the clusters are denoted by blue and red circles, and are inter-connected by the underlying topology, so that agents may have in-cluster neighbors as well as neighbors from other clusters. For example, agent $k$ from blue cluster $\Ccal_1$ has the in-cluster sub-neighborhood $\Ncal_k^+ = \{k, 3, 4\}$, which is a subset of its neighborhood $\Ncal_k = \{ k, 1, 2, 3, 4, 5 \}$. If the cluster information is available to all agents, then the network can be split into two sub-networks, one for each cluster, as illustrated in Figs. \ref{fig:illustration_res_blue} and \ref{fig:illustration_res_red}.

However, in this work we consider the more challenging scenario in which the cluster information $\{\Ccal_q\}$ is only \emph{partially} available to the agents beforehand, or even completely unavailable. When the cluster information is completely absent, each agent $k$ must first identify neighbors belonging to $\Ncal_k^+$. When the cluster information is partially known, meaning that some agents from the same cluster already know each other, then these agents can cooperate to identify the other members in their cluster. In order to study these two scenarios in a uniform manner, we introduce the concept of a group.

\begin{definition}[Group]
\label{def:group}
A group $m$, denoted by $\Gcal_m$, is a collection of connected agents from the same cluster and knowing that they belong to this same cluster. \hfill \IEEEQED
\end{definition}

\begin{figure}[h]
\centerline{
\subfloat[The underlying topology.]
{\includegraphics[width=2.5in]{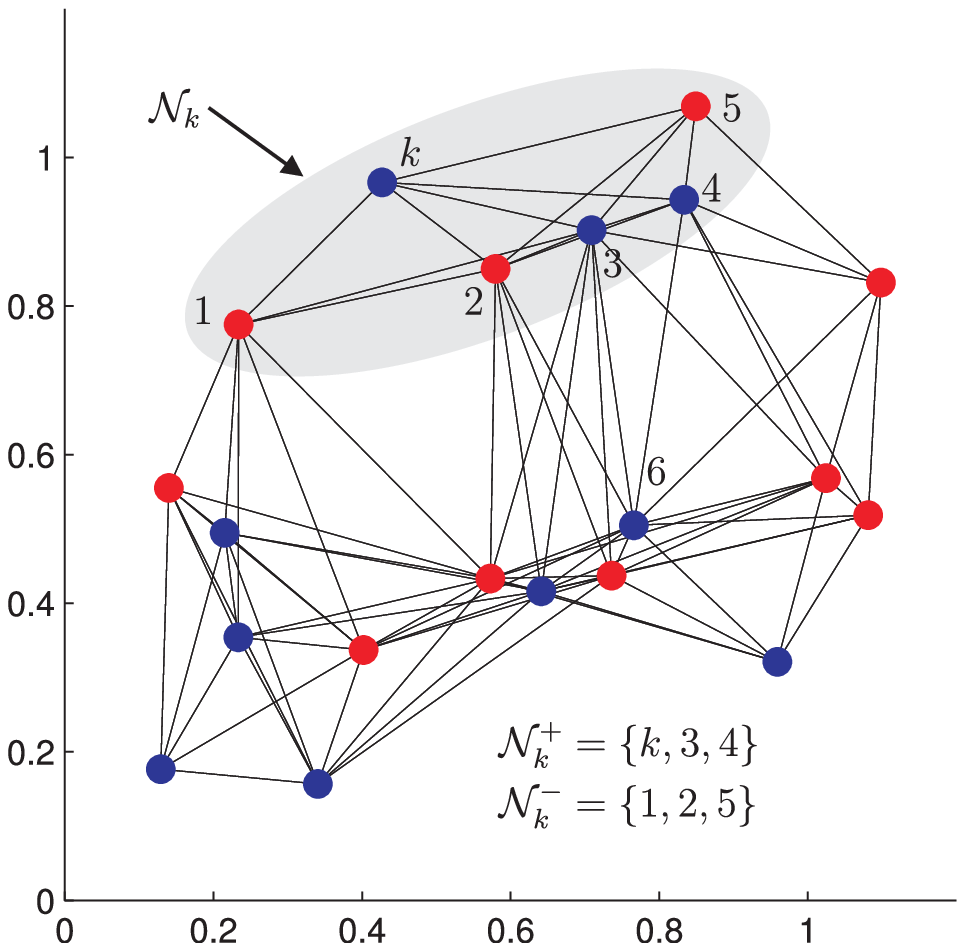}
\label{fig:illustration_init_tot}}
\hfil
\subfloat[The clustered topology for $\Ccal_1$.]
{\includegraphics[width=2.5in]{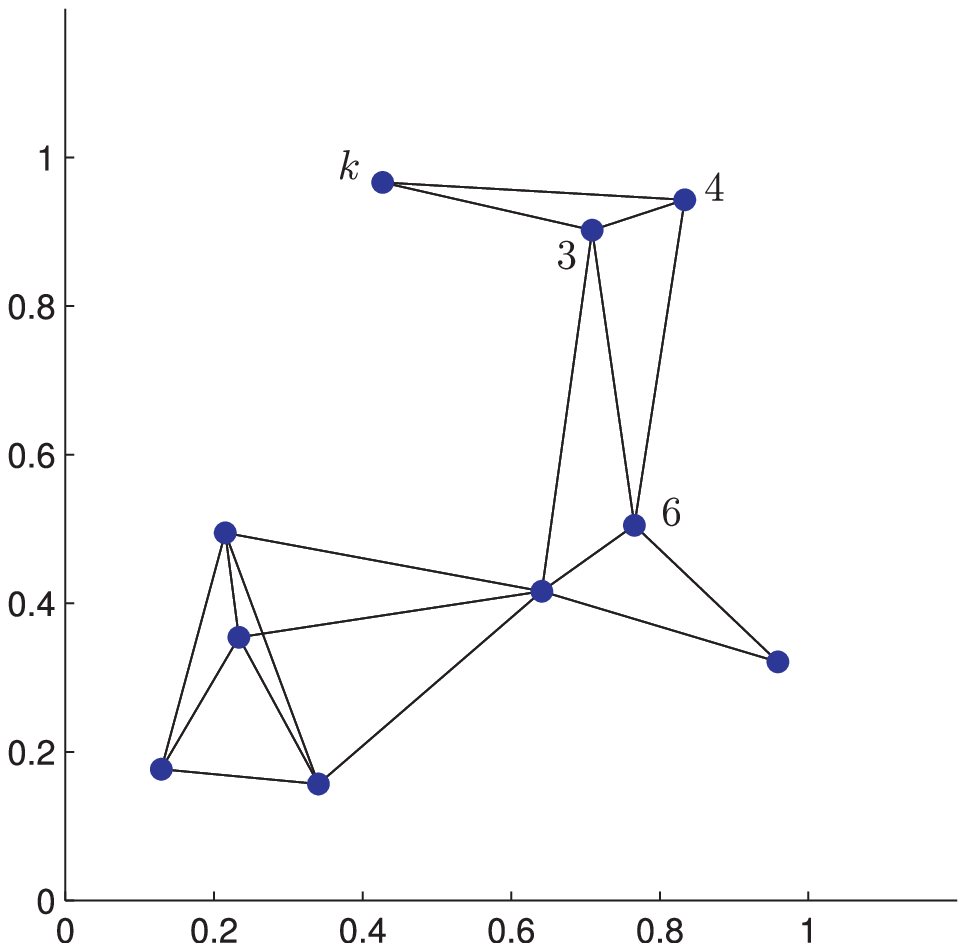}
\label{fig:illustration_res_blue}}}
\centerline{
\subfloat[The clustered topology for $\Ccal_2$.]
{\includegraphics[width=2.5in]{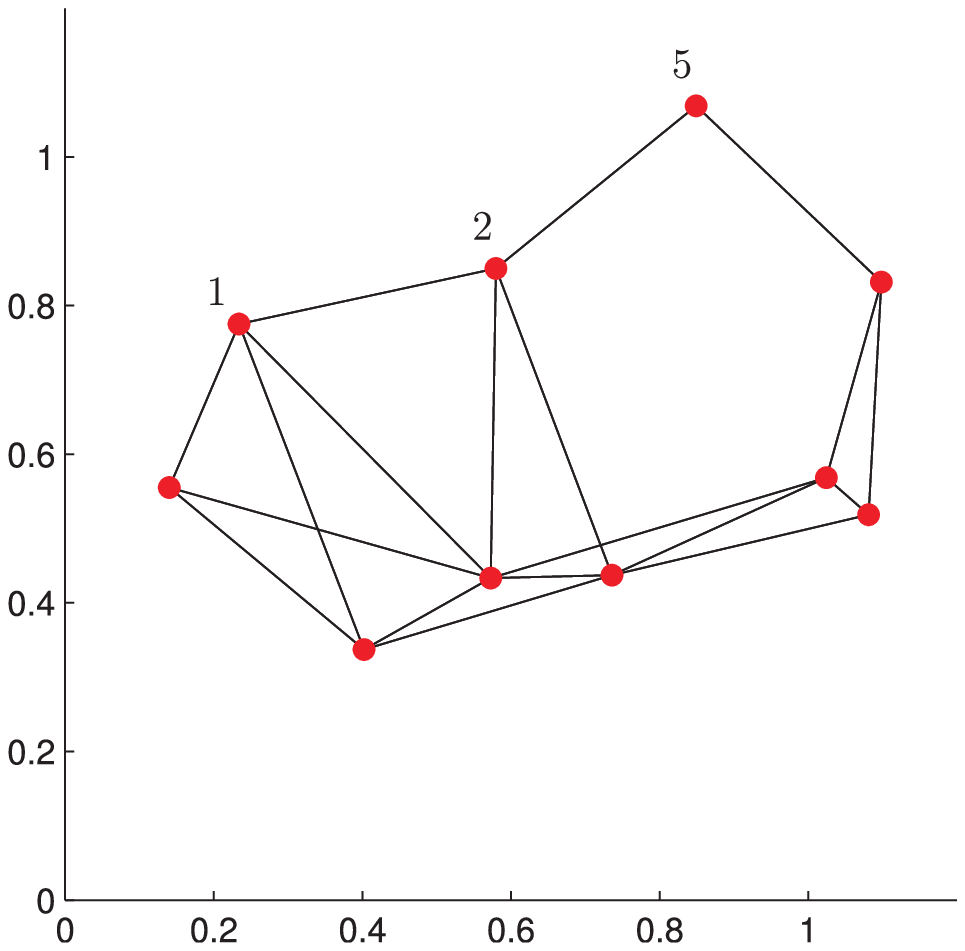}
\label{fig:illustration_res_red}}
\hfil
\subfloat[Five groups from cluster $\Ccal_1$.]
{\includegraphics[width=2.5in]{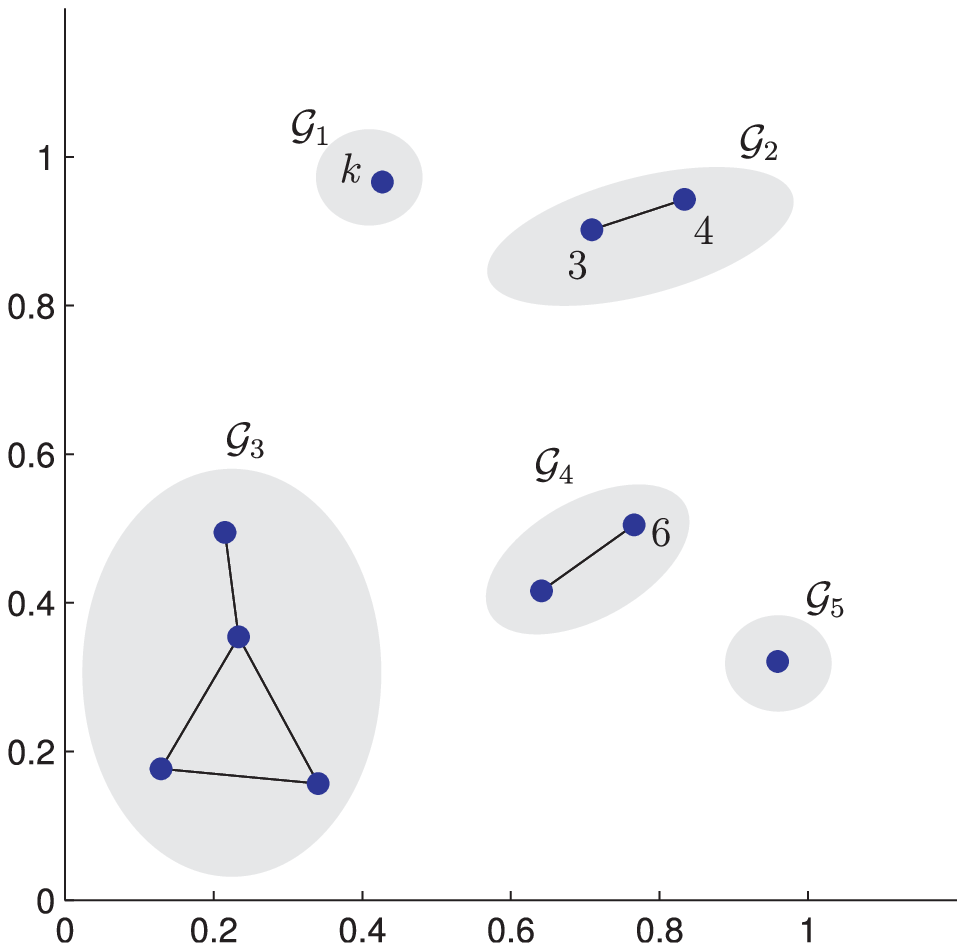}
\label{fig:illustration_init_blue}}
}
\caption{A network with $N=20$ nodes and $Q = 2$ clusters. Cluster $\Ccal_1$ consists of 10 agents in blue. Cluster $\Ccal_2$ consists of another 10 agents in red. Agent $k$ belongs to Cluster $\Ccal_1$, and its neighborhood is denoted by $\Ncal_k = \{ k, 1, 2, 3, 4, 5 \}$ with $\Ncal_k^+ = \{k, 3, 4\}$. With perfect cluster information, the underlying topology splits into two sub-networks, one for each cluster. With partial cluster information, cluster $\Ccal_1$ breaks down into five groups: two singleton groups $\Gcal_1$ and $\Gcal_5$, and three non-trivial groups $\Gcal_2$, $\Gcal_3$, and $\Gcal_4$. Through adaptive learning and clustering, the five groups in (b) will end up merging into one largest group corresponding to the entire cluster in (c).}
\label{fig:network}
\vspace{-1\baselineskip}
\end{figure}

Figure \ref{fig:illustration_init_blue} illustrates the concept of groups when cluster information is only partially available to the agents in the network from Fig. \ref{fig:illustration_init_tot}. If an agent has no information about its neighbors, then it falls into a singleton group, such as groups $\Gcal_1$ and $\Gcal_5$ in Fig. \ref{fig:illustration_init_blue}. If some neighboring agents know the cluster information of each other, then they form a non-trivial group, such as groups $\Gcal_2$, $\Gcal_3$, and $\Gcal_4$. If every agent in a cluster knows the cluster information of all its neighbors, then all cluster members form one group and this group coincides with the cluster itself, as shown in Fig. \ref{fig:illustration_res_blue}.

Since cooperation among neighbors belonging to different clusters can lead to biased results \cite{Chen13JSTSP, Sayed14NOW, ChenJ14TSP}, agents should only cooperate within clusters. However, when agents have access to partial cluster information, then they only know their group neighbors but not \emph{all} cluster neighbors. Therefore, at this stage, agents can only cooperate within groups, leaving behind some potential opportunity for cooperation with neighbors from the same cluster. The purpose of this work is to devise a procedure to enable agents to identify all of their cluster neighbors, such that small groups from the same cluster can merge automatically into larger groups. At the same time, the procedure needs to be able to turn off links between different clusters in order to avoid interference. By using such a procedure, agents in multi-task networks with \emph{partial} cluster information will be able to cluster themselves in an \emph{adaptive} manner, and then solve problem \eqref{eqn:Jclusterdef} by solving \eqref{eqn:Jclusterqdef} collaboratively \emph{within} each cluster. We shall examine closely the probability of successful clustering and evaluate the steady-state mean-square-error performance for the overall learning process. In particular, we will show that the probability of correct clustering approaches one for sufficiently small step-sizes. We will also show that, with the enhanced cooperation that results from adaptive clustering, the mean-square-error performance for the network will be improved relative to the network without adaptive clustering.

\section{Models and Assumptions}

We summarize the main conditions on the network topology in the following statement.

\begin{assumption}[Topology, clusters, and groups] \hfill
\label{ass:topology} 
\begin{enumerate}
\item The network consists of $Q$ clusters, $\{\Ccal_q; q = 1, 2, \dots, Q\}$. The size of cluster $\Ccal_q$ is denoted by $N_q^c$ such that $| \Ccal_q | = N_q^c$ and $\sum_{q=1}^{Q} N_q^c = N$.

\item The underlying topology for each cluster $\Ccal_q$ is connected. Clusters are also inter-connected by some links so that agents from different clusters may still be neighbors of each other.

\item There is a total of $G$ groups, $\{\Gcal_m; m = 1, 2, \dots, G\}$, in the network. The size of group $\Gcal_m$ is denoted by $N_m^g$ such that $|\Gcal_m| = N_m^g$ and $\sum_{m=1}^G N_m^g = N$. \hfill \IEEEQED
\end{enumerate} 
\end{assumption}

\noindent It is obvious that $Q \le G \le N$ because each cluster has at least one group and each group has at least one agent.

\begin{definition}[Indexing rule]
\label{def:index}
Without loss of generality, we index groups according to their cluster indexes such that groups from the same cluster will have consecutive indexes. Likewise, we index agents according to their group indexes such that agents from the same group will have consecutive indexes. \hfill \IEEEQED
\end{definition}

According to this indexing rule, if group $\Gcal_m$ belongs to cluster $\Ccal_q$, then the next group $\Gcal_{m+1}$ will belong either to cluster $\Ccal_q$ or the next cluster, $\Ccal_{q+1}$; if agent $k$ belongs to group $\Gcal_m$, then the next agent $k+1$ will belong either to group $\Gcal_m$ or the next group, $\Gcal_{m+1}$.

Based on the problem formulation in Section \ref{sec:problem}, although agents in the same cluster are  connected, they are generally not aware of each other's cluster information, and therefore some agents in the same cluster may not cooperate in the initial stage of adaptation. On the other hand, agents in the same group are aware of each other's cluster information, so these agents can cooperate. As the learning process proceeds, agents from different groups in the same cluster will recognize each other through information sharing. Once cluster information is inferred, small groups will merge into larger groups, and agents will start cooperating with more neighbors. Through this adaptive clustering procedure, cooperative learning will grow until all agents within the same cluster become cooperative and the network performance is enhanced.

To proceed with the modeling assumptions, we introduce the following network Hessian matrix function:
\be
\label{eqn:bigHwdef}
\nabla^2 J( \sw ) \defeq \diag\{ \nabla^2 J_1(w_1), \dots, \nabla^2 J_N(w_N) \}
\ee
where the vector $\sw$ collects the parameters from across the network:
\be
\label{eqn:wnetdef}
\sw \defeq \col\{ w_1, \dots, w_N \} \in \mbbR^{NM\times 1}
\ee
We also collect the individual minimizers into a vector:
\be
\label{eqn:wonetworkdef}
\sw^o \defeq \col\{w_1^o, \dots, w_N^o\} = \col\{ \one_{N_q^c} \kron w_q^\star ; q = 1, \dots, Q \}
\ee
where the second equality is due to the indexing rule in Definition \ref{def:index}, and $\one_n$ denotes an $n \times 1$ vector with all its entries equal to one. We next list two standard assumptions for stochastic distributed learning over adaptive networks to guide the subsequent analysis in this work. One assumption relates to the analytical properties of the cost functions, and is meant to ensure well-defined minima and well-posed problems. The second assumption relates to stochastic properties of the gradient noise processes that result from approximating the true gradient vectors. This assumption is meant to ensure that the gradient approximations are unbiased and with moments satisfying some regularity conditions. Explanations and motivation for these assumptions in the context of inference problems can be found in \cite{Polyak87, Sayed14PROC, Sayed14NOW}.

\begin{assumption}[Cost functions] \hfill
\label{ass:costfunctions}
\begin{enumerate}
\item Each individual cost $J_k(w)$ is assumed to be strictly-convex, twice-differentiable, and with bounded Hessian matrix function satisfying:
\be
\lambda_{k,L} I_M  \le \nabla^2 J_k(w) \le \lambda_{k,U} I_M
\ee
where $0 \le \lambda_{k,L} \le \lambda_{k,U} < \infty$.

\item In each group $\Gcal_m$, at least one individual cost, say, $J_{k^o}(w)$, is strongly-convex, meaning that the lower bound, $\lambda_{k^o, L}$, on the Hessian of this cost is positive.

\item The network Hessian function $\nabla^2 J(\sw)$ in \eqref{eqn:bigHwdef} satisfies the Lipschitz condition:
\be
\label{eqn:lipschitzHessian}
\| \nabla^2 J(\sw_1) - \nabla^2 J(\sw_2) \| \le \kappa_H \| \sw_1 - \sw_2 \|
\ee
for any $\sw_1, \sw_2 \in \mbbR^{NM\times 1}$ and some $\kappa_H \ge 0$. \hfill \IEEEQED
\end{enumerate}
\end{assumption}

\noindent The second set of assumptions relate to conditions on the gradient noise processes. For this purpose, we introduce the filtration $\{\F_i;i \ge 0\}$ to represent the information flow that is available up to the $i$-th iteration of the learning process. The true network gradient function and its stochastic approximation are respectively denoted by
\begin{align}
\label{eqn:biggdef}
\nabla J(\sw) & \defeq \col\{ \nabla J_1(w_1), \dots, \nabla J_N(w_N) \} \\
\label{eqn:biggapproxdef}
\wh{\nabla J}(\sw) & \defeq \col\{ \wh{\nabla  J_1}(w_1), \dots, \wh{\nabla  J_N}(w_N) \}
\end{align}
The gradient noise at iteration $i$ and agent $k$ is denoted by:
\be
\label{eqn:additivegradienterror}
\bm{s}_{k,i}(\bm{w}_{k,i-1}) \defeq \wh{\nabla J_k}(\bm{w}_{k,i-1}) - \nabla J_k(\bm{w}_{k,i-1})
\ee
where $\bm{w}_{k,i-1}$ denotes the estimate for $w_k^o$ that is available to agent $k$ at iteration $i-1$. The network gradient noise is denoted by $\ssb_i(\swb_{i-1})$ and is the random process that is obtained by aggregating all noise processes from across the network into a vector:
\be
\label{eqn:bigsidef}
\ssb_i(\swb_{i-1}) \defeq \col\{ \bm{s}_{1,i}(\bm{w}_{1,i-1}), \dots, \bm{s}_{N,i}(\bm{w}_{N,i-1}) \}
\ee
Using \eqref{eqn:additivegradienterror}, we can write
\be
\label{eqn:additivegradienterrornetwork}
\wh{\nabla J}(\swb_{i-1}) =  \nabla J(\swb_{i-1}) + \ssb_i(\swb_{i-1})
\ee
We denote the conditional covariance of $\ssb_i(\swb_{i-1})$ by
\be
\label{eqn:bigRsidef}
\Rcal_{s,i}(\swb_{i-1}) \defeq \E [ \ssb_i(\swb_{i-1}) \ssb_i^\T(\swb_{i-1}) | \F_{i-1}]
\ee
where $\swb_{i-1}$ is in $\F_{i-1}$.

\begin{assumption}[Gradient noise]
\label{ass:gradienterrors}
It is assumed that the gradient noise process satisfies the following properties for any $\swb_{i-1}$ in $\F_{i-1}$:
\begin{enumerate}
\item Martingale difference \cite{Kushner03, Sayed14NOW}:
\be
\label{eqn:martingaledifference}
\E [ \ssb_i(\swb_{i-1}) | \F_{i-1} ] = 0
\ee

\item Bounded fourth-order moment \cite{Chen13TIT, Zhao13TSPasync1, Sayed14NOW}:
\be
\label{eqn:bounded4thorder}
\E [ \| \ssb_i(\swb_{i-1}) \|^4 | \F_{i-1} ] \le \alpha^2 \| \sw^o - \swb_{i-1} \|^4 + \sigma_s^4
\ee
for some $\alpha, \sigma_s \ge 0$, and where $\sw^o$ is from \eqref{eqn:wonetworkdef}.

\item Lipschitz conditional covariance function \cite{Chen13TIT, Zhao13TSPasync1, Sayed14NOW}:
\be
\label{eqn:lipschitzcovariance}
\| \Rcal_{s,i}(\sw^o) - \Rcal_{s,i}(\swb_{i-1}) \| \le \kappa_s \| \sw^o - \swb_{i-1} \|^{\gamma_s}
\ee
for some $\kappa_s \ge 0$ and $0 < \gamma_s \le 4$.

\item Convergent conditional covariance matrix \cite{Kushner03, Chen13TIT, Zhao13TSPasync1, Sayed14NOW}:
\be
\label{eqn:convergentcovariance}
\Rcal_s \defeq \lim_{i \rightarrow \infty} \Rcal_{s,i}(\sw^o) > 0
\ee
where $\Rcal_s$ is symmetric and positive definite. \hfill \IEEEQED
\end{enumerate}
\end{assumption}

It is easy to verify from \eqref{eqn:bounded4thorder} that the second-order moment of the gradient noise process also satisfies:
\be
\label{eqn:boundedvariance}
\E [ \| \ssb_i(\swb_{i-1}) \|^2 | \F_{i-1} ] \le \alpha \| \sw^o - \swb_{i-1} \|^2 + \sigma_s^2
\ee

\section{Proposed Algorithm and Main Results}
\label{sec:algorithm}

In order to minimize all cluster cost functions $\{ J_q^\clst(w); q = 1, 2, \dots, Q \}$ defined by \eqref{eqn:Jclusterqdef}, agents need to cooperate only within their clusters. Although cluster information is in general not available beforehand, groups within each cluster are available according to Assumption \ref{ass:topology}. Therefore, based on this prior information, agents can instead focus on solving the following problem based on partitioning by groups rather than by clusters:
\be
\label{eqn:Jgroupdef}
\minimize_{\{ w_m \}_{m = 1}^G } \quad J'(w_1, \dots, w_G) \defeq \sum_{m = 1}^G \sum_{k \in \Gcal_m} J_k(w_m)
\ee
with one parameter vector $w_m$ for each group $\Gcal_m$. In the extreme case when prior clustering information is totally absent, groups will collapse into singletons and problem \eqref{eqn:Jgroupdef} will reduce to the individual non-cooperative case with each agent running its own stochastic-gradient algorithm to minimize its cost function. In another extreme case when cluster information is completely available, groups will be equivalent to clusters and problem \eqref{eqn:Jgroupdef} will reduce to the formation in \eqref{eqn:Jclusterdef}. Therefore, problem \eqref{eqn:Jgroupdef} is general and includes many scenarios of interest as special cases. We shall argue in the sequel that during the process of solving \eqref{eqn:Jgroupdef}, agents will be able to gradually learn their neighbors' clustering information. This information will be exploited by a \emph{separate} learning procedure by each group to dynamically involve more neighbors (from outside the group) in local cooperation. In this way, we will be able to establish analytically that, with high probability, agents will be able to successfully solve problem \eqref{eqn:Jclusterdef} (and not just \eqref{eqn:Jgroupdef}) even \emph{without} having the complete clustering information in advance.

We motivate the algorithm by examining problem \eqref{eqn:Jgroupdef}. Since the groups $\{\Gcal_m\}$ are already formed and they are disjoint, problem \eqref{eqn:Jgroupdef} can be decomposed into $G$ separate optimization problems, one for each group:
\be
\label{eqn:Jgroupmdef}
\minimize_w \quad J_m^g(w) \defeq \sum_{k \in \Gcal_m} J_k(w) 
\ee
with $m = 1, 2, \dots, G$. For any agent $k$ belonging to group $\Gcal_m$ in cluster $\Ccal_q$, i.e., $k \in \Gcal_m \subseteq \Ccal_q$, it is easy to verify that
\be
\{k\} \subseteq \Ncal_k \cap \Gcal_m \subseteq \Ncal_k \cap \Ccal_q = \Ncal_k^+
\ee
Then, agents in group $\Gcal_m$ can seek the solution of $J_m^g(w)$ in \eqref{eqn:Jgroupmdef} by using the adapt-then-combine (ATC) diffusion learning strategy over $\Gcal_m$, namely,
\begin{subequations}
\begin{align}
\label{eqn:distributedadaptgroup}
\bm{\psi}_{k,i} & = \bm{w}_{k,i-1} - \mu_k \wh{\nabla J_k} (\bm{w}_{k,i-1}) \\
\label{eqn:distributedcombinegroup}
\bm{w}_{k,i} & = \sum_{\ell \in \Ncal_k \cap \Gcal_m} a_{\ell k} \bm{\psi}_{\ell,i}
\end{align}
\end{subequations}
for all $k\in\Gcal_m$, where $\mu_k > 0$ denotes the step-size parameter, and $\{ a_{\ell k} \}$ are  convex combination coefficients that satisfy
\be
\label{eqn:alkcondition}
\left\{
\begin{aligned}
a_{\ell k} > 0 & \;\; \mbox{if} \;\; \ell \in \Ncal_k \cap \Gcal_m \\
a_{\ell k} = 0 & \;\; \mbox{otherwise} 
\end{aligned}
\right., \;\; \mbox{and} \;\; \sum_{\ell = 1}^{N} a_{\ell k} = 1
\ee
Moreover, $\bm{w}_{k,i}$ denotes the random estimate computed by agent $k$ at iteration $i$, and $\bm{\psi}_{k,i}$ is the intermediate iterate. We collect the coefficients $\{a_{\ell k}\}$ into a matrix $A \defeq [a_{\ell k}]_{\ell, k = 1}^N$. Obviously, $A$ is a left-stochastic matrix, namely,
\be
A^\T \one_N = \one_N
\ee
We collect the iterates generated from \eqref{eqn:distributedadaptgroup}--\eqref{eqn:distributedcombinegroup} by group $\Gcal_m$ into a vector:
\be
\label{eqn:wqigroupdef}
\swb_{m,i} \defeq \col\{ \bm{w}_{k,i}; k \in \Gcal_m \} \in \mbbR^{N_m^g M \times 1}
\ee
where $N_m^g$ is the size of $\Gcal_m$. According to the indexing rule from Definition \ref{def:index} for agents and groups, the estimate for the entire network from \eqref{eqn:distributedadaptgroup}--\eqref{eqn:distributedcombinegroup} can be obtained by stacking the group estimates $\{\swb_{m,i}\}$:
\be
\label{eqn:winetworkdef}
\swb_i \defeq \col\{ \bm{w}_{1,i}, \dots, \bm{w}_{N,i} \} = \col\{ \swb_{1,i}, \dots, \swb_{G,i} \}
\ee

The procedure used by the agents to enlarge their groups will be based on the following results to be established in later sections. We will show in Theorem \ref{theorem:normaldistribution} that after sufficient iterations, i.e., as $i \rightarrow \infty$, and for small enough step-sizes, i.e., $\mu_k \ll 1$ for all $k$, the network estimate $\swb_i$ defined by \eqref{eqn:winetworkdef} exhibits a distribution that is \emph{nearly} Gaussian:
\be
\swb_i \sim \mbbN(\sw^o, \; \mumax \Pi)
\ee
where $\mbbN(\phi, \Psi)$ denotes a Gaussian distribution with mean $\phi$ and covariance $\Psi$, $\sw^o$ is from \eqref{eqn:wonetworkdef}, 
\be
\label{eqn:mumaxdef}
\mumax \defeq \max_{k = 1, \dots, N} \mu_k
\ee
and $\Pi \in \mbbR^{ NM \times NM }$ is a symmetric, positive semi-definite matrix, independent of $\mumax$, and defined later by \eqref{eqn:Pinetworkdef}. In addition, we will show that for any pair of agents from two different groups, for example, $k \in \Gcal_m$ and $\ell \in \Gcal_n$, where the two groups $\Gcal_m$ and $\Gcal_n$ may or may not originate from the same cluster, the difference between their estimates will also be distributed approximately according to a Gaussian distribution: 
\be
\label{eqn:wdiffGaussian}
\bm{w}_{\ell,i} - \bm{w}_{k,i} \sim \mbbN( w_\ell^o - w_k^o, \; \mumax \Delta_{\ell,k} )
\ee
where
\be
\Delta_{\ell,k} \defeq \Pi_{\ell,\ell} + \Pi_{k,k} - \Pi_{k, \ell} - \Pi_{\ell, k}
\ee
is a symmetric, positive semi-definite matrix, and $\Pi_{k,\ell}$ denotes the $(k,\ell)$-th block of $\Pi$ with block size $M \times M$. These results are useful for inferring the cluster information for agents $k$ and $\ell$. Indeed, since the covariance matrix in \eqref{eqn:wdiffGaussian} is on the order of $\mumax$, the probability density function (pdf) of $\bm{w}_{\ell,i} - \bm{w}_{k,i}$ will concentrate around its mean, namely, $w_\ell^o - w_k^o$, when $\mumax$ is sufficiently small. Therefore, if these agents belong to the same cluster such that $w_\ell^o = w_k^o$, then we will be able to conclude from \eqref{eqn:wdiffGaussian} that with high probability, $\| \bm{w}_{\ell,i} - \bm{w}_{k,i} \|^2 = O(\mumax)$. On the other hand, if the agents belong to different clusters such that $w_\ell^o \neq w_k^o$, then it will hold with high probability that $\| \bm{w}_{\ell,i} - \bm{w}_{k,i} \|^2 = O(\mumax^0)$. This observation suggests that a hypothesis test can be formulated for agents $\ell$ and $k$ to determine whether or not they are members of the same cluster: 
\be
\label{eqn:decisionrule}
\| \bm{w}_{\ell,i} - \bm{w}_{k,i} \|^2 \overset{\mbbH_0}{\underset{\mbbH_1}{\lessgtr}} \theta_{k,\ell}
\ee
where $\mbbH_0$ denotes the hypothesis $w_\ell^o = w_k^o$, $\mbbH_1$ denotes the hypothesis $w_\ell^o \neq w_k^o$, and $\theta_{k,\ell} > 0$ is a predefined threshold. Both agents $\ell$ and $k$ will test \eqref{eqn:decisionrule} to reach a symmetric pattern of cooperation. Since $\bm{w}_{k,i}$ and $\bm{w}_{\ell,i}$ are accessible through local interactions within neighborhoods, the hypothesis test \eqref{eqn:decisionrule} can be carried out in a distributed manner. We will further show that the probabilities for both types of errors incurred by \eqref{eqn:decisionrule}, i.e., the false alarm (Type-I) and the missing detection (Type-II) errors, decay at exponential rates, namely,
\begin{align*}
\mbox{Type-I:}  & \; \Pr[ \| \bm{w}_{\ell,i} - \bm{w}_{k,i} \|^2 > \theta_{k,\ell} | w_\ell^o = w_k^o ] \le O(e^{-c_1/\mumax}) \\
\mbox{Type-II:} & \; \Pr[ \| \bm{w}_{\ell,i} - \bm{w}_{k,i} \|^2 < \theta_{k,\ell} | w_\ell^o \neq w_k^o ] \le O(e^{-c_2/\mumax})
\end{align*}
for some constants $c_1 > 0$ and $c_2 > 0$. Therefore, for long enough iterations and small enough step-sizes, agents are able to successfully infer the cluster information with very high probability.

The clustering information acquired at each iteration $i$ is used by the agents to dynamically adjust their \emph{inferred} cluster neighborhoods. The $\bm{\Ncal}_{k,i}^+$ for agent $k\in\Gcal_m$ at iteration $i$ consists of the neighbors that are accepted under hypothesis $\mbbH_0$ and the other neighbors that are already in the same group:
\be
\label{eqn:neighborhoodki+def}
\bm{\Ncal}_{k,i}^+ \defeq \{\ell \in \Ncal_k; \| \bm{w}_{\ell,i} - \bm{w}_{k,i} \|^2 < \theta_{k,\ell} \;\; \mbox{or} \;\; \ell \in \Gcal_m  \}
\ee
Using these dynamically-evolving cluster neighborhoods, we introduce a \emph{separate} ATC diffusion learning strategy:
\begin{subequations}
\begin{align}
\label{eqn:distributedadaptdynamic}
\bm{\psi}_{k,i}' & = \bm{w}_{k,i-1}' - \mu_k \wh{\nabla J_k} (\bm{w}_{k,i-1}') \\
\label{eqn:distributedcombinedynamic}
\bm{w}_{k,i}' & = \sum_{\ell \in \bm{\Ncal}_{k,i-1}^+ } \bm{a}_{\ell k}'(i-1) \bm{\psi}_{\ell,i}'
\end{align}
\end{subequations}
where the combination coefficients $\{ \bm{a}_{ \ell k}'(i-1)\}$ become random because $\bm{\Ncal}_{k,i-1}^+$ is random and may vary over iterations. The iteration index $i-1$ is used for these coefficients to enforce causality. Since $\Ncal_k \cap \Gcal_m$ denotes the neighbors of agent $k$ that are already in the same group $\Gcal_m$ as $k$, it is obvious that $\Ncal_k \cap \Gcal_m \subseteq \bm{\Ncal}_{k,i-1}^+$ for any $i\ge0$. This means that recursion \eqref{eqn:distributedadaptdynamic}--\eqref{eqn:distributedcombinedynamic} generally involves a larger range of interactions among agents than the first recursion \eqref{eqn:distributedadaptgroup}--\eqref{eqn:distributedcombinegroup}. We summarize the algorithm in the following listing.

\vspace{0.5\baselineskip}
\begin{algorithmic}
\hrule
\vskip 0.1\baselineskip
\hrule
\vskip 0.2\baselineskip
\STATE \textbf{Distributed clustering and learning over networks}
\vskip 0.2\baselineskip
\hrule
\vskip 0.2\baselineskip

\STATE Initialization: $\bm{w}_{k,-1} = \bm{w}_{k,-1}' = 0$ and $\bm{\Ncal}_{k,-1}^+ = \Ncal_k \cap \Gcal_m$ for all $k \in \Gcal_m$ and $m = 1, 2, \dots, G$. 

\FOR{$i\ge0$}

\STATE (1) Each agent $k$ updates $\bm{w}_{k,i}$ according to the first recursion \eqref{eqn:distributedadaptgroup}--\eqref{eqn:distributedcombinegroup} over $\Ncal_k \cap \Gcal_m$.

\STATE (2) Each agent $k$ updates $\bm{w}_{k,i}'$ according to the second recursion \eqref{eqn:distributedadaptdynamic}--\eqref{eqn:distributedcombinedynamic} over $\bm{\Ncal}_{k,i-1}^+$.

\STATE (3) Each agent $k$ updates $\bm{\Ncal}_{k,i}^+$ by using \eqref{eqn:neighborhoodki+def} with $\{\bm{w}_{\ell,i}; \ell \in \Ncal_k\}$ from step (1).

\ENDFOR

\vskip 0.2\baselineskip
\hrule
\vskip 0.1\baselineskip
\hrule
\end{algorithmic}

\vspace{0.5\baselineskip}

\section{Mean-Square-Error Analysis}
\label{sec:performancegroup}

In the previous section, we mentioned that Theorem \ref{theorem:normaldistribution} in Section \ref{subsection:normalpdf} is the key result for the design of the clustering criterion. To arrive this theorem, we shall derive two useful intermediate results, Lemmas \ref{lemma:approxerrorrecursion} and \ref{lemma:lowrankapprox}, in this section. These two results are related to the MSE analysis of the first recursion \eqref{eqn:distributedadaptgroup}--\eqref{eqn:distributedcombinegroup}, which is used in step (1) of the proposed algorithm. We shall therefore examine the stability and the MSE performance of recursion \eqref{eqn:distributedadaptgroup}--\eqref{eqn:distributedcombinegroup} in the sequel. It is clear that the evolution of this recursion is not influenced by the other two steps. Thus, we can study recursion \eqref{eqn:distributedadaptgroup}--\eqref{eqn:distributedcombinegroup} independently.

\subsection{Network Error Recursion}
Using model \eqref{eqn:additivegradienterrornetwork}, recursion \eqref{eqn:distributedadaptgroup}--\eqref{eqn:distributedcombinegroup} leads to
\be
\label{eqn:uniformiteration}
\swb_i = \Acal^\T \swb_{i-1} - \Acal^\T \Mcal \,  \nabla J(\swb_{i-1}) - \Acal^\T \Mcal \ssb_i(\swb_{i-1})
\ee
where $\swb_i$ is from \eqref{eqn:winetworkdef}, $ \nabla J(\cdot)$ is from \eqref{eqn:biggdef}, $\ssb_i(\cdot)$ is from \eqref{eqn:bigsidef}, and 
\begin{align}
\label{eqn:bigMdef}
\Mcal & \defeq \diag\{ \mu_1, \dots, \mu_N \} \kron I_M \\
\label{eqn:bigAdef}
\Acal & \defeq A \kron I_M
\end{align}
We introduce the network error vector:
\be
\label{eqn:networkerrorvectordef}
\wt{\swb}_i \defeq \sw^o - \swb_i = \col\{ \wt{\bm{w}}_{1,i}, \dots, \wt{\bm{w}}_{N,i} \}
\ee
where $\sw^o$ is from \eqref{eqn:wonetworkdef}, and the individual error vectors:
\be
\wt{\bm{w}}_{k,i} \defeq w_k^o - \bm{w}_{k,i}
\ee
Using the mean-value theorem \cite{Polyak87, Sayed14NOW}, we can write
\be
\label{eqn:gwlinearize}
 \nabla J(\swb_{i-1}) =  \nabla J(\sw^o) - \left[ \int_{0}^{1} \nabla^2 J(\sw^o - t \wt{\swb}_{i-1}) dt \right] \wt{\swb}_{i-1}
\ee
where $\nabla^2 J(\cdot)$ is from \eqref{eqn:bigHwdef}. Since $\sw^o$ consists of individual minimizers throughout the network, it follows that $ \nabla J(\sw^o) = 0$. Let
\be
\label{eqn:bigbarHdef}
\bm{\Hcal}_{i-1} \defeq \int_{0}^{1} \nabla^2 J(\sw^o - t \wt{\swb}_{i-1}) dt = \diag\{ \bm{H}_{k,i-1} \}_{k=1}^{N}
\ee
where
\be
\label{eqn:Hkodef}
\bm{H}_{k,i-1} \defeq \int_{0}^{1} \nabla^2 J_k(w_k^o - t \wt{\bm{w}}_{k,i-1}) dt
\ee
Then, expression \eqref{eqn:gwlinearize} can be rewritten as
\be
\label{eqn:gwlinearizenew}
 \nabla J(\swb_{i-1}) = - \bm{\Hcal}_{i-1} \wt{\swb}_{i-1}
\ee
where it is worth noting that the random matrix $\bm{\Hcal}_{i-1}$ is dependent on $\wt{\swb}_{i-1}$. Substituting \eqref{eqn:gwlinearizenew} into \eqref{eqn:uniformiteration} yields:
\be
\label{eqn:uniformiteration1}
\swb_i = \Acal^\T \swb_{i-1} + \Acal^\T \Mcal \bm{\Hcal}_{i-1} \wt{\swb}_{i-1} - \Acal^\T \Mcal \ssb_i(\swb_{i-1})
\ee
By the indexing rule from Definition \ref{def:index} and condition \eqref{eqn:alkcondition}, the combination matrix $A$ possesses a block diagonal structure:
\be
\label{eqn:Ablockstructure}
A = \diag\{ A_m; m = 1, \dots, G \}
\ee
where each $A_m$ collects the combination coefficients within group $\Gcal_m$:
\be
A_m \defeq [a_{\ell k}; \ell, k \in \Gcal_m]
\ee
From the same condition \eqref{eqn:alkcondition}, we have that each $A_m$ is itself an 
$N_m^g \times N_m^g$ left-stochastic matrix:
\be
\label{eqn:Am1eq1}
A_m^\T \one_{N_m^g} = \one_{N_m^g}
\ee
If group $\Gcal_m$ is a subset of cluster $\Ccal_q$, then the agents in $\Gcal_m$ share the same minimizer at $w_q^\star$. Thus, for any $\Gcal_m \subseteq \Ccal_q$, let
\be
\label{eqn:wmgroupdef}
\sw_m^o \defeq \col\{w_k^o; k \in \Gcal_m \} = \one_{N_m^g} \kron w_q^\star
\ee
It follows from \eqref{eqn:Am1eq1} and \eqref{eqn:wmgroupdef} that
\be
\label{eqn:AmImwm}
(A_m^\T \kron I_M) \sw_m^o = (A_m^\T \kron I_M) (\one_{N_m^g} \kron w_q^\star) = \sw_m^o
\ee
Again, from the indexing rule in Definition \ref{def:index}, we have from \eqref{eqn:wonetworkdef} and \eqref{eqn:wmgroupdef} that
\be
\label{eqn:wogroupclusters}
\sw^o = \col\{\sw_m^o; m = 1, \dots, G \}
\ee
Then, it follows from \eqref{eqn:Ablockstructure} and  \eqref{eqn:wogroupclusters} that
\be
\label{eqn:Awoequalwo}
\Acal^\T \sw^o = \begin{bmatrix}
A_1^\T \kron I_M & & \\
& \ddots & \\
& & A_G^\T \kron I_M \\
\end{bmatrix} \begin{bmatrix}
\sw_1^o \\
\vdots \\
\sw_G^o \\
\end{bmatrix} = \sw^o
\ee
Accordingly, subtracting $\sw^o$ from both sides of \eqref{eqn:uniformiteration1} and using \eqref{eqn:Awoequalwo} yields the network error recursion:
\be
\label{eqn:uniformiteration2}
\wt{\swb}_i = \Acal^\T ( I_{NM} - \Mcal \bm{\Hcal}_{i-1} ) \wt{\swb}_{i-1} + \Acal^\T \Mcal \ssb_i(\swb_{i-1})
\ee
We denote the coefficient matrix appearing in \eqref{eqn:uniformiteration2} by
\be
\label{eqn:bigBdef}
\bm{\Bcal}_{i-1} \defeq \Acal^\T ( I_{NM} - \Mcal \bm{\Hcal}_{i-1} )
\ee
Then, the network error recursion \eqref{eqn:uniformiteration2} can be rewritten as
\be
\label{eqn:networkerrorrecursiondef}
\wt{\swb}_i = \bm{\Bcal}_{i-1} \wt{\swb}_{i-1} + \Acal^\T \Mcal \ssb_i(\swb_{i-1})
\ee
We further introduce the group quantities:
\begin{align}
\label{eqn:bigAqdef}
\Acal_m & \defeq A_m \kron I_M \\
\label{eqn:wqidef}
\swb_{m,i} & \defeq \col\{ \bm{w}_{k,i}; k \in \Gcal_m \} \in \mbbR^{N_m^g M \times 1} \\
\label{eqn:bigMqdef}
\Mcal_m & \defeq \diag\{\mu_k; k \in \Gcal_m\} \kron I_M \\
\label{eqn:bigHqdef}
\bm{\Hcal}_{m,i-1} & \defeq \diag\{ \bm{H}_{k,i-1}; k \in \Gcal_m \} \\
\label{eqn:bigsqdef}
\ssb_{m,i}(\swb_{m,i-1}) & \defeq \col\{ \bm{s}_{k,i}(\bm{w}_{k,i-1}); k \in \Gcal_m \}
\end{align}
It follows from the indexing rule in Definition \ref{def:index} that
\begin{align}
\label{eqn:bigAandAq}
\Acal & = \diag \{ \Acal_1, \dots, \Acal_G \} \\
\label{eqn:wiandwqi}
\swb_i & = \col\{ \swb_{1,i}, \dots, \swb_{G,i} \} \\
\label{eqn:bigMandMq}  
\Mcal & = \diag\{ \Mcal_1, \dots, \Mcal_G \} \\
\label{eqn:bigHandHq}
\bm{\Hcal}_{i-1} & = \diag\{ \bm{\Hcal}_{1,i-1}, \dots, \bm{\Hcal}_{G,i-1} \} \\
\label{eqn:bigsandsq}
\ssb_i(\swb_{i-1}) & = \col\{ \ssb_{1,i}(\swb_{1,i-1}), \dots, \ssb_{G,i}(\swb_{G,i-1}) \}
\end{align}
Using \eqref{eqn:bigAandAq}--\eqref{eqn:bigHandHq}, the matrix $\bm{\Bcal}_{i-1}$ in \eqref{eqn:bigBdef} can be expressed by
\be
\label{eqn:bigBandBq}
\bm{\Bcal}_{i-1} = \diag\{ \bm{\Bcal}_{1,i-1}, \dots, \bm{\Bcal}_{G,i-1} \}
\ee
where
\be
\label{eqn:bigBqdef}
\bm{\Bcal}_{m,i-1} \defeq \Acal_m^\T (I_{N_m^g M} - \Mcal_m \bm{\Hcal}_{m,i-1}) 
\ee
Due to the block structures in \eqref{eqn:bigAandAq}--\eqref{eqn:bigBandBq}, groups are isolated from each other. Therefore, using these group quantities, the network error recursion \eqref{eqn:networkerrorrecursiondef} is automatically decoupled into a total of $G$ group error recursions, where the $m$-th recursion is given by
\be
\label{eqn:networkerrorrecursiongroup}
\wt{\swb}_{m,i} = \bm{\Bcal}_{m,i-1} \wt{\swb}_{m,i-1} + \Acal_m^\T \Mcal_m \ssb_{m,i}(\swb_{m,i-1})
\ee

\subsection{Mean-Square and Mean-Fourth-Order Error Stability}
\label{subsection:stability}
The stability of the network error recursion \eqref{eqn:networkerrorrecursiondef} is now reduced to studying the stability of the group recursions \eqref{eqn:networkerrorrecursiongroup}. Recall that, by Definition \ref{def:group}, the agents in each group are connected. Moreover, condition \eqref{eqn:alkcondition} implies that agents in each group have non-trivial self-loops, meaning that $a_{kk} > 0$ for all $k \in \Gcal_m$. It follows that each $A_m$ is a primitive matrix \cite{BermanPF, Sayed14PROC} (which is satisfied as long as there exists at least one $a_{kk} > 0$ in each group). Under these conditions, we are now able to ascertain the stability of the second and fourth-order error moments of the network error recursion \eqref{eqn:networkerrorrecursiondef} by appealing to results from \cite{Sayed14NOW}. 

\begin{theorem}[Stability of error moments]
\label{theorem:stability}
For sufficiently small step-sizes, the network error recursion \eqref{eqn:networkerrorrecursiondef} is mean-square and mean-fourth-order stable in the sense that
\begin{align}
\label{eqn:varianceasymptoticbound}
\limsup_{i \rightarrow \infty} \; \E \| \wt{\swb}_i \|^2 & = O(\mumax) \\
\label{eqn:4thorderasymptoticbound}
\limsup_{i \rightarrow \infty} \; \E \| \wt{\swb}_i \|^4 & = O(\mumax^2) 
\end{align}
\end{theorem}

\begin{proof}
It is obvious that the network error recursion \eqref{eqn:networkerrorrecursiondef} is mean-square and mean-fourth-order stable if, and only if, each group error recursion \eqref{eqn:networkerrorrecursiongroup} is stable in a similar sense. From Assumption \ref{ass:costfunctions}, we know that there exists at least one strongly-convex cost in each group. Since the combination matrix $A_m$ for each group is primitive and left-stochastic, we can now call upon Theorems 9.1 and 9.2 from \cite[p. 508, p. 522]{Sayed14NOW} to conclude that every group error recursion is mean-square and mean-fourth-order stable, namely,
\begin{align}
\label{eqn:varianceasymptoticboundgroup}
\limsup_{i \rightarrow \infty} \; \E \| \wt{\swb}_{m,i} \|^2 & = O(\mumax) \\
\label{eqn:4thorderasymptoticboundgroup}
\limsup_{i \rightarrow \infty} \; \E \| \wt{\swb}_{m,i} \|^4 & = O(\mumax^2) 
\end{align}
from which \eqref{eqn:varianceasymptoticbound} and \eqref{eqn:4thorderasymptoticbound} follow.
\end{proof}

\subsection{Long-Term Model}
\label{subsection:steadystatecov}
Once network stability is established, we can proceed to assess the performance of the adaptive clustering and learning procedure. To do so, it becomes more convenient to first introduce a long-term model for the error dynamics \eqref{eqn:networkerrorrecursiondef}. Note that recursion \eqref{eqn:networkerrorrecursiondef} represents a non-linear, time-variant, and stochastic system that is driven by a state-dependent random noise process. Analysis of recursion \eqref{eqn:networkerrorrecursiondef} is facilitated by noting (see Lemma \ref{lemma:approxerrorrecursion} below) that when the step-size parameter $\mumax$ is small enough, the mean-square behavior of \eqref{eqn:networkerrorrecursiondef} in steady-state, when $i \gg 1$, can be well approximated by the behavior of the following long-term model:
\be
\label{eqn:longtermerrorrecursion}
\wt{\swb}_i^\longterm = \Bcal \; \wt{\swb}_{i-1}^\longterm + \Acal^\T \Mcal \ssb_i(\swb_{i-1})
\ee
where we replaced the random matrix $\bm{\Bcal}_{i-1}$ in \eqref{eqn:networkerrorrecursiondef} by the constant matrix 
\be
\label{eqn:bigBfixeddef}
\Bcal \defeq \Acal^\T ( I_{NM} - \Mcal \Hcal )
\ee
In \eqref{eqn:bigBfixeddef}, the matrix $\Hcal$ is defined by
\be
\label{eqn:bigHdef}
\Hcal \defeq \diag\{ H_1, \dots, H_N \}
\ee
where
\be
\label{eqn:Hkdef}
H_k \defeq \nabla^2 J_k(w_k^o)
\ee
Note that the long-term model \eqref{eqn:longtermerrorrecursion} is now a \emph{linear time-invariant} system, albeit one that continues to be driven by the \emph{same} random noise process as in \eqref{eqn:networkerrorrecursiondef}. Similarly to the original error recursion \eqref{eqn:networkerrorrecursiondef}, the long-term recursion \eqref{eqn:longtermerrorrecursion} can also be decoupled into $G$ recursions, one for each group:
\be
\label{eqn:longtermerrorrecursiongroup}
\wt{\swb}_{m,i}^\longterm = \Bcal_m \, \wt{\swb}_{m,i-1}^\longterm + \Acal_m^\T \Mcal_m \ssb_{m,i}(\swb_{m,i-1})
\ee
where
\begin{align}
\wt{\swb}_{m,i}^\longterm & \defeq \col\{ \wt{\bm{w}}_{k,i}^\longterm; k \in \Gcal_m \} \in \mbbR^{N_m^g M \times 1} \\
\Bcal_m & \defeq \Acal_m^\T ( I_{N_m^g M} - \Mcal_m \Hcal_m ) \\
\label{eqn:bigHmdef}
\Hcal_m & \defeq \diag\{ H_k; k \in \Gcal_m \} \\
\sw_m^o & \defeq \col\{ w_k^o; k \in \Gcal_m \}
\end{align}

\begin{lemma}[Accuracy of long-term model]
\label{lemma:approxerrorrecursion}
For sufficiently small step-sizes, the evolution of the long-term model \eqref{eqn:longtermerrorrecursion} is close to the original error recursion \eqref{eqn:networkerrorrecursiondef} in MSE sense:
\be
\label{eqn:longtermmodelgap}
\limsup_{i\rightarrow \infty} \E \| \wt{\swb}_i - \wt{\swb}_i^\longterm \|^2 = O(\mumax^2)
\ee
\end{lemma}

\begin{proof}
We call upon Theorem 10.2 from \cite[p. 557]{Sayed14NOW} to conclude that the difference between each group error recursion \eqref{eqn:networkerrorrecursiongroup} and its long-term model \eqref{eqn:longtermerrorrecursiongroup} satisfies:
\be
\label{eqn:longtermmodelgroupgap}
\limsup_{i\rightarrow \infty} \E \| \wt{\swb}_{m,i} - \wt{\swb}_{m,i}^\longterm \|^2 = O(\mumax^2)
\ee
for all $m$. It is then immediate to conclude that \eqref{eqn:longtermmodelgap} holds.
\end{proof}

\subsection{Low-Dimensional Model}
\label{subsection:blockstructure}
Lemma \ref{lemma:approxerrorrecursion} indicates that we can assess the MSE dynamics of the original network recursion \eqref{eqn:networkerrorrecursiondef} to first-order in $\mumax$ by working with the long-term model \eqref{eqn:longtermerrorrecursion}. It turns out that the state variable of the long-term model can be split into two parts, one consisting of the \emph{centroids} of each group and the other consisting of in-group discrepancies. The details of this splitting are not important for our current discussion but interested readers can refer to Sec. V of \cite{Chen13TIT} and Eq. (10.37) of \cite[p. 558]{Sayed14NOW} for a detailed explanation. Here we only use this fact to motivate the introduction of the low-dimensional model. Moreover, it also turns out that the first part, i.e, the part corresponding to the centroids, is the dominant component in the evolution of the error dynamics and that the evolution of the two parts (centroids and in-group discrepancies) is weakly-coupled. By retaining the first part, we can therefore arrive at a low-dimensional model that will allow us to assess performance in closed-form to first-order in $\mumax$. To arrive at the low-dimensional model, we need to exploit the eigen-structure of the combination matrix $A$, or, equivalently, that of each $A_m$.

Recall that we indicated earlier prior to the statement of Theorem \ref{theorem:stability} that each $A_m$ is a primitive and left-stochastic matrix. By the Perron-Frobenius theorem \cite{BermanPF, Pillai05SPM, Sayed14NOW}, it follows that each $A_m$ has a simple eigenvalue at one with all other eigenvalues lying strictly inside the unit circle. Moreover, if we let $p_m^g \in \mbbR^{N_m^g \times 1}$ denote the right-eigenvector of $A_m$ that is associated with the eigenvalue at one, and normalize its entries to add up to one, then the same theorem ensures that all entries of $p_m^g$ will be positive:
\be
\label{eqn:pqdef}
p_m^g \defeq \col\{ p_{m,k}^g \}_{k=1}^{N_m^g} \succ 0, \;
A_m p_m^g = p_m^g, \; 
\one_{N_m^g}^\T p_m^g = 1
\ee
where $p_{m,k}^g$ denotes the $k$-th entry of $p_m^g$. This means that we can express each $A_m$ in the form (see \eqref{eqn:Aeigendecompdef} further ahead):
\be
\label{eqn:Amrankoneexp}
A_m = p_m^g \one_{N_m^g}^\T + V_{m,R} J_{m,\epsilon} V_{m,L}^\T
\ee
for some eigenvector matrices $V_{m,R}$ and $V_{m,L}$, and where $J_{m,\epsilon}$ denotes the collection of the Jordan blocks with eigenvalues inside the unit circle and with their unit entries on the first lower sub-diagonal replaced by some arbitrarily small constant $0<\epsilon\ll1$. The first rank-one component on the RHS of \eqref{eqn:Amrankoneexp} represents the contribution by the largest eigenvalue of $A_m$, and this component will be used further ahead to describe the centroid of group $\Gcal_m$. The network Perron eigenvector is obtained by stacking the group Perron eigenvectors $\{ p_m^g \}$:
\be
\label{eqn:pdef}
p \defeq \col\{ p_1^g, \dots, p_G^g \} \defeq \col\{ p_1, \dots, p_N \}
\ee
where $p_k$ denotes the $k$-th entry of $p \in \mbbR^{N \times 1}$. According to the indexing rule from Definition \ref{def:index}, it is obvious that $p_m^g = \col\{ p_k; k \in \Gcal_m \}$.

Now, for each group $\Gcal_m$, we introduce the low-dimensional (centroid) error recursion defined by (compare with \eqref{eqn:longtermerrorrecursiongroup}):
\be
\label{eqn:lowdimensionerrorrecursiongroup}
\wt{\bm{w}}_{m,i}^\lowdim = D_m \wt{\bm{w}}_{m,i-1}^\lowdim + (p_m^g \kron I_M)^\T \Mcal_m \ssb_{m,i}(\swb_{m,i-1})
\ee
where $\wt{\bm{w}}_{m,i}^\lowdim$ is $M\times1$, and $D_m$ is $M \times M$ and defined by
\be
\label{eqn:Dqdef}
D_m \defeq I_M - \mumax \bar{H}_m
\ee
where
\begin{align}
\label{eqn:barHmdef}
\bar{H}_m & \defeq \mumax^{-1} (p_m^g \kron I_M)^\T \Mcal_m \Hcal_m (\one_{N_m^g} \kron I_M) \nn \\
& = \sum_{k\in\Gcal_m} \frac{ p_k \mu_k}{\mumax} H_k = O(\mumax^0)
\end{align}
The matrix $\bar{H}_m$ is positive definite since there is at least one Hessian matrix in $\{ H_k; k \in \Gcal_m \}$ that is positive definite according to Assumption \ref{ass:costfunctions}. We collect the low-rank recursions \eqref{eqn:lowdimensionerrorrecursiongroup} for groups into one recursion for the entire network by stacking them on top of each other:
\be
\label{eqn:lowdimensionerrorrecursion}
\wt{\swb}_i^\lowdim = \Dcal \wt{\swb}_{i-1}^\lowdim  + \Pcal^\T \Mcal \ssb_i( \swb_{i-1} )
\ee
where
\begin{align}
\label{eqn:werrorlowdimdef}
\wt{\swb}_i^\lowdim & \defeq \col\{ \wt{\bm{w}}_{1,i}^\lowdim, \dots, \wt{\bm{w}}_{G,i}^\lowdim \} \in \mbbR^{GM \times 1} \\
\label{eqn:bigDdef}
\Dcal & \defeq \diag\{ D_1, \dots, D_G \} \in \mbbR^{GM \times GM} \\
\label{eqn:bigPdef}
\Pcal & \defeq \diag\{ p_1^g, \dots, p_G^g \} \kron I_M \in \mbbR^{NM \times GM}
\end{align}
Recursion \eqref{eqn:lowdimensionerrorrecursion} describes the joint dynamics of all the centroids (one for each group). Note that the dimension of $\wt{\swb}_i^\lowdim$ in \eqref{eqn:lowdimensionerrorrecursion} is $GM$, which is lower than the dimension, $NM$, of $\wt{\swb}_i^\longterm$ in \eqref{eqn:longtermerrorrecursion} or $\wt{\swb}_i$ in \eqref{eqn:networkerrorrecursiondef}, because $G \le N$ by Assumption \ref{ass:topology}. In order to measure the difference between the dynamics of the long-term model \eqref{eqn:longtermerrorrecursion} and the low-dimensional model \eqref{eqn:lowdimensionerrorrecursion}, we expand $\wt{\swb}_i^\lowdim$ in the following manner (compare with \eqref{eqn:werrorlowdimdef}):
\begin{align}
\label{eqn:zidef}
\bar{\swb}_i^\lowdim & \defeq \col\{ \bar{\swb}_{1,i}^\lowdim, \dots, \bar{\swb}_{G,i}^\lowdim \} \in \mbbR^{NM \times 1} \\
\label{eqn:zqidef}
\bar{\swb}_{m,i}^\lowdim & \defeq \one_{N_m^g} \kron \wt{\bm{w}}_{m,i}^\lowdim \in \mbbR^{N_m^g M \times 1}
\end{align}
because $\sum_{m=1}^G N_m^g = N$ according to Assumption \ref{ass:topology}.

\begin{lemma}[Accuracy of low-dimensional model]
\label{lemma:lowrankapprox}
For sufficiently small step-sizes, the low-dimensional model \eqref{eqn:lowdimensionerrorrecursion} is close to the network long-term model \eqref{eqn:longtermerrorrecursion} in the following sense:
\be
\label{eqn:lowdimensionmodelgap}
\limsup_{i\rightarrow \infty} \E \| \wt{\swb}_i^\longterm - \bar{\swb}_i^\lowdim \|^2 = O(\mumax^2)
\ee
where $\bar{\swb}_i^\lowdim$ is given by \eqref{eqn:zidef} and is related to $\wt{\swb}_i^\lowdim$ via \eqref{eqn:zqidef}.
\end{lemma}

\begin{IEEEproof}
See Appendix \ref{app:lowrankapprox}.
\end{IEEEproof}

\begin{lemma}[Low-dimensional error covariance]
\label{lemma:lowrankerrorcov}
For sufficiently small step-sizes, the covariance matrix for $\wt{\swb}_i^\lowdim$ satisfies
\be
\label{eqn:ThetaiandTheta}
\limsup_{i\rightarrow \infty} \| \E [\wt{\swb}_i^\lowdim  (\wt{\swb}_i^\lowdim)^\T ] - \Theta \| = O(\mumax^{1+\gamma_s/2})
\ee
where $\Theta \in \mbbR^{GM \times GM}$ is symmetric, positive-definite, and uniquely solves the discrete Lyapunov equation:
\be
\label{eqn:ThetaDTLEdef}
\Theta = \Dcal \Theta \Dcal + \Pcal^\T \Mcal \Rcal_s \Mcal \Pcal
\ee
\end{lemma}

\begin{proof}
See Appendix \ref{app:lowrankerrorcov}.
\end{proof}

\subsection{Steady-State MSE Performance}
From Theorem \ref{theorem:stability}, we know that the limit superior of the MSE is bounded within $O(\mumax)$. In order to define meaningful steady-state performance metrics, we consider the case in which the step-sizes approach zero asymptotically. Results obtained in this case are representative of operation in the slow adaptation regime (see Sec. 11.2 of \cite[pp. 581--583]{Sayed14NOW}).

\begin{lemma}[Steady-state normalized MSD]
\label{lemma:msdblock}
The normalized total MSD of $\wt{\swb}_i$ in \eqref{eqn:networkerrorrecursiondef} is given by 
\be
\label{eqn:steadystateMSDdefblock}
\lim_{\mumax \rightarrow 0} \limsup_{i \rightarrow \infty} \mumax^{-1} \E \| \wt{\swb}_i \|^2 =  \sum_{m = 1}^G \frac{N_m^g}{2 \mumax} \Tr \left[ \left( \sum_{k \in \Gcal_m} p_k \mu_k H_k \right)^{-1} \left( \sum_{k \in \Gcal_m} p_k^2 \mu_k^2 R_k \right) \right] 
\ee
where $H_k$ is from \eqref{eqn:Hkdef} and $R_k$ is the $m$-th block on the diagonal of $\Rcal_s$ from \eqref{eqn:convergentcovariance} with block size $M \times M$.
\end{lemma}

\begin{proof}
The normalized total MSD is the sum of the normalized MSD for each group. From Lemma 11.3 of \cite[p. 594]{Sayed14NOW}, the normalized MSD for each group $\Gcal_m$ is given by
\be
\label{eqn:steadystateMSDdefgroup}
\lim_{\mumax \rightarrow 0} \limsup_{i \rightarrow \infty} \mumax^{-1} \E \| \wt{\swb}_{m,i} \|^2 = \frac{N_m^g}{2 \mumax} \Tr \left[ \left( \sum_{k \in \Gcal_m} p_k \mu_k H_k \right)^{-1} \left( \sum_{k \in \Gcal_m} p_k^2 \mu_k^2 R_k \right) \right]
\ee
Note that we calculate the \emph{normalized total} MSD rather than the \emph{average} MSD in \eqref{eqn:steadystateMSDdefblock} and \eqref{eqn:steadystateMSDdefgroup}.
\end{proof}

In order to examine the statistical properties of the error vector $\wt{\swb}_i$, we need to strengthen the result in Lemma \ref{lemma:msdblock} by evaluating the full normalized error covariance matrix of $\wt{\swb}_i$ in steady-state. From Lemmas \ref{lemma:approxerrorrecursion} and \ref{lemma:lowrankapprox}, it is clear that the mean-square dynamics of the original error recursion \eqref{eqn:networkerrorrecursiondef} can be well approximated by the low-dimensional model \eqref{eqn:lowdimensionerrorrecursion}. And it was shown in Eq. (10.78) of \cite[p. 563]{Sayed14NOW} that the variances of the centroids $\{\wt{\bm{w}}_{k,i}^\lowdim\}$ are in the order of $\mumax$ in steady-state, which implies that
\be
\label{eqn:wilowdimnormalizedvardef}
\lim_{\mumax \rightarrow 0} \limsup_{i\rightarrow\infty} \mumax^{-1} \E \| \wt{\swb}_i^\lowdim \|^2 = O(\mumax^0)
\ee
Since the induced-2 norm of the covariance matrix of any random vector is always bounded by its variance, i.e., $\| \E \bm{x} \bm{x}^\T \| \le \E \| \bm{x} \|^2$ by using Jensen's inequality, it follows from \eqref{eqn:wilowdimnormalizedvardef} that the normalized covariance matrix of $\wt{\swb}_i^\lowdim$ is finite in steady-state. Moreover, since Lemma \ref{lemma:lowrankerrorcov} applies to any positive value of $\mumax$ as long as it is small enough to ensure stability, we can take the limit of $\mumax$ in \eqref{eqn:ThetaiandTheta} by letting it approach zero asymptotically. That is,
\be
\label{eqn:PiinftyequalPhi}
\lim_{\mumax \rightarrow 0} \limsup_{i\rightarrow \infty} \| \mumax^{-1} \E [\wt{\swb}_i^\lowdim  (\wt{\swb}_i^\lowdim)^\T ] - \Phi \| = 0
\ee
where
\be
\label{eqn:PhiandTheta}
\Phi \defeq \lim_{\mumax \rightarrow 0} ( \mumax^{-1} \Phi_i )
\ee
Due to \eqref{eqn:wilowdimnormalizedvardef} and \eqref{eqn:PiinftyequalPhi}, $\Phi$ is in the order of $\mumax^0$, i.e., $\| \Phi \| = O(\mumax^0)$. In fact, by introducing $\Phi_i \defeq \mumax^{-1} \E [\wt{\swb}_i^\lowdim  (\wt{\swb}_i^\lowdim)^\T ]$ and using the triangle inequality, we have
\begin{align}
\label{eqn:Phiorder1}
\| \Phi \| & = \| \Phi - \Phi_i + \Phi_i \| \le \| \Phi - \Phi_i \| + \| \Phi_i \| \\
\label{eqn:Phiorder2}
\| \Phi_i \| & =  \| \Phi_i - \Phi + \Phi \| \le \| \Phi_i - \Phi \| + \| \Phi \|
\end{align}
Taking $i \rightarrow \infty$ and $\mumax \rightarrow 0$ for both \eqref{eqn:Phiorder1} and \eqref{eqn:Phiorder2} yields:
\begin{align}
\label{eqn:Phiorder1_2}
\| \Phi \| & \le \lim_{\mumax \rightarrow 0} \limsup_{i\rightarrow \infty} \| \Phi_i \| \\
\label{eqn:Phiorder2_2}
\| \Phi \| & \ge \lim_{\mumax \rightarrow 0} \limsup_{i\rightarrow \infty} \| \Phi_i \|
\end{align}
by using \eqref{eqn:PiinftyequalPhi}. From \eqref{eqn:Phiorder1_2} and \eqref{eqn:Phiorder2_2}, we get 
\be
\label{eqn:Phibound}
\| \Phi \| = \lim_{\mumax \rightarrow 0} \limsup_{i\rightarrow \infty} \| \Phi_i \|
\ee
Since $\Phi_i \in \mbbR^{GM \times GM}$ is positive semi-definite, it holds that
\be
\label{eqn:Phiibounds}
(GM)^{-1} \Tr(\Phi_i) \le \| \Phi_i \| \le \Tr(\Phi_i)
\ee
where we used the fact for any positive semi-definite matrix $X \ge 0$ that (i) all the eigenvalues of $X$ are nonnegative, (ii) $\| X \|$  is equal to the largest eigenvalue of $X$, and (iii) $\Tr(X)$ is equal to the sum of all the eigenvalues of $X$. Moreover,
\be
\label{eqn:Phiibounds2}
\Tr(\Phi_i) = \Tr( \mumax^{-1} \E [\wt{\swb}_i^\lowdim  (\wt{\swb}_i^\lowdim)^\T ] ) \!= \mumax^{-1} \E \| \wt{\swb}_i^\lowdim \|^2 \!\!
\ee
Using \eqref{eqn:wilowdimnormalizedvardef}, it follows from \eqref{eqn:Phiibounds} and \eqref{eqn:Phiibounds2} that
\be
\label{eqn:Phiibounds3}
\lim_{\mumax \rightarrow 0} \limsup_{i\rightarrow \infty} \| \Phi_i \| = O(\mumax^0)
\ee
Substituting \eqref{eqn:Phiibounds3} into \eqref{eqn:Phibound} yields the desired result, namely, $\| \Phi \| = O(\mumax^0)$. Then, according to \eqref{eqn:PhiandTheta}, $\Phi$ is the unique solution to equation \eqref{eqn:ThetaDTLEdef} when $\mumax \rightarrow 0$ asymptotically. Introduce two $GM \times GM$ matrices:
\begin{align}
\label{eqn:barbigHdef}
\bar{\Hcal} & \defeq \diag\{ \bar{H}_1, \dots, \bar{H}_G \} = O(\mumax^0) \\
\label{eqn:Rsdef}
\bar{\Rcal} & \defeq \mumax^{-2} \Pcal^\T \Mcal \Rcal_s \Mcal \Pcal = O(\mumax^0)
\end{align}
where $\bar{H}_m$ is from \eqref{eqn:barHmdef} and $\Rcal_s$ is from \eqref{eqn:convergentcovariance}. It is easy to verify that $\bar{\Hcal}$ and $\bar{\Rcal}$ are symmetric and positive-definite according to Assumptions \ref{ass:costfunctions} and \ref{ass:gradienterrors}. From \eqref{eqn:bigDdef}, \eqref{eqn:barbigHdef}, and \eqref{eqn:Dqdef}, we get
\be
\label{eqn:bigDandbigbarH}
\Dcal = I_{GM} - \mumax \bar{\Hcal}
\ee
Using \eqref{eqn:PhiandTheta}--\eqref{eqn:bigDandbigbarH}, equation \eqref{eqn:ThetaDTLEdef} reduces to
\be
\label{eqn:ThetaDTLEdef2}
\bar{\Hcal} \Phi + \Phi \bar{\Hcal} = \bar{\Rcal} + \mumax \bar{\Hcal} \Phi \bar{\Hcal}
\ee
Since $\bar{\Hcal}$ and $\bar{\Rcal}$ are constant matrices, and $\Phi$ is finite, the last term on the RHS of \eqref{eqn:ThetaDTLEdef2} disappears as $\mumax \rightarrow 0$ asymptotically. Therefore, we conclude that $\Phi$ is the unique solution to the continuous Lyapunov equation:
\be
\label{eqn:continuoustimeLyapunovEqn}
\bar{\Hcal} \Phi + \Phi \bar{\Hcal} = \bar{\Rcal}
\ee
Let us define the \emph{normalized} network error covariance matrix for $\wt{\swb}_i$ from \eqref{eqn:networkerrorrecursiondef} by
\be
\label{eqn:bigPidef}
\Pi_i \defeq \mumax^{-1} \E ( \wt{\swb}_i \wt{\swb}_i^\T )
\ee

\begin{theorem}[Block structure]
\label{theorem:blockstructure}
In steady-state, and as the step-sizes approach zero asymptotically, the normalized network error covariance matrix $\Pi_i$ in \eqref{eqn:bigPidef} satisfies
\be
\label{eqn:Piinftyoblockstructure}
\lim_{\mumax \rightarrow 0} \limsup_{i \rightarrow \infty} \| \Pi_i - \Pi \| = 0
\ee
where
\be
\label{eqn:Pinetworkdef}
\Pi \defeq \begin{bmatrix}
( \one_{N_1^g} \one_{N_1^g}^\T ) \kron \Phi_{1,1} \!&\! \dots \!&\! ( \one_{N_1^g} \one_{N_G^g}^\T ) \kron \Phi_{1,G} \\
\vdots \!&\! \ddots \!&\! \vdots \\
( \one_{N_G^g} \one_{N_1^g}^\T ) \kron \Phi_{G,1} \!&\! \dots \!&\! ( \one_{N_G^g} \one_{N_G^g}^\T ) \kron \Phi_{G,G} \\
\end{bmatrix}
\ee
and $\Phi_{m, r}$ denotes the $(m, r)$-th block of $\Phi$ from \eqref{eqn:continuoustimeLyapunovEqn} with block size $M \times M$.
\end{theorem}

\begin{proof}
See Appendix \ref{app:block}.
\end{proof}

\section{Error Probability Analysis for Clustering}
Using the results from the previous section, we now move on to assess the error probabilities for the hypothesis testing problem \eqref{eqn:decisionrule}. To do so, we need to determine the probability distribution of the decision statistic that is generated by recursion \eqref{eqn:distributedadaptgroup}--\eqref{eqn:distributedcombinegroup}.

\subsection{Asymptotic Joint Distribution of Estimation Errors}
\label{subsection:normalpdf}
Using \eqref{eqn:bigDandbigbarH}, we rewrite the low-dimensional model \eqref{eqn:lowdimensionerrorrecursion} as
\be
\label{eqn:lowdimensionerrorrecursionnew}
\wt{\swb}_i^\lowdim = \wt{\swb}_{i-1}^\lowdim - \mumax \bar{\Hcal} \wt{\swb}_{i-1}^\lowdim + \mumax \bar{\bm{s}}_i
\ee
where $\bar{\Hcal}$ is from \eqref{eqn:barbigHdef} and 
\be
\label{eqn:barsidef}
\bar{\bm{s}}_i \defeq \mumax^{-1} \Pcal^\T \Mcal \ssb_i( \swb_{i-1} ) \in \mbbR^{GM \times 1}
\ee

\begin{lemma}[Rate of weak convergence] 
\label{lemma:normallowdimensional}
The normalized sequence, $\{\wt{\swb}_i^\lowdim /\sqrt{\mumax}; i \ge 0\}$, from \eqref{eqn:lowdimensionerrorrecursionnew} converges in distribution as $i \rightarrow \infty$ and $\mumax \rightarrow 0$ to the Gaussian random variable:
\be
\label{eqn:xidef}
\bm{\xi} \defeq \col\{ \bm{\xi}_1, \dots, \bm{\xi}_G \} \sim \mbbN(0, \Phi)
\ee
where $\bm{\xi}_m \in \mbbR^{M \times 1}$ for all $m$, and $\Phi \in \mbbR^{GM \times GM}$ is the unique solution to the Lyapunov equation \eqref{eqn:continuoustimeLyapunovEqn}.
\end{lemma}

\begin{proof}
See Appendix \ref{app:normal}.
\end{proof}

In the sequel we establish the main result that the distribution of the normalized error sequence from \eqref{eqn:networkerrorrecursiondef}, $\{\wt{\swb}_i/\sqrt{\mumax}; i \ge 0\}$, asymptotically approaches a Gaussian distribution. According to Definition 4 from \cite[p. 253]{Shiryaev80}, a random sequence $\{\bm{\zeta}_i; i\ge 0\}$ converges in distribution to some random variable $\bm{\zeta}$ if, and only if, 
\be
\lim_{i \rightarrow \infty} \E \left| f(\bm{\zeta}_i) - f(\bm{\zeta}) \right| = 0
\ee
for \emph{any} bounded continuous function $f(\cdot)$. We use this fact together with the following lemma to establish Theorem \ref{theorem:normaldistribution} further ahead. 

\begin{lemma}[Weak convergence]
\label{lemma:convergence}
Let $\{\bm{\zeta}_i; i\ge 0 \}$ and $\{\bm{\eta}_i; i \ge 0\}$ be two random sequences that are dependent on the parameter $\mumax$. If $\{\bm{\zeta}_i; i\ge 0\}$ approaches $\{\bm{\eta}_i; i\ge 0\}$ in mean-square sense:
\be
\label{eqn:convergencemoments}
\lim_{\mumax \rightarrow 0} \limsup_{i \rightarrow \infty} \E \| \bm{\zeta}_i - \bm{\eta}_i \|^2 = 0
\ee
and the variances of $\{\bm{\zeta}_i\}$ converge in the following sense:
\be
\label{eqn:convergencesigma}
\lim_{\mumax \rightarrow 0} \limsup_{i \rightarrow \infty} \E \| \bm{\zeta}_i \|^2 = \sigma^2
\ee
then it holds for any bounded continuous function $f(\cdot)$ that
\be
\label{eqn:convergenceweakly}
\lim_{\mumax \rightarrow 0} \limsup_{i \rightarrow \infty} \E | f(\bm{\zeta}_i) - f(\bm{\eta}_i) | = 0 
\ee
\end{lemma}

\begin{IEEEproof}
See Appendix \ref{app:convergence}.
\end{IEEEproof}

\begin{theorem}[Asymptotic normality]
\label{theorem:normaldistribution}
As $i \rightarrow \infty$ and $\mumax \rightarrow 0$, the normalized error sequence from \eqref{eqn:networkerrorrecursiondef}, $\{\wt{\swb}_i/\sqrt{\mumax}; i \ge 0\}$, converges in distribution \emph{close} to the Gaussian random variable:
\be
\label{eqn:zetadef}
\bm{\zeta} \defeq \col\{ \one_{N_1^g} \kron \bm{\xi}_1, \dots, \one_{N_G^g} \kron \bm{\xi}_G \} \sim \mbbN(0, \Pi)
\ee
in the following sense:
\be
\label{eqn:weakconvergence}
\lim_{\mumax \rightarrow 0} \limsup_{i \rightarrow \infty} \E \left| f\left(\frac{\wt{\swb}_i}{\sqrt{\mumax}}\right) - f(\bm{\zeta}) \right| = 0
\ee
for any bounded continuous function $f(\cdot):\mbbR^{NM \times 1} \mapsto \mbbR$, where $\{ \bm{\xi}_m\}$ are from \eqref{eqn:xidef}, and $\Pi$ is from \eqref{eqn:Pinetworkdef}. 
\end{theorem}

\begin{proof}
Using the triangle inequality, we have
\begin{align}
\label{eqn:weakconvergence2}
\E \left| f\left( \frac{ \wt{\swb}_i }{ \sqrt{\mumax} } \right) - f(\bm{\zeta}) \right|
& \le \E \left| f\left( \frac{ \wt{\swb}_i }{ \sqrt{\mumax} } \right) - f\left( \frac{ \wt{\swb}_i^\longterm }{ \sqrt{\mumax} } \right) \right| + \E \left| f\left( \frac{ \wt{\swb}_i^\longterm }{ \sqrt{\mumax} } \right) - f\left( \frac{ \bar{\swb}_i^\lowdim }{ \sqrt{\mumax} } \right) \right| \nn \\
& \qquad + \E \left| f\left( \frac{ \bar{\swb}_i^\lowdim }{ \sqrt{\mumax} } \right) - f( \bm{\zeta} ) \right|
\end{align}
where $\wt{\swb}_i^\longterm$ is from the long-term model \eqref{eqn:longtermerrorrecursion}, and $\bar{\swb}_i^\lowdim$ is from \eqref{eqn:zidef} and is related to the low-dimensional model \eqref{eqn:lowdimensionerrorrecursion}. By Lemma \ref{lemma:msdblock}, the variances of the sequence $\{ \wt{\swb}_i / \sqrt{\mumax}; i \ge 0 \}$ converge to its normalized MSD in \eqref{eqn:steadystateMSDdefblock} in a sense similar to \eqref{eqn:convergencesigma}. Using Lemma \ref{lemma:approxerrorrecursion}, it is clear that $\{ \wt{\swb}_i / \sqrt{\mumax}; i \ge 0 \}$ approaches $\{ \wt{\swb}_i^\longterm / \sqrt{\mumax}; i \ge 0 \}$ in a sense similar to \eqref{eqn:convergencemoments}. Therefore, by calling upon Lemma \ref{lemma:convergence}, we conclude that the limit superior of the first term on the RHS of \eqref{eqn:weakconvergence2} vanishes. Likewise, using Lemmas \ref{lemma:approxerrorrecursion} and \ref{lemma:msdblock}, it can be verified that the variances of the sequence $\{ \wt{\swb}_i^\longterm / \sqrt{\mumax}; i \ge 0 \}$ also converge to the same normalized MSD in \eqref{eqn:steadystateMSDdefblock}. Therefore, from Lemmas \ref{lemma:lowrankapprox} and \ref{lemma:convergence}, the limit superior of the second term on the RHS of \eqref{eqn:weakconvergence2} vanishes. The limit superior of the third term vanishes since $\{ \bar{\swb}_i^\lowdim / \sqrt{\mumax}; i \ge 0 \}$ converges in distribution to $\bm{\zeta}$, which follows from Lemma \ref{lemma:normallowdimensional}. Therefore, the limit superior of the RHS of \eqref{eqn:weakconvergence2} vanishes when $i \rightarrow \infty$ and $ \mumax \rightarrow 0$.
\end{proof}

Theorem \ref{theorem:normaldistribution} allows us to approximate the distribution of $\wt{\swb}_i/\sqrt{\mumax}$ by the Gaussian distribution $\mbbN(0, \Pi)$ for large enough $i$ and small enough $\mumax$.

\subsection{Statistical Decision on Clustering}
In Theorem \ref{theorem:normaldistribution}, we established that for large enough $i$ and for sufficiently small $\mumax$, the joint distribution of the individual estimators $\{ \bm{w}_{k,i}; k = 1, 2, \dots, N \}$ can be well approximated by a Gaussian distribution \eqref{eqn:zetadef}. Therefore, the marginal distribution for any pair of estimators, say, $\bm{w}_{k,i}$ and $\bm{w}_{\ell,i}$, can be well approximated by the Gaussian distribution:
\be
\label{eqn:steadystatedistributiontwoagents}
\begin{bmatrix}
\bm{w}_{k,i} \\
\bm{w}_{\ell,i} \\
\end{bmatrix} \sim \mbbN\left( \begin{bmatrix}
w_k^o \\
w_\ell^o \\
\end{bmatrix}, \mumax \begin{bmatrix}
\Pi_{k,k} & \Pi_{k,\ell} \\
\Pi_{\ell,k} & \Pi_{\ell,\ell} \\
\end{bmatrix} \right)
\ee
where $w_k^o$ and $w_\ell^o$ are their individual minimizers, and $\Pi_{k,\ell}$ denotes the $(k,\ell)$-th block of $\Pi$ with block size $M \times M$. Without loss of generality, let us consider the scenario where agent $k$ is from group $\Gcal_m$ in cluster $\Ccal_q$ and agent $\ell$ is from group $\Gcal_n$ in cluster $\Ccal_r$, i.e., $k \in \Gcal_m \subseteq \Ccal_q$ and $\ell \in \Gcal_n \subseteq \Ccal_r$. Then, we have from Definition \ref{def:cluster} that
\be
\label{eqn:meankellandqr}
w_k^o = w_q^\star, \qquad w_\ell^o = w_r^\star
\ee
From Theorem \ref{theorem:blockstructure}, the covarince matrix $\Pi$ possesses the block structure shown in \eqref{eqn:Pinetworkdef}. Using \eqref{eqn:Pinetworkdef}, and noticing that $k \in \Gcal_m$ and $\ell \in \Gcal_n$, it is obvious that 
\be
\label{eqn:varkellandqr}
\Pi_{k,k} = \Phi_{m,m}, \; \Pi_{k,\ell} = \Phi_{m,n}, \;
\Pi_{\ell,k} = \Phi_{n,m}, \; \Pi_{\ell,\ell} = \Phi_{n,n} 
\ee
Then, it follows from \eqref{eqn:steadystatedistributiontwoagents}--\eqref{eqn:varkellandqr} that
\be
\label{eqn:steadystatedistributiontwoagentsGroup}
\begin{bmatrix}
\bm{w}_{k,i} \\
\bm{w}_{\ell,i} \\
\end{bmatrix} \sim \mbbN\left( \begin{bmatrix}
w_q^\star \\
w_r^\star \\
\end{bmatrix}, \mumax \begin{bmatrix}
\Phi_{m,m} & \Phi_{m,n} \\
\Phi_{n,m} & \Phi_{n,n} \\
\end{bmatrix} \right)
\ee
which means that the mean and covariance of the joint distribution for any pair of agents $k$ and $\ell$ only depends on their groups. In other words, for any two agents $k_1$ and $k_2$ from the same group $\Gcal_m$, the joint distribution of $\{k_1, \ell\}$ and the joint distribution of $\{k_2, \ell\}$ will be well approximated by the same Gaussian distribution in \eqref{eqn:steadystatedistributiontwoagentsGroup}. Therefore, if both agents $k_1$ and $k_2$ need to decide whether agent $\ell$ is in the same cluster as they are, then they will have the same error probabilities in the hypothesis test \eqref{eqn:decisionrule}.

Based on \eqref{eqn:steadystatedistributiontwoagentsGroup}, the hypothesis test problem for clustering now becomes that of determining whether or not the two (near) Gaussian random vectors $\bm{w}_{k,i}$ and $\bm{w}_{\ell,i}$ have the same mean. Suppose the samples from the two variables are paired. The difference 
\be
\label{eqn:dkelldef}
\bm{d}_{k,\ell} \defeq \bm{w}_{k,i} - \bm{w}_{\ell,i}
\ee
serves as a sufficient statistics \cite{Anderson58}. Since $\bm{w}_{k,i}$ and $\bm{w}_{\ell,i}$ are jointly Gaussian in \eqref{eqn:steadystatedistributiontwoagentsGroup}, their difference $\bm{d}_{k,\ell}$ is also Gaussian: 
\be
\label{eqn:dkldistribution}
\bm{d}_{k,\ell} \sim \mbbN( d_{q,r}^\star, \mumax \Delta_{m,n})
\ee
where
\begin{align}
\label{eqn:meandkldef}
d_{q,r}^\star & \defeq w_q^\star - w_r^\star \\
\label{eqn:vardkldef}
\Delta_{m,n} & \defeq \Phi_{m,m} + \Phi_{n,n} - \Phi_{m,n} - \Phi_{n,m} \ge 0
\end{align}
If the agents $k$ and $\ell$ are from the same cluster such that $q = r$, then hypothesis $\mbbH_0$ in \eqref{eqn:decisionrule} is true and $d_{q,r}^\star = 0$; otherwise, hypothesis $\mbbH_1$ in \eqref{eqn:decisionrule} is true and $d_{q,r}^\star \ne 0$. The hypothesis test for clustering becomes to test whether or not the difference $\bm{d}_{k,\ell}$ in \eqref{eqn:dkelldef} is zero mean \emph{without} knowing its covariance matrix $\mumax \Delta_{m,n}$. If $N_\sam$ independent samples of $\bm{d}_{k,\ell}$ are available for testing, where $N_\sam > M$, and $\Delta_{m,n}$ is non-singular, then according to the Neyman-Pearson criterion \cite{Poor98}, the likelihood ratio test is given by \cite[p. 164]{Anderson58}
\be
\label{eqn:Tsquaretest}
\bm{T}_{k,\ell}^2 \defeq N_\sam \bar{\bm{x}}^\T \bm{S}^{-1} \bar{\bm{x}} \overset{\mbbH_0}{\underset{\mbbH_1}{\lessgtr}} \theta_{k,\ell}
\ee
where $\bm{T}_{k,\ell}^2$ is called Hotelling's T-square statistic, $\bar{\bm{x}}$ is the sample mean of $\bm{d}_{k,\ell}$, $\bm{S}$ is the unbiased sample covariance matrix, and $\theta_{k,\ell}$ is the predefined threshold from \eqref{eqn:decisionrule}. The scaled T-square statistics $\frac{ N_\sam - M }{ (N_\sam - 1) M } \cdot \bm{T}_{k,\ell}^2$ has a non-central F-distribution with $M$ and $N_\sam - M$ degrees of freedom and non-centrality parameter $N_\sam \mumax^{-1} (d_{q,r}^\star)^\T \Delta_{m,n}^{-1} d_{q,r}^\star$ \cite[p. 480]{Johnson95v2}. When $d_{q,r}^\star = 0$, it reduces to a central F-distribution \cite[p. 322]{Johnson95v2}.

However, because stochastic iterative algorithms employ very small step-sizes, sampling their steady-state estimators over time does not produce independent samples. In many scenarios we only have one sample available for testing, where the sample mean reduces to the sample itself, and the sample covariance matrix is not even available. In order to carry out the hypothesis test, we replace the sample covariance matrix by the identity matrix. Then, the Hotelling's T-square test \eqref{eqn:Tsquaretest} becomes
\be
\label{eqn:deltakelldef}
\bm{\delta}_{k,\ell}^2 \defeq \| \bm{d}_{k,\ell} \|^2 \overset{\mbbH_0}{\underset{\mbbH_1}{\lessgtr}} \theta_{k,\ell}
\ee
where we re-used $\bm{d}_{k,\ell}$ to denote the only available sample for testing. The decision statistic $\bm{\delta}_{k,\ell}^2$ is a quadratic form of the (near) Gaussian random vector $\bm{d}_{k,\ell}$. Using \eqref{eqn:dkldistribution}, the mean of $\bm{\delta}_{k,\ell}^2$ is given by
\be
\label{eqn:chikl2testmean}
\E \bm{\delta}_{k,\ell}^2 = \E \| \bm{d}_{k,\ell} \|^2 = \E \Tr( \bm{d}_{k,\ell} \bm{d}_{k,\ell}^\T ) = \Tr( \E \bm{d}_{k,\ell} \bm{d}_{k,\ell}^\T ) = \| d_{q,r}^\star \|^2 + \mumax \Tr(\Delta_{m,n}) 
\ee
and the variance of $\bm{\delta}_{k,\ell}^2$ is given by (see Appendix \ref{app:moments})
\be
\label{eqn:chikl2testvar}
\Var(\bm{\delta}_{k,\ell}^2) = \E \| \bm{d}_{k,\ell} \|^4 - ( \E \| \bm{d}_{k,\ell} \|^2 )^2 = 4 \mumax \| d_{q,r}^\star \|_{\Delta_{m,n}}^2  + 2 \mumax^2 \Tr(\Delta_{m,n}^2)
\ee
It is seen that the mean of $\bm{\delta}_{k,\ell}^2$ is dominated by $\| d_{q,r}^\star \|^2$ for sufficiently small step sizes. Since the variance of $\bm{\delta}_{k,\ell}^2$ is in the order of $\mumax$, according to Chebyshev's inequality \cite[p. 47]{Shiryaev80}, we have
\be
\Pr[ | \bm{\delta}_{k,\ell}^2 - \E \bm{\delta}_{k,\ell}^2 | \ge c ] \le \frac{\Var(\bm{\delta}_{k,\ell}^2)}{c} = O(\mumax)
\ee
for any constant $c > 0$. Therefore, for sufficiently small step sizes, the probability mass of $\bm{\delta}_{k,\ell}^2$ will highly concentrate around $\E \bm{\delta}_{k,\ell}^2$. When hypothesis $\mbbH_0$ is true, we have $d_{q,r}^\star = 0$ and $\E \bm{\delta}_{k,\ell}^2 = \mumax \Tr(\Delta_{m,n}) = O(\mumax) \approx 0$; when hypothesis $\mbbH_1$ is true, we have $d_{q,r}^\star \neq 0$ and $\E \bm{\delta}_{k,\ell}^2 = \| d_{q,r}^\star \|^2 + O(\mumax) \approx \| d_{q,r}^\star \|^2$. That is, the probability mass of $\bm{\delta}_{k,\ell}^2$ under $\mbbH_0$ concentrates near $0$ while the probability mass of $\bm{\delta}_{k,\ell}^2$ under $\mbbH_1$ concentrates near $\| d_{q,r}^\star \|^2  = \| w_q^\star - w_r^\star \|^2 > 0$ (which is a constant that is independent of $\mumax$). Obviously, the threshold $\theta_{k,\ell}$ should be chosen between 0 and $\| d_{q,r}^\star \|^2$. By doing so, the Type-I error will correspond to the right tail probability of $\bm{\delta}_{k,\ell}^2$ when $d_{q,r}^\star = 0$ (see \eqref{eqn:typeIerrordef} further ahead) and the Type-II error will correspond to the left tail probability of $\bm{\delta}_{k,\ell}^2$ when $d_{q,r}^\star \neq 0$ (see \eqref{eqn:typeIIerrordef} further ahead).

In order to examine the statistical properties of $\bm{\delta}_{k,\ell}^2$ and to perform the analysis for error probabilities, let us introduce the eigen-decomposition of $\Delta_{m,n}$ in \eqref{eqn:vardkldef} and denote it by
\be
\label{eqn:Deltaeigen}
\Delta_{m,n} = U_\Delta \Lambda_\Delta U_\Delta^\T 
\ee
where $U_\Delta$ is orthonormal and $\Lambda_\Delta$ is diagonal and nonnegative. Let further 
\be
\label{eqn:xdef}
\bm{x} \defeq \Lambda_\Delta^{- 1/2} U_\Delta^\T \bm{d}_{k,\ell}, \quad 
\bar{x} \defeq \Lambda_\Delta^{- 1/2} U_\Delta^\T d_{q,r}^\star
\ee
Since $\bm{d}_{k,\ell} \sim \mbbN( d_{q,r}^\star, \mumax \Delta_{m,n})$, it follows from \eqref{eqn:Deltaeigen} and \eqref{eqn:xdef} that $\bm{x} \sim \mbbN(\bar{x}, \mumax I_M)$. Substituting \eqref{eqn:Deltaeigen} and \eqref{eqn:xdef} into \eqref{eqn:deltakelldef} yields
\be
\label{eqn:delta2andx}
\bm{\delta}_{k,\ell}^2 = \bm{x}^\T \Lambda_\Delta \bm{x} = \sum_{h=1}^{M} \lambda_h \bm{x}_h^2
\ee
where $\bm{x}_h$ denotes the $h$-th elements of $\bm{x}$, and $\lambda_h$ denotes the $h$-th element on the diagonal of $\Lambda_\Delta$. From \eqref{eqn:delta2andx}, it is obvious that $\bm{\delta}_{k,\ell}^2$ is a weighted sum of independent squared Gaussian random variables. When hypothesis $\mbbH_0$ is true, we have $d_{q,r}^\star = 0$ and $\bar{x} = 0$ by \eqref{eqn:xdef}. In this case, $\bm{\delta}_{k,\ell}^2$ reduces to a weighted sum of independent Gamma random variables (because squared zero-mean Gaussian random variables follow Gamma distributions \cite[p. 337]{Johnson95v1}), whose pdf is available in closed-form (but is very complicated) \cite{Moschopoulos85, Kara06TCOM}. When hypothesis $\mbbH_1$ is true and $\| d_{q,r}^\star \|^2 > 0$, the pdf of $\bm{\delta}_{k,\ell}^2$ is generally not available in closed-form. Several procedures have been proposed in \cite{Imhof61biometrika, Sheil77JRSS, Liu09CSDA, Duchesne10CSDA, Ha13REVSTAT} for numerical evaluation of its tail probability.  Instead of relying on the precise pdf of $\bm{\delta}_{k,\ell}^2$, we shall provide some useful constructions in the sequel for the error probabilities in the hypothesis test problem \eqref{eqn:deltakelldef}.

\subsection{Error Probabilities}
For any $k \in \Gcal_m \subseteq \Ccal_q$ and $\ell \in \Gcal_n \subseteq \Ccal_r$, the Type-I error, namely, the false alarm for incorrect rejection of a true $\mbbH_0$, is given by 
\be
\label{eqn:typeIerrordef}
\mbox{Type-I error}: \qquad \Pr[\bm{\delta}_{k,\ell}^2 > \theta_{k,\ell} | d_{q,r}^\star = 0]
\ee
and the Type-II error, namely, the missing detection for incorrect rejection of a true $\mbbH_1$, is given by 
\be
\label{eqn:typeIIerrordef}
\mbox{Type-II error}: \qquad \Pr[\bm{\delta}_{k,\ell}^2 < \theta_{k,\ell} | d_{q,r}^\star \neq 0]
\ee
It is seen that the Type-I error corresponds to the right tail probability of $\bm{\delta}_{k,\ell}^2$ with $d_{q,r}^\star = 0$ and the Type-II error corresponds to the left tail probability of $\bm{\delta}_{k,\ell}^2$ with $d_{q,r}^\star \neq 0$. This is a fundamental difference between the two types of errors and, therefore, different techniques are needed to approximate them. Specifically, for the Type-II error, the pdf of $\bm{\delta}_{k,\ell}^2$ is close to a bell shape and can be well approximated by a Gaussian pdf. Then, the Type-II error probability can be bounded by using Chernoff bound \cite{Cover06}. However, this technique does not apply to the Type-I error because when $d_{q,r}^\star = 0$, the pdf of $\bm{\delta}_{k,\ell}^2$ concentrates on the positive side of the origin point and is  skewed with a long right tail. Consequently, we need to take a different approach to bound the Type-I error probability.

\subsubsection{Type-I Error}
We first note that
\be
\label{eqn:deltaandx}
\bm{\delta}_{k,\ell}^2 = \bm{x}^\T \Lambda_\Delta \bm{x} \le \|\Delta_{m,n}\| \cdot \| \bm{x} \|^2 
\ee
where $\Lambda_\Delta$ is from \eqref{eqn:Deltaeigen}. This means that if $\bm{\delta}_{k,\ell}^2 > \theta_{k,\ell}$, then $\|\Delta_{m,n}\| \cdot \| \bm{x} \|^2 > \theta_{k,\ell}$ must be true, which further implies that the event $\{ \bm{\delta}_{k,\ell}^2 > \theta_{k,\ell} \}$ is a subset of the event $\{ \|\Delta_{m,n}\| \cdot \| \bm{x} \|^2 > \theta_{k,\ell} \}$. Therefore,
\be
\label{eqn:boundType1error}
\Pr[\bm{\delta}_{k,\ell}^2 > \theta_{k,\ell} | d_{q,r}^\star = 0] \le \Pr [ \| \bm{x} \|^2 > \theta_{k,\ell}' | \bar{x} = 0 ] 
\ee
where $\bar{x}$ is from \eqref{eqn:xdef}, and 
\be
\theta_{k,\ell}' \defeq \frac{\theta_{k,\ell}}{ \| \Delta_{m,n} \|}
\ee
Since $\bar{x} = 0$, $\mumax^{-1} \| \bm{x} \|^2$ follows a central chi-square distribution with $M$ degrees of freedom \cite[p. 415]{Johnson95v1}. Therefore, using the Chernoff bound for the central chi-square distribution \cite[Lemma 1, p. 2500]{Li07JMLR}, we get from \eqref{eqn:boundType1error} that 
\be
\Pr[\bm{\delta}_{k,\ell}^2 > \theta_{k,\ell} | d_{q,r}^\star = 0] \le 1 - \Pr [ \| \bm{x} \|^2 \le \theta_{k,\ell}' | \bar{x} = 0 ] \le \left(\frac{ \theta_{k,\ell}' e }{\mumax M }\right)^{ M/2 } \exp\left(- \frac{\theta_{k,\ell}'}{ 2\mumax} \right)
\ee
for $\mumax < \theta_{k,\ell}'/M$, where $e$ is Euler's number. Therefore, when $\mumax$ is small enough, the Type-I error probability decays exponentially at a rate of $O(e^{-c_1/\mumax})$ for some constant $c_1 > 0$. 

\subsubsection{Type-II Error}
We consider the characteristic function of $\bm{\delta}_{k,\ell}^2$. Since $\{\bm{x}_h \}$ are mutually-independent, the characteristic function of $\bm{\delta}_{k,\ell}^2$ is given by
\be
\label{eqn:characteristiczeta2}
c_{\bm{\delta}_{k,\ell}^2}(t) \defeq \E \left[e^{j t \bm{\delta}_{k,\ell}^2}\right] \!=\! \E \left[e^{j t \sum_{h=1}^{M} \lambda_h \bm{x}_h^2 }\right] \!=\! \prod_{h=1}^{M} \E \left[ e^{j t \lambda_h \bm{x}_h^2 } \right]
\ee
where we used \eqref{eqn:deltaandx}. Since $d_{q,r}^\star \neq 0$ in this case, $\bm{x}$ from \eqref{eqn:xdef} has nonzero mean $\bar{x} \neq 0$. Therefore, each $\mumax^{-1} \bm{x}_h^2$ is a non-central chi-square random variable with one degree of freedom and non-centrality $\mumax^{-1} \bar{x}_h^2$ \cite[p. 433]{Johnson95v2}. The characteristic function of $\bm{x}_h^2$ is then given by \cite[p. 437]{Johnson95v2}:
\be
\label{eqn:characteristicncx2}
\E \left[ e^{jt \bm{x}_h^2 } \right] = \frac{1}{\sqrt{1 - 2 j t \mumax}} e^{j \bar{x}_h^2 t / ( 1 - 2 j t \mumax)}
\ee
Substituting \eqref{eqn:characteristicncx2} into \eqref{eqn:characteristiczeta2} yields:
\be
\label{eqn:characteristiczeta3}
c_{\bm{\delta}_{k,\ell}^2}(t) = \prod_{h=1}^{M} \frac{1}{\sqrt{1 - 2 j t \mumax \lambda_h}} \cdot e^{j \bar{x}_h^2 t \lambda_h / (1 - 2 j t \mumax \lambda_h)}
\ee
When $\mumax$ is sufficiently small, we have
\be
\label{eqn:approxfactors}
\frac{1}{\sqrt{1 - 2 j t \mumax \lambda_h}} \approx 1,  \frac{1}{1 - 2 j t \mumax \lambda_h} \approx 1 + 2 j t \mumax \lambda_h
\ee
Using \eqref{eqn:dmnoandbarx}, we can approximate $c_{\bm{\delta}_{k,\ell}^2}(t)$ in \eqref{eqn:characteristiczeta3} by 
\begin{align}
\label{eqn:characteristiczeta4}
c_{\bm{\delta}_{k,\ell}^2}(t) & \approx \prod_{h=1}^{M} e^{j \bar{x}_h^2 t \lambda_h (1 + 2 j t \mumax \lambda_h)} \nn \\
& = e^{j t (\sum_{h=1}^{M} \lambda_h \bar{x}_h^2) - 2 t^2 \mumax (\sum_{h=1}^{M} \lambda_h^2 \bar{x}_h^2)} \nn \\
& = e^{j t \| d_{q,r}^\star \|^2 - 2 t^2 \mumax \| d_{q,r}^\star \|_{\Lambda_\Delta}^2}
\end{align}
where we used the fact that
\be
\label{eqn:dmnoandbarx}
\sum_{h=1}^{M} \lambda_h \bar{x}_h^2 = \| d_{q,r}^\star \|^2, \quad
\sum_{h=1}^{M} \lambda_h^2 \bar{x}_h^2 = \| d_{q,r}^\star \|_{\Lambda_\Delta}^2
\ee
Note that the RHS of \eqref{eqn:characteristiczeta4} coincides with the characteristic function of a Gaussian distribution with mean $\| d_{q,r}^\star \|^2$ and variance $4 \mumax \| d_{q,r}^\star \|_{\Lambda_\Delta}^2$ \cite[p. 89]{Johnson95v1}. Since the distribution of a random variable is uniquely determined by its characteristic function, result \eqref{eqn:characteristiczeta4} implies that $\bm{\delta}_{k,\ell}^2 \sim \mbbN(\| d_{q,r}^\star \|^2, 4 \mumax \| d_{q,r}^\star \|_{\Lambda_\Delta}^2)$ approximately for sufficiently small $\mumax$. Thus, 
\be
\Pr[\bm{\delta}_{k,\ell}^2 < \theta_{k,\ell} | d_{q,r}^\star \neq 0] \approx Q\left( \frac{ \| d_{q,r}^\star \|^2 - \theta_{k,\ell}}{2 \mumax^{1/2} \| d_{q,r}^\star \|_{\Lambda_\Delta}} \right) 
\le \frac{1}{2} e^{ - ( \| d_{q,r}^\star \|^2 - \theta_{k,\ell} )^2 / 8 \mumax \| d_{q,r}^\star \|_{\Lambda_\Delta}^2 }
\ee
where $Q(\cdot)$ denotes the $Q$-function, which is the tail probability of the standard Gaussian distribution, and the last step is by using the Chernoff bound \cite[p. 380]{Cover06}. Therefore, when $\mumax$ is small enough, the Type-II error decays exponentially at a rate of $O(e^{-c_2/\mumax})$ for some constant $c_2 > 0$.

\subsubsection{A Special Case}
For the purpose of illustration only, we consider a special case where $\Delta_{m,n} = \sigma_{m,n}^2  I_M$. In this case, the pdf of $\bm{\delta}_{k,\ell}^2$ has a closed-form pdf. When $\mbbH_1$ is true and $\| d_{q,r}^\star \|^2 > 0$, the quadratic form $\bm{\delta}_{k,\ell}^2 / (\mumax \sigma_{m,n}^2)$ reduces to a non-central chi-square random variable with $M$ degrees of freedom and non-centrality parameter $\| d_{q,r}^\star \|^2 / \mumax \sigma_{m,n}^2$ \cite[p. 433]{Johnson95v2}. Let us denote the non-central chi-square distribution with $d$ degrees of freedom and non-centrality parameter $\lambda$ by $\chi_d^2(\lambda)$. The pdf of $\chi_d^2(\lambda)$ is then given by \cite[p. 433]{Johnson95v2}:
\be
\label{eqn:fchi2def}
f_{\chi^2}(x; d, \lambda) = \frac{1}{2} \left( \frac{x}{\lambda} \right)^{(d-2)/4}  e^{-(x+\lambda)/2} I_{(d-2)/2} (\sqrt{\lambda x}) \!\!
\ee
for $x \ge 0$, where $I_h(x)$ denotes the $h$-th order modified Bessel function of the first kind. Then,
\be
\frac{\bm{\delta}_{k,\ell}^2}{\mumax \sigma_{m,n}^2} \sim \chi_M^2\left( \frac{\| d_{q,r}^\star \|^2}{\mumax \sigma_{m,n}^2} \right)
\ee
and the pdf of $\bm{\delta}_{k,\ell}^2$ is given by
\be
\label{eqn:fzdef}
f(z) = \frac{1}{\mumax \sigma_{m,n}^2} \cdot f_{\chi^2}\left(\frac{z}{\mumax \sigma_{m,n}^2}; M, \frac{\| d_{q,r}^\star \|^2}{\mumax \sigma_{m,n}^2} \right)
\ee
where $f_{\chi^2}(\cdot)$ is from \eqref{eqn:fchi2def}. When $\mbbH_0$ is true and $\| d_{q,r}^\star \|^2 = 0$, the pdf $f(z)$ in \eqref{eqn:fzdef} reduces to a scaled central chi-square distribution \cite[p. 415]{Johnson95v1}:
\be
\label{eqn:fzcentraldef}
f(z) = \frac{1}{\mumax \sigma_{m,n}^2} \cdot f_{\chi^2}\left(\frac{z}{\mumax \sigma_{m,n}^2}; M, 0 \right)
\ee
We plot the pdf $f(z)$ from \eqref{eqn:fzdef} and \eqref{eqn:fzcentraldef} in Fig. \ref{fig:pdf_chi2}. It can be observed that when $M$, $\| d_{q,r}^\star \|^2$, and $\sigma_{m,n}^2$ are fixed, in both $\mbbH_0$ (blue curves) and $\mbbH_1$ (red curves) cases, the probability mass of $\bm{\delta}_{k,\ell}^2$ concentrates more around its mean as $\mumax$ decreases. When $q \neq r$ (i.e., $\mbbH_1$ is true), the mean of $\bm{\delta}_{k,\ell}$ is close to $\| d_{q,r}^\star \|^2 = 1$ for sufficiently small $\mumax$; when $q = r$ (i.e., $\mbbH_0$ is true), the mean is close to zero. The right tail probabilities of the blue curves (under $\mbbH_0$) and the left tail probabilities of the red curves (under $\mbbH_1$) all decay exponentially. In addition, it is seen that the pdf of $\bm{\delta}_{k,\ell}^2$ under $\mbbH_1$ (the red curves with $\| d_{q,r}^\star \|^2 > 0$) is near symmetric and is in bell-shape, which agrees with the Gaussian approximation we made when evaluating the Type-II error (mis-detection) for the general case. On the other hand, the pdf of $\bm{\delta}_{k,\ell}^2$ under $\mbbH_0$ (the blue curves with $\| d_{q,r}^\star \|^2 = 0$) concentrates close to zero and has large skewness with a long tail on the RHS, which distinguishes itself from Gaussian distributions; this demonstrates our previous statement that it is not appropriate to assess the Type-I error (false alarm) by approximating the pdf of $\bm{\delta}_{k,\ell}^2$ under $\mbbH_0$ with Gaussian distributions.

\begin{figure}
\centering
\includegraphics[width=3.5in]{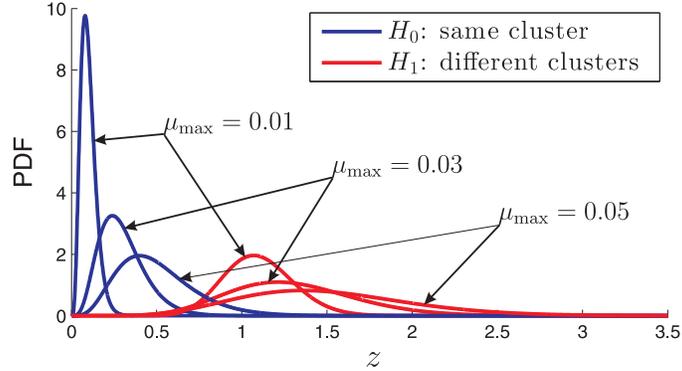}
\caption{The pdf of $\bm{\delta}_{k,\ell}^2$ defined in \eqref{eqn:fzdef} and \eqref{eqn:fzcentraldef} with $M = 10$, $\| d_{q,r}^\star \|^2 = 1$, $\sigma_{m,n}^2 = 1$, $\mumax = 0.01, 0.03, 0.05$.}
\label{fig:pdf_chi2}
\vspace{-1\baselineskip}
\end{figure}

\subsection{Dynamics of Diffusion with Adaptive Clustering}
Since both Type-I and Type-II errors decay exponentially with exponent proportional to $1/\mumax$, it is expected that incorrect clustering decisions will become rare as the iteration proceeds. We can therefore assume that enough iterations have elapsed and the first recursion \eqref{eqn:distributedadaptgroup}--\eqref{eqn:distributedcombinegroup} is operating in steady-state. Under these conditions, we can examine the dynamics of the second recursion \eqref{eqn:distributedadaptdynamic}--\eqref{eqn:distributedcombinedynamic} with adaptive clustering. 

From Assumption \ref{ass:topology}, correct clustering decisions split the underlying topology into $Q$ sub-networks one for each cluster. Within each cluster, correct clustering decisions merge all disjoint groups into a bigger group. Therefore, the resulting topology for the entire network will now consist of $Q$ separate sub-networks and each sub-network will be strongly-connected. In addition, since the step-sizes are sufficiently small, the decision statistics $\|\bm{w}_{\ell,i} - \bm{w}_{k,i}\|^2$ generated by the first recursion \eqref{eqn:distributedadaptgroup}--\eqref{eqn:distributedcombinegroup} in steady-state will be nearly time-invariant. The clustering decisions will therefore also be nearly time-invariant. Then, with high probability, the cooperative sub-neighborhoods $\{ \bm{\Ncal}_{k,i}^+ \}$ produced by \eqref{eqn:neighborhoodki+def} will become nearly time-invariant after the first recursion \eqref{eqn:distributedadaptgroup}--\eqref{eqn:distributedcombinegroup} reaches steady-state:
\be
\label{eqn:Neighborhoodconverge}
\bm{\Ncal}_{k,i}^+ \rightarrow \Ncal_k^+, \quad \mbox{as} \quad i \rightarrow \infty
\ee
for all $k$, where $\Ncal_k^+$ is from \eqref{eqn:Neighborhood+and-def}. 

In order to gain from enhanced cooperation via adaptive clustering, it is critical to choose proper combination policies for recursion \eqref{eqn:distributedadaptdynamic}--\eqref{eqn:distributedcombinedynamic}. From the discussion in Chapter 12 of \cite[p. 624-635]{Sayed14NOW}, we know that doubly-stochastic combination policies are able to exploit the benefit of cooperation when more agents are included in cooperation. For example, one can choose the Metropolis rule \cite[p. 664]{Sayed14NOW}, i.e., 
\be
\label{eqn:metropolisrule}
\bm{a}_{\ell k}'(i) = \left\{ 
\begin{aligned}
& \frac{1}{\max\{ |\bm{\Ncal}_{\ell,i}^+|, |\bm{\Ncal}_{k,i}^+| \}}, & \;\; & \ell \in \bm{\Ncal}_{k,i}^+ \backslash \{k\}  \\
& 1 - \sum_{n \in \bm{\Ncal}_{k,i}^+\backslash\{k\}} \bm{a}_{n k}'(i), & \;\; & \ell = k \\
& 0, & \;\; & \ell \in \Ncal_k \backslash \bm{\Ncal}_{k,i}^+  \\
\end{aligned}
\right.
\ee
When the combination coefficients $\{ \bm{a}_{\ell k}'(i) \}$ are chosen according to \eqref{eqn:metropolisrule}, their values are determined by the size of their cooperative sub-neighborhood $\bm{\Ncal}_{k,i}^+$. It is then obvious that coefficients $\{ \bm{a}_{\ell k}'(i) \}$ will tend to be constant values:
\be
\label{eqn:aellkconverge}
\bm{a}_{\ell k}'(i) \rightarrow a_{\ell k}', \quad \mbox{as} \quad i \rightarrow \infty
\ee
which will be determined by the size of $\Ncal_k^+$. Therefore, we can rewrite the second recursion \eqref{eqn:distributedadaptdynamic}--\eqref{eqn:distributedcombinedynamic} for small enough $\mumax$ and large enough $i$ as
\begin{subequations}
\begin{align}
\label{eqn:distributedadaptstable}
\bm{\psi}_{k,i}' & = \bm{w}_{k,i-1}' - \mu_k \wh{\nabla  J_k} (\bm{w}_{k,i-1}') \\
\label{eqn:distributedcombinestable}
\bm{w}_{k,i}' & = \sum_{\ell \in \Ncal_k^+ } a_{\ell k}' \bm{\psi}_{\ell,i}'
\end{align}
\end{subequations}
by using \eqref{eqn:Neighborhoodconverge} and \eqref{eqn:aellkconverge}. We collect the $\{ a_{\ell k}' \}$ into a matrix and denote it by $A'$. The matrix $A'$ is block diagonal and each block on its diagonal corresponds to a cluster. Recursion \eqref{eqn:distributedadaptstable}--\eqref{eqn:distributedcombinestable} only involves in-cluster cooperative learning for common minimizers, where all agents from a cluster form a single big group. Therefore, the performance analysis in Section \ref{sec:performancegroup} applies to this case as well.

\section{Simulation Results}

We first simulate a network consisting of $N = 200$ agents. Each agent observes a data stream $\{\bm{d}_k(i), \bm{u}_{k,i}; i \ge 0 \}$ that satisfies the linear regression model \cite{Sayed08}:
\be
\bm{d}_k(i) = \bm{u}_{k,i} w_k^o + \bm{v}_k(i)
\ee
where $\bm{d}_k(i) \in \mbbR$ is a scalar response variable and $\bm{u}_{k,i} \in \mbbR^{1\times M}$ is a row vector feature variable with $M = 2$. The feature variable $\bm{u}_{k,i}$ is randomly generated at every iteration by using a Gaussian distribution with zero mean and scaled identity covariance matrix $\sigma_{u,k}^2 I_M$. The model noise $\bm{v}_k(i)\in\mbbR$ is also randomly generated at every iteration by using another independent Gaussian distribution with zero mean and variance $\sigma_{v,k}^2$. The values of $\{\sigma_{u,k}^2\}$ and $\{\sigma_{v,k}^2\}$ are positive and randomly generated.

There are $Q = 2$ clusters in the network. The first $N_1 = 100$ agents belong to cluster $\Ccal_1$, i.e., $\Ccal_1 = \{1,2,\dots,100\}$. The second $N_2 = 100$ agents belong to cluster $\Ccal_2$, i.e., $\Ccal_2 = \{101,102,\dots,200\}$. The loading factors for the two clusters, namely, $w_1^\star$ and $w_2^\star$, are randomly generated. The step-size is uniform and is set to $\mu = 0.05$. The underlying topology that connects all agents is shown in Fig. \ref{fig:total_topology}. Agents from cluster $\Ccal_1$ are in red and agents from $\Ccal_2$ are in blue. We simulated the scenario where agents have some partial knowledge about the grouping at the beginning of the learning process. The partial knowledge is non-trivial, meaning that the groups $\{\Gcal_m\}$ used in the first recursion \eqref{eqn:distributedadaptgroup}--\eqref{eqn:distributedcombinegroup} are not just singletons. The topologies that reflect the $\{\Gcal_m\}$ are plotted in Figs. \ref{fig:initial_topology_group1} and \ref{fig:initial_topology_group2} for the two clusters. The Metropolis rule \eqref{eqn:metropolisrule} is used in both recursions, \eqref{eqn:distributedadaptgroup}--\eqref{eqn:distributedcombinegroup} and \eqref{eqn:distributedadaptdynamic}--\eqref{eqn:distributedcombinedynamic}.

\begin{figure}[h]
\centerline{
\subfloat[The initial topology with all links.]
{\includegraphics[width=2in]{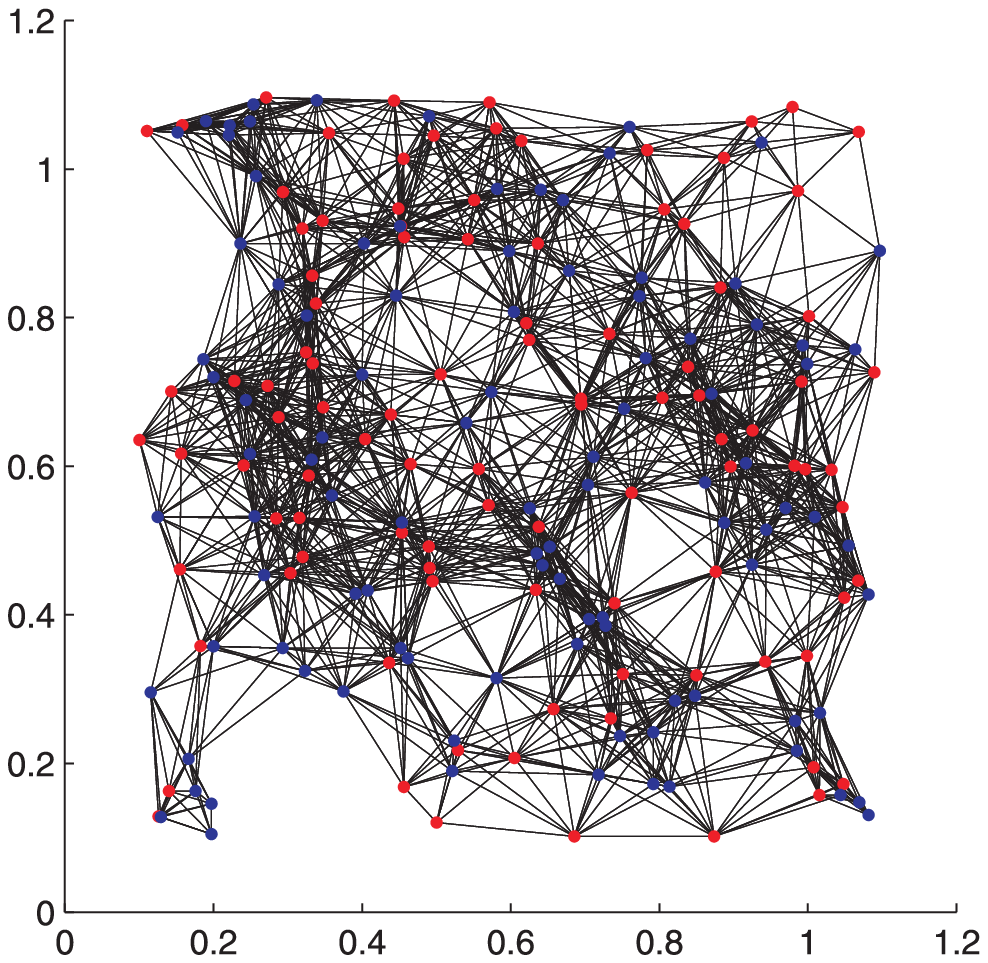}
\label{fig:total_topology}}
\hfil
\subfloat[Initial topology of cluster 1.]
{\includegraphics[width=2in]{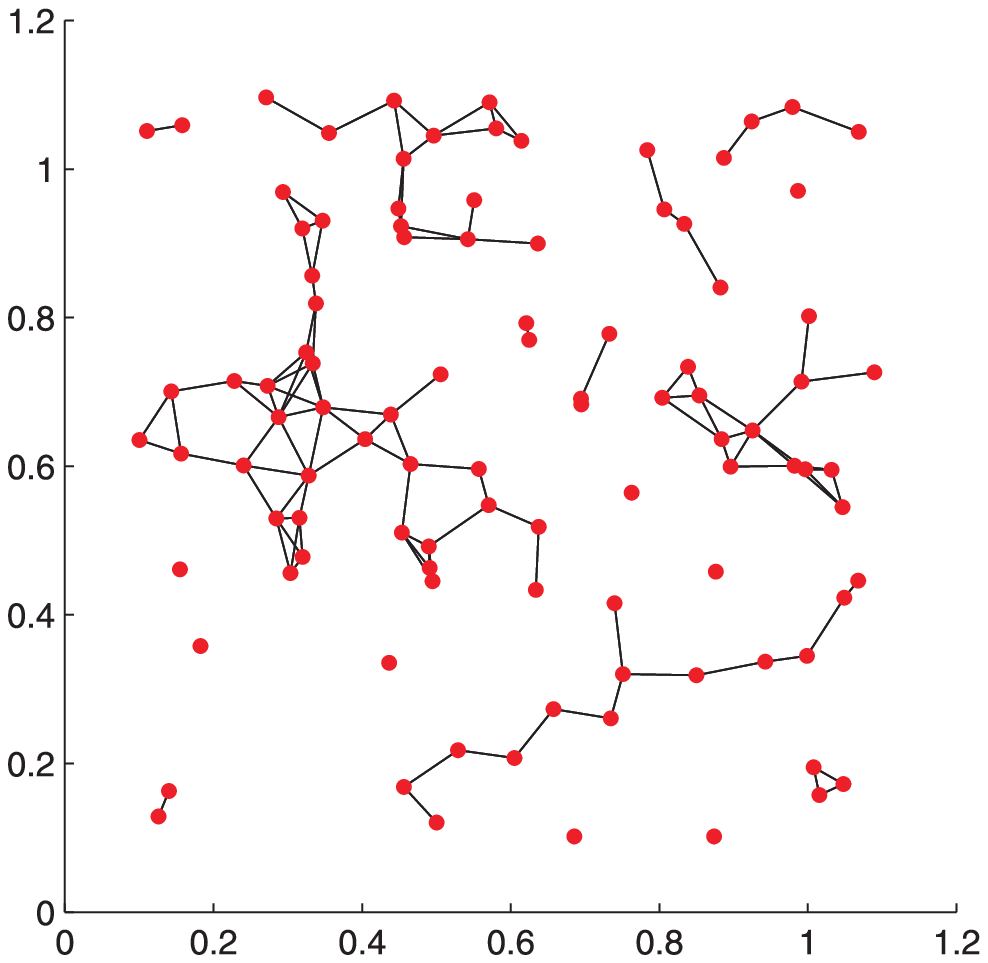}
\label{fig:initial_topology_group1}}
\hfil
\subfloat[Initial topology of cluster 2.]
{\includegraphics[width=2in]{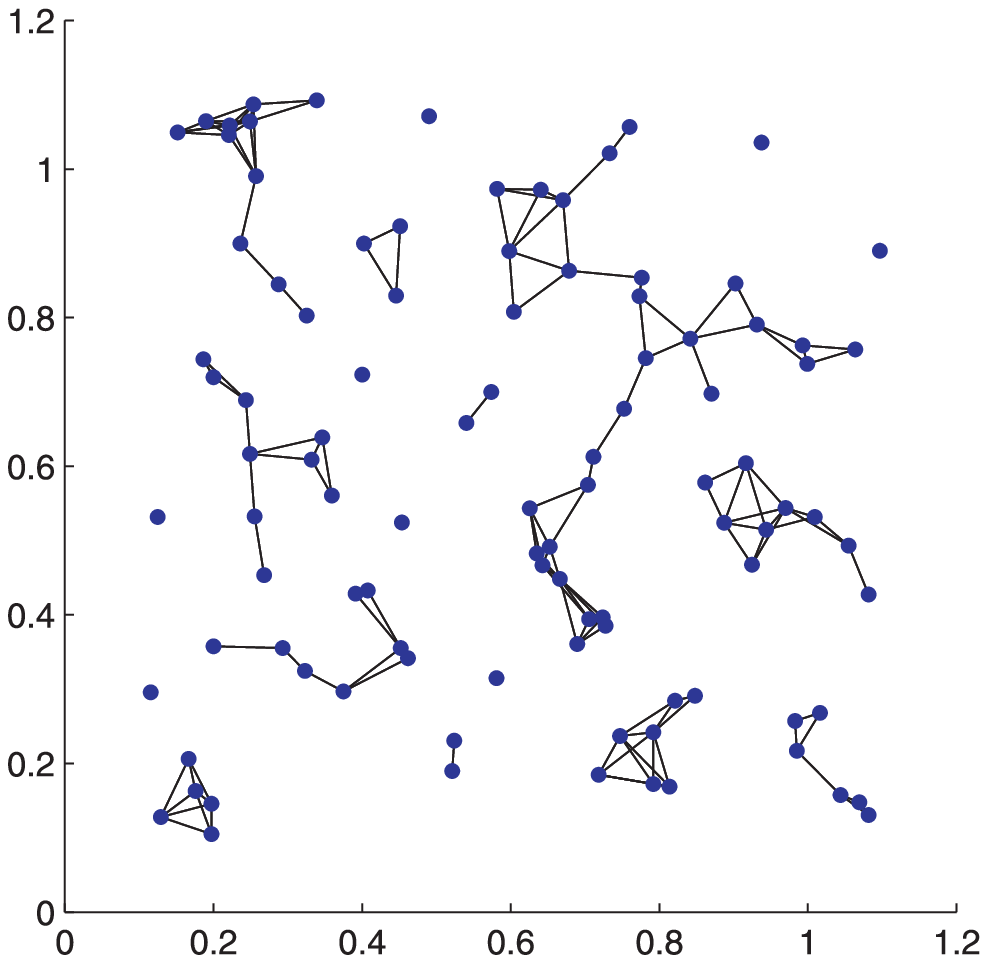}
\label{fig:initial_topology_group2}}}
\centerline{
\subfloat[The final topology at steady-state.]
{\includegraphics[width=2in]{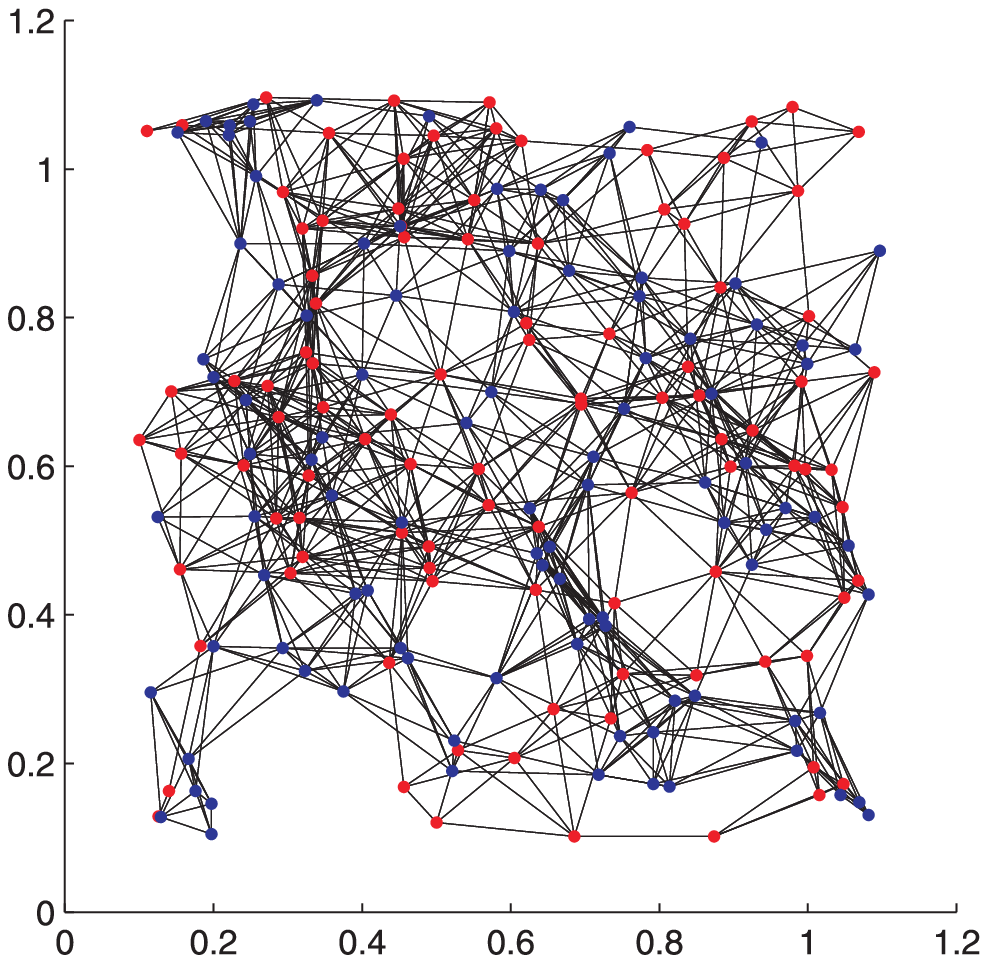}
\label{fig:split_topology}}
\hfil
\subfloat[Resulting topology of cluster 1.]
{\includegraphics[width=2in]{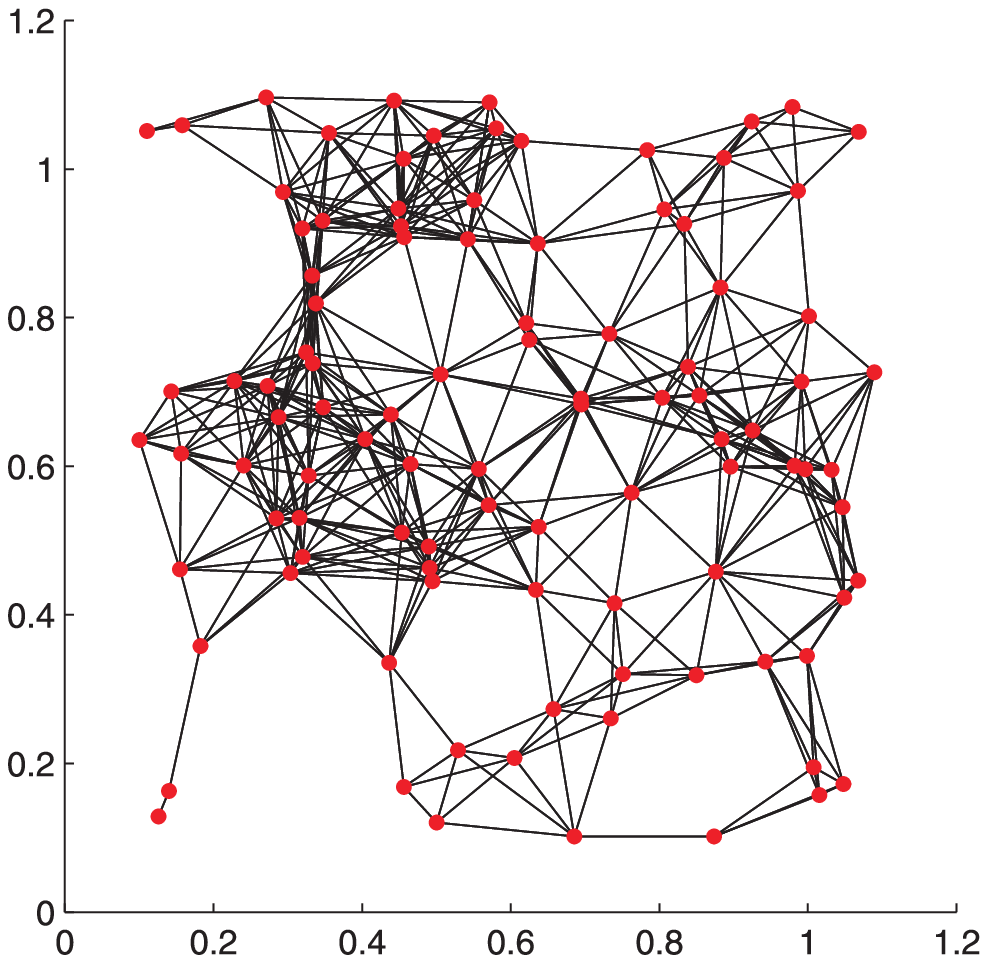}
\label{fig:resulting_topology_group1}}
\hfil
\subfloat[Resulting topology of cluster 2.]
{\includegraphics[width=2in]{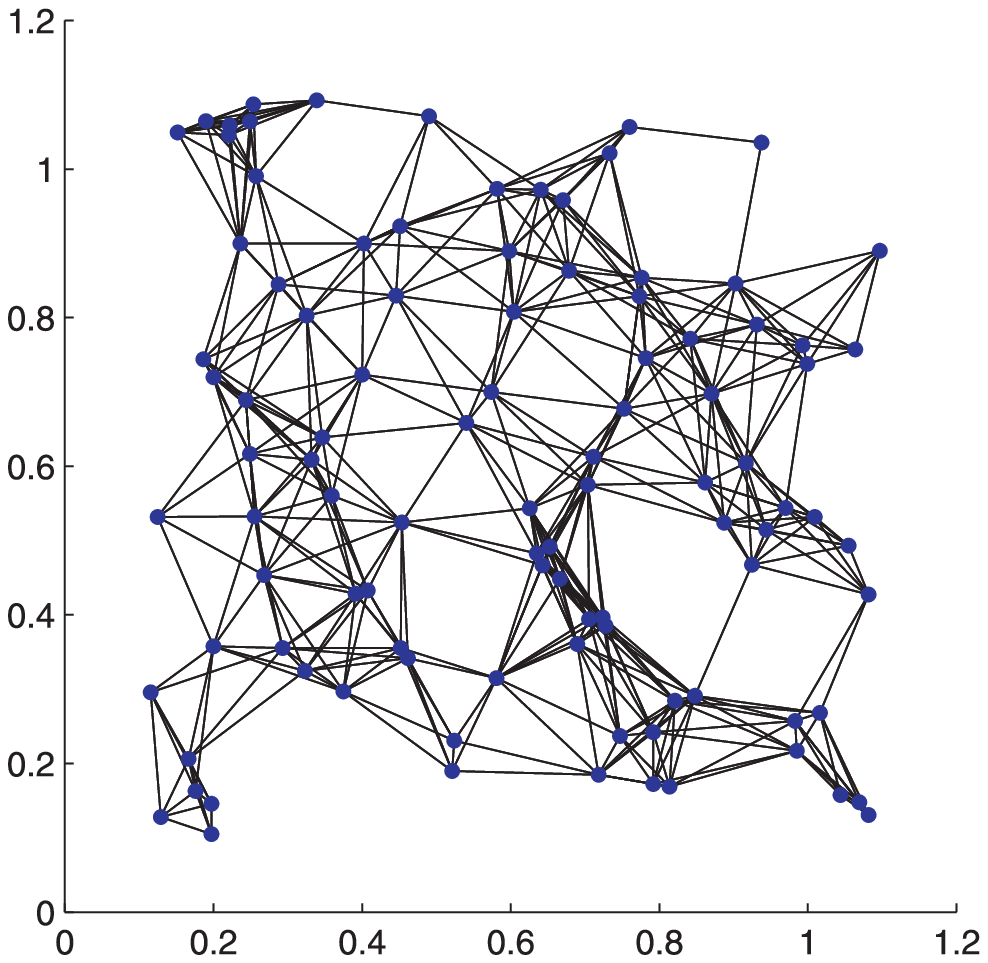}
\label{fig:resulting_topology_group2}}}
\caption{The underlying topology of the entire network where agents from different clusters are connected. As the learning process progresses, the disjoint groups in each cluster merge into a bigger group to enable collaborative learning among more agents. In steady-state, only in-cluster links remain active.}
\label{fig:cluster_topology}
\vspace{-1\baselineskip}
\end{figure}

As we explained before, in steady-state the clustering decisions become time-invariant and small groups in the same cluster merge into bigger groups. The links between neighbors within the same cluster are active while links to neighbors from different clusters are dropped. We plot the resulting topology in steady-state with active links in Fig. \ref{fig:split_topology}. Compared to Fig. \ref{fig:total_topology}, the underlying topology in Fig. \ref{fig:split_topology} is trimmed and split into two disjoint sub-networks. This result implies that the interference between two clusters is suppressed. The two sub-networks are themselves connected at steady-state and are shown in Figs \ref{fig:resulting_topology_group1} and \ref{fig:resulting_topology_group2}. Comparing the resulting cluster topologies in Figs \ref{fig:resulting_topology_group1} and \ref{fig:resulting_topology_group2} with the initial cluster topologies in Figs. \ref{fig:initial_topology_group1} and \ref{fig:initial_topology_group2}, it can be observed that all separate small groups from the same cluster merge into a bigger group and collaborative learning involving more agents emerges.

The MSD learning curves are plotted in Fig. \ref{fig:steadystateMSD} where the cluster MSDs are obtained by averaging over 100 trials. The cluster MSDs for the first recursion \eqref{eqn:distributedadaptgroup}--\eqref{eqn:distributedcombinegroup} are in black and green for clusters 1 and 2, respectively. The cluster MSDs for the second recursion \eqref{eqn:distributedadaptdynamic}--\eqref{eqn:distributedcombinedynamic} are in red and blue for clusters 1 and 2, respectively. Obviously both clusters improve their steady-state MSD performance on average by  forming larger clusters for cooperation.

\begin{figure}[h]
\centering
\includegraphics[width=4in]{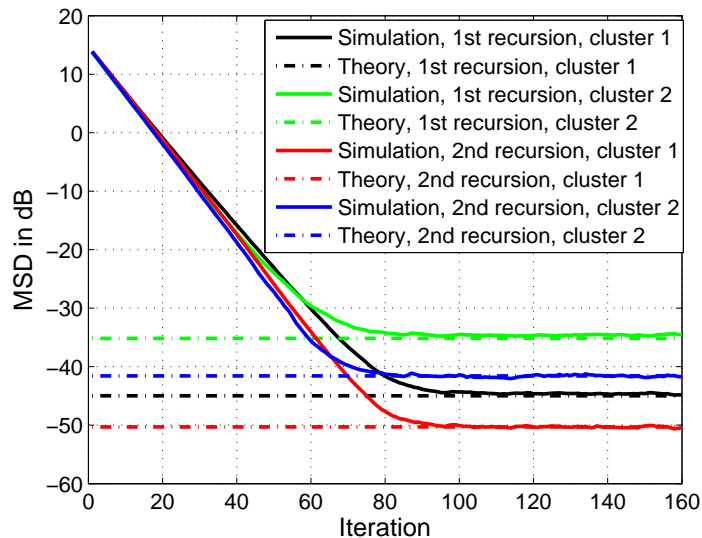}
\caption{The steady-state cluster average MSDs for the first recursion \eqref{eqn:distributedadaptgroup}--\eqref{eqn:distributedcombinegroup} and the second recursion \eqref{eqn:distributedadaptdynamic}--\eqref{eqn:distributedcombinedynamic}.}
\label{fig:steadystateMSD}
\vspace{-1\baselineskip}
\end{figure}

In the second simulation, we simulate a network with $N = 50$ nodes in $Q = 5$ clusters. The sizes of the five clusters are 8, 9, 10, 11, and 12, respectively. The initial topology is shown in Fig. \ref{fig:50init}. We choose the uniform step-size $\mu = 0.01$. After 1000 iterations, the resulting topology is separated into five clusters and is shown in Fig. \ref{fig:50res}, and the topologies for the five clusters are given in Figs. \ref{fig:50red}--\ref{fig:50magenta}, respectively. The MSD learning curves that are obtained by averaging over 500 trials match the theory well, as shown in Figs. \ref{fig:50MSD1} and \ref{fig:50MSD2}.

\section{Conclusions}

In this work we proposed a distributed strategy for  adaptive learning and clustering over multi-cluster networks. Detailed performance analysis is conducted and the results are supported by simulations. The proposed algorithm can be used in applications to segment heterogeneous networks into sub-networks to enhance in-cluster cooperation and suppress cross-cluster interference. It can also be applied to homogeneous networks to prevent intrusion or jamming by isolating malicious nodes from normal nodes. Furthermore, it can be used to trim and grow adaptive networks according to the objectives of the agents in the network.

\begin{figure}[h]
\centerline{
\subfloat[The initial topology with five clusters.]
{\includegraphics[width=3in]{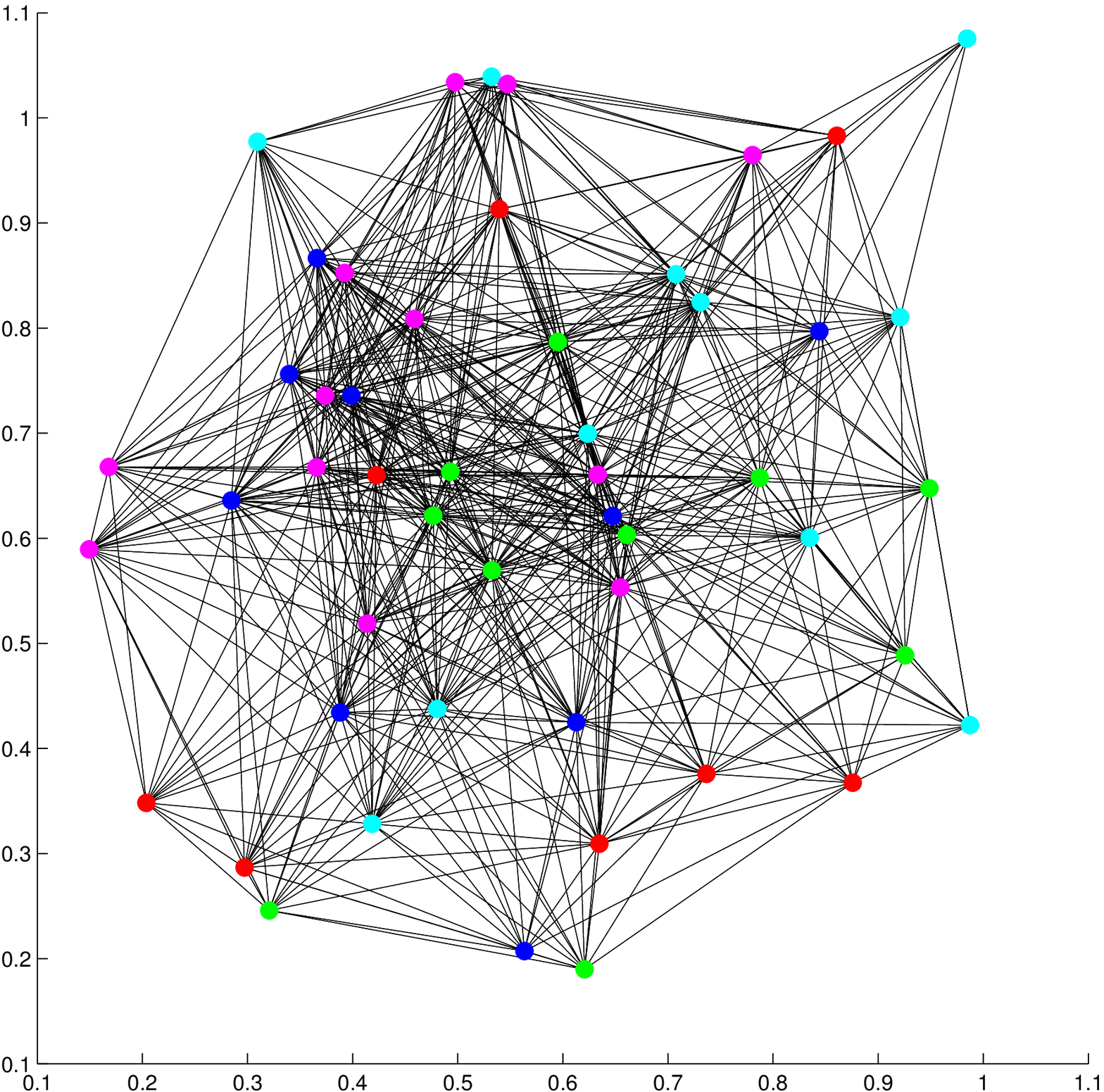}
\label{fig:50init}}
\hfil
\subfloat[The remaining topology with five clusters.]
{\includegraphics[width=3in]{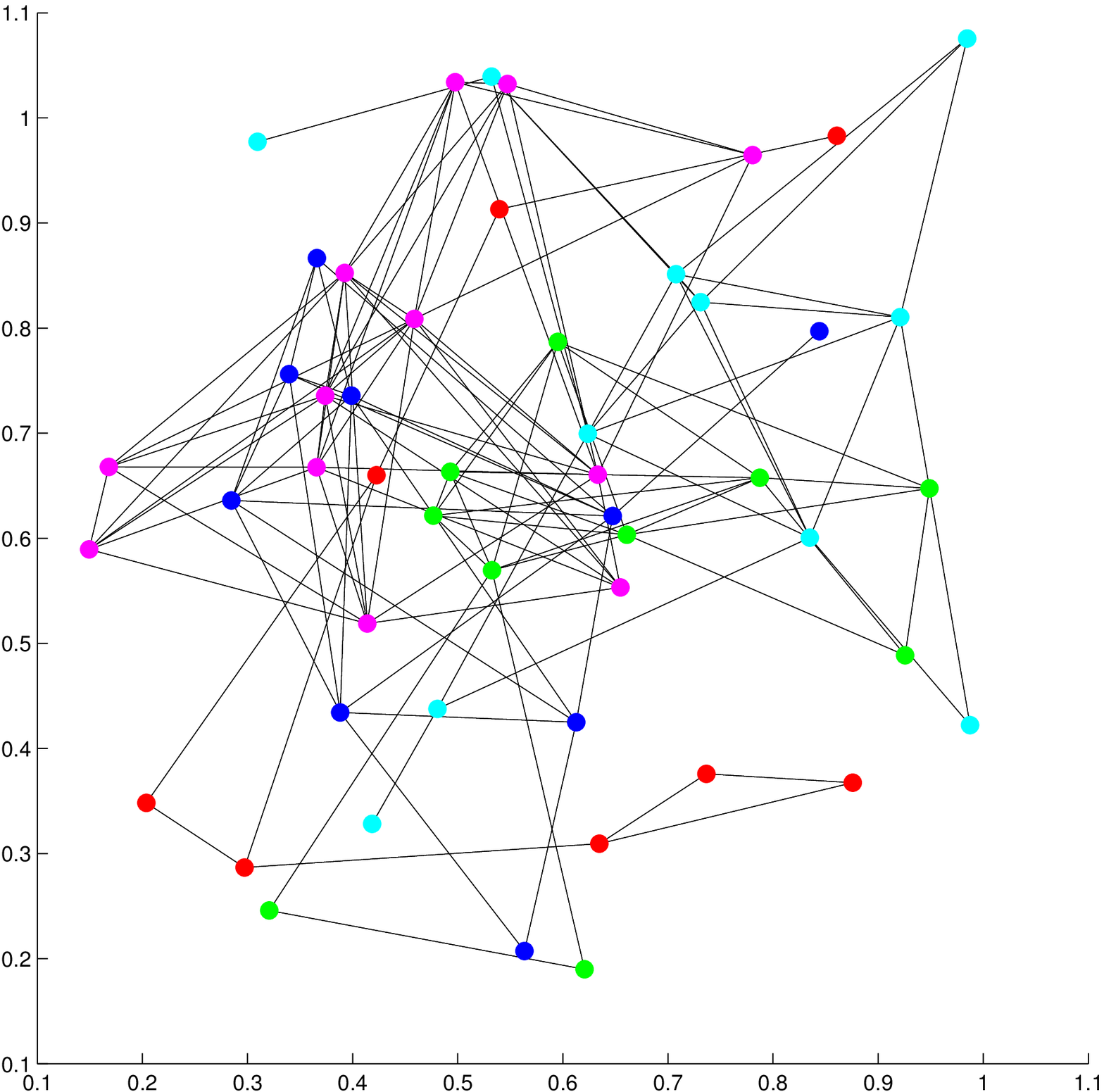}
\label{fig:50res}}}
\centerline{
\subfloat[Final topology of $\Ccal_1$.]
{\includegraphics[width=1.5in]{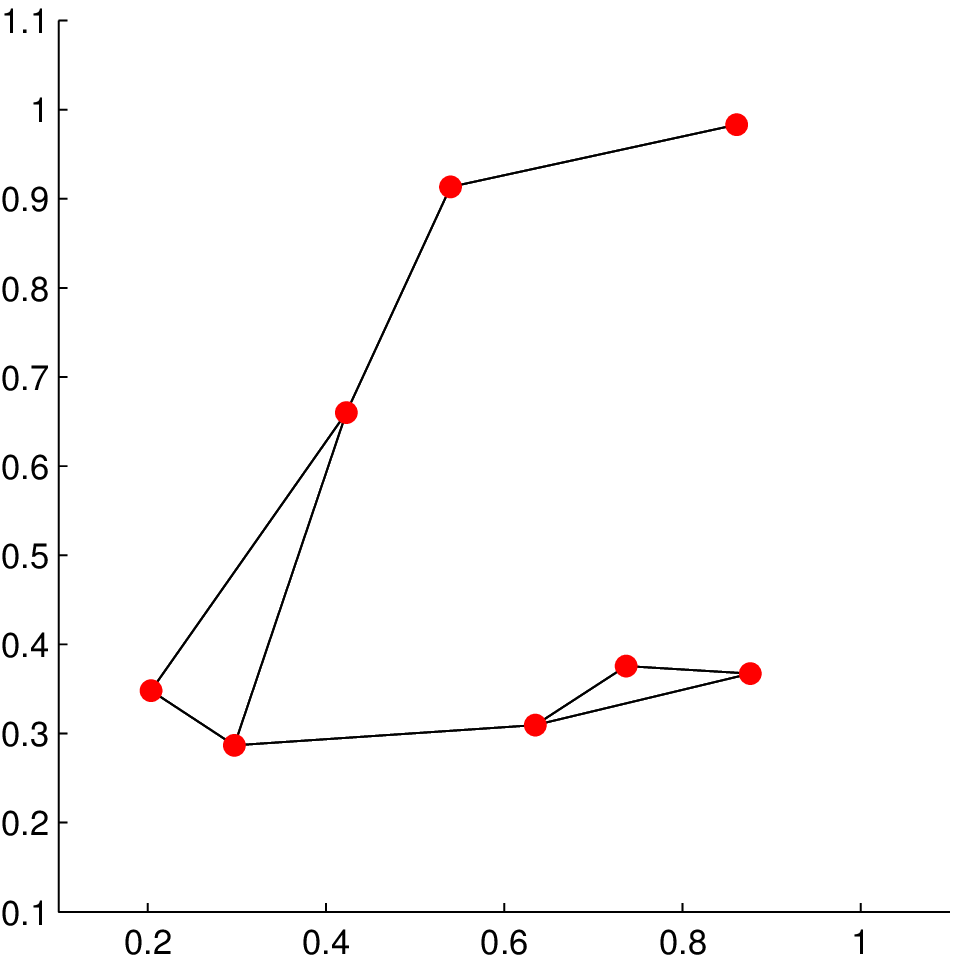}
\label{fig:50red}}
\hfil
\subfloat[Final topology of $\Ccal_2$.]
{\includegraphics[width=1.5in]{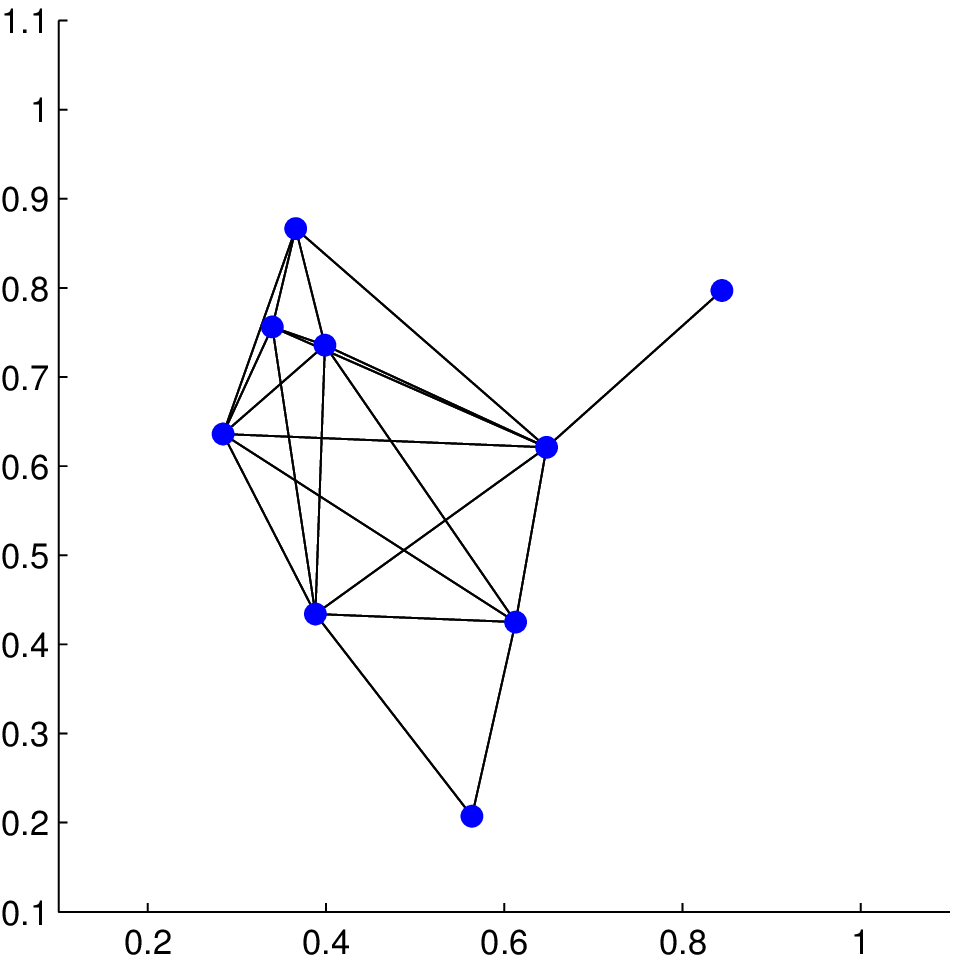}
\label{fig:50blue}}
\hfil
\subfloat[Final topology of $\Ccal_3$.]
{\includegraphics[width=1.5in]{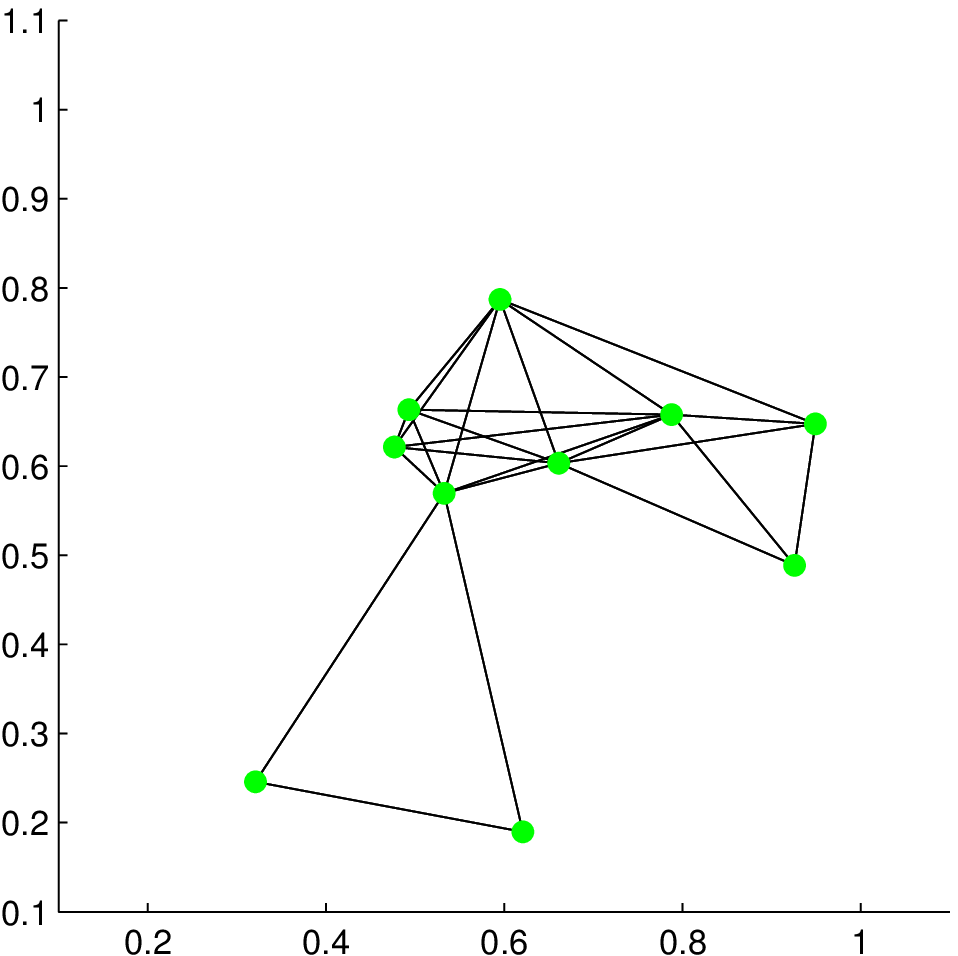}
\label{fig:50green}}}
\centerline{
\subfloat[Final topology of $\Ccal_4$.]
{\includegraphics[width=1.5in]{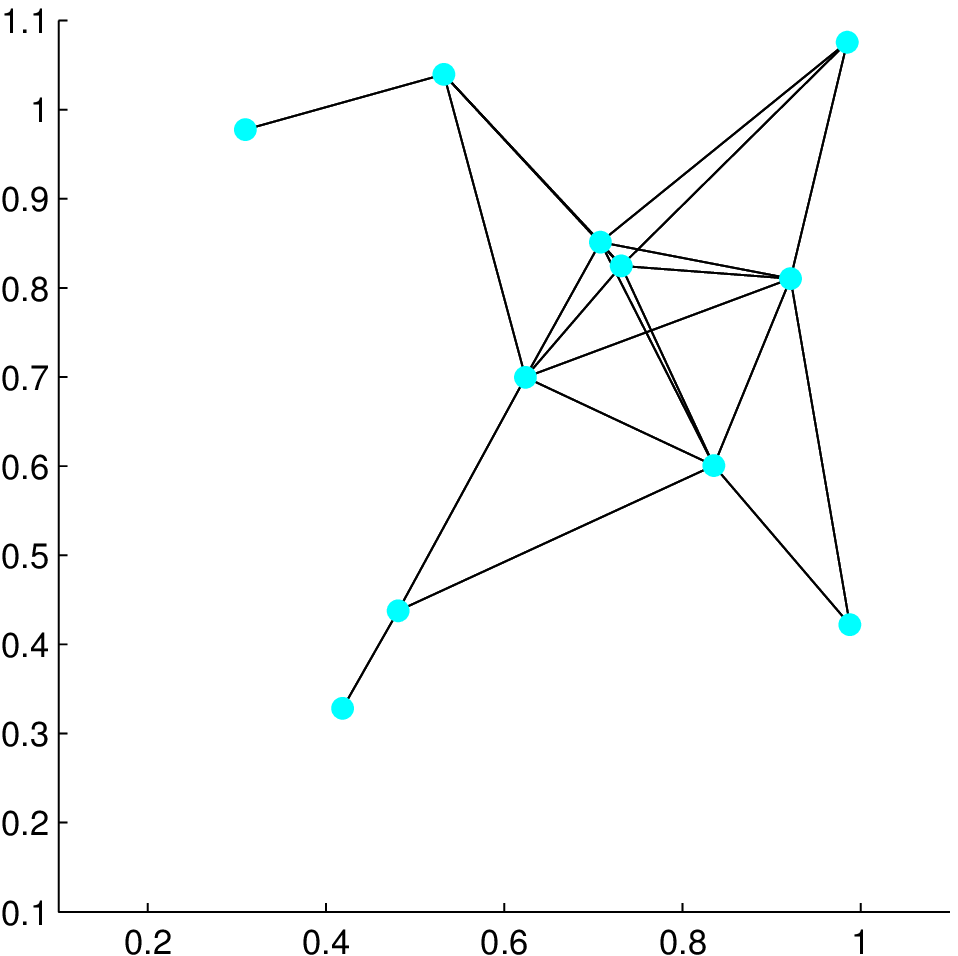}
\label{fig:50cyan}}
\hfil
\subfloat[Final topology of $\Ccal_5$.]
{\includegraphics[width=1.5in]{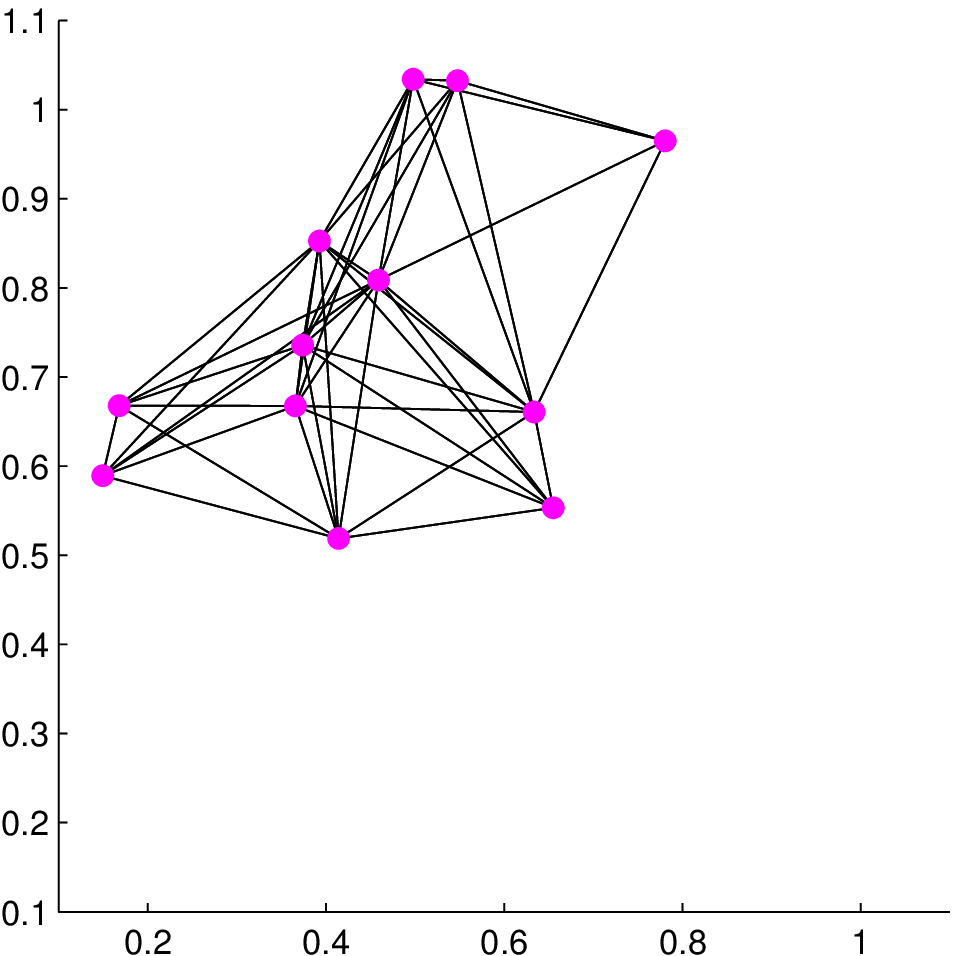}
\label{fig:50magenta}}
}
\caption{The initial topology with $N=50$ nodes and $Q=5$ clusters. In steady-state, the five clusters are successfully separated from each other while each cluster remains connected.}
\label{fig:secondsim}
\vspace{-1\baselineskip}
\end{figure}

\begin{figure}[h]
\centerline{
\subfloat[The MSD learning curves for the first recursion \eqref{eqn:distributedadaptgroup}--\eqref{eqn:distributedcombinegroup}.]
{\includegraphics[width=3.3in]{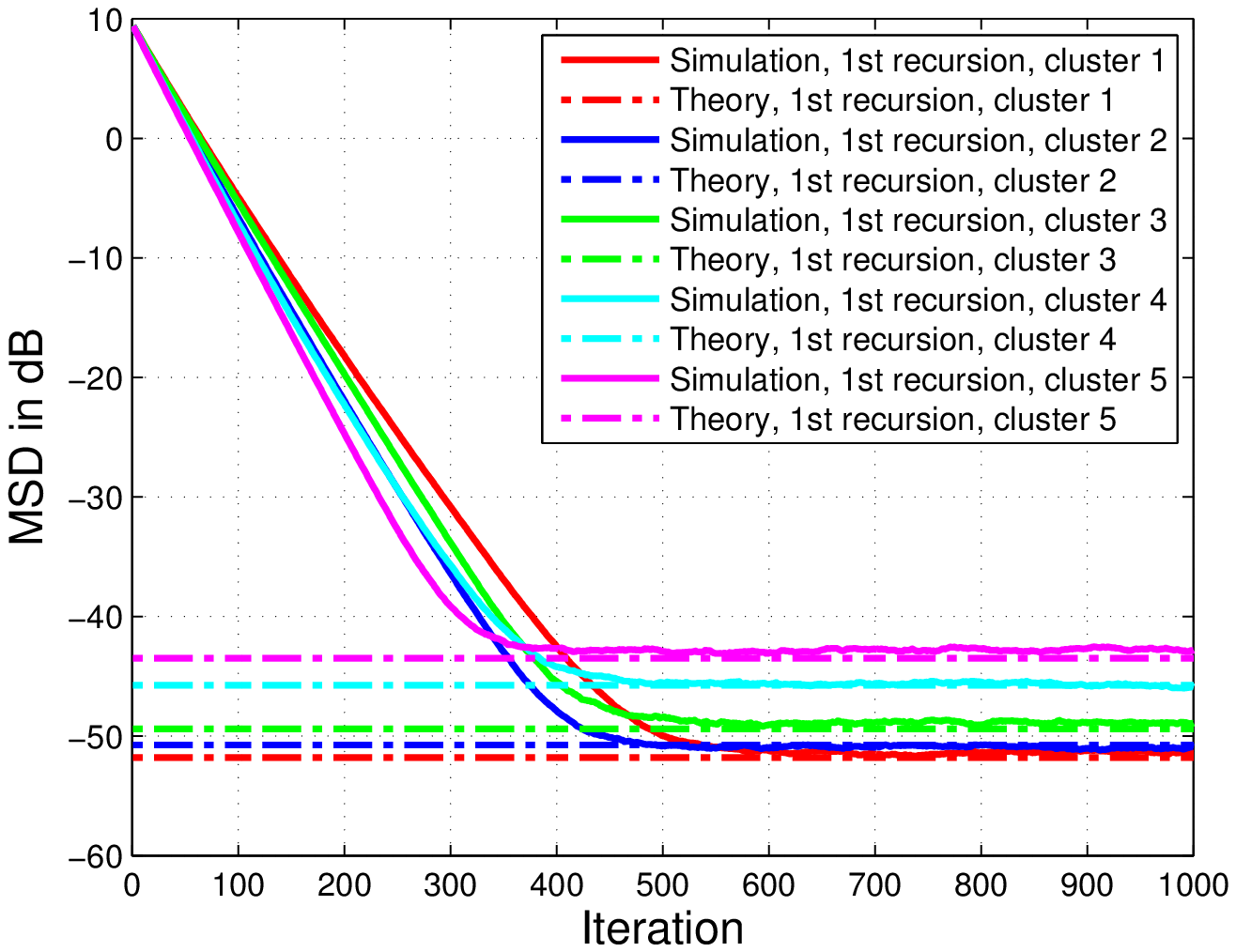}
\label{fig:50MSD1}}
\hfil
\subfloat[The MSD learning curves for the second recursion \eqref{eqn:distributedadaptdynamic}--\eqref{eqn:distributedcombinedynamic}.]
{\includegraphics[width=3.3in]{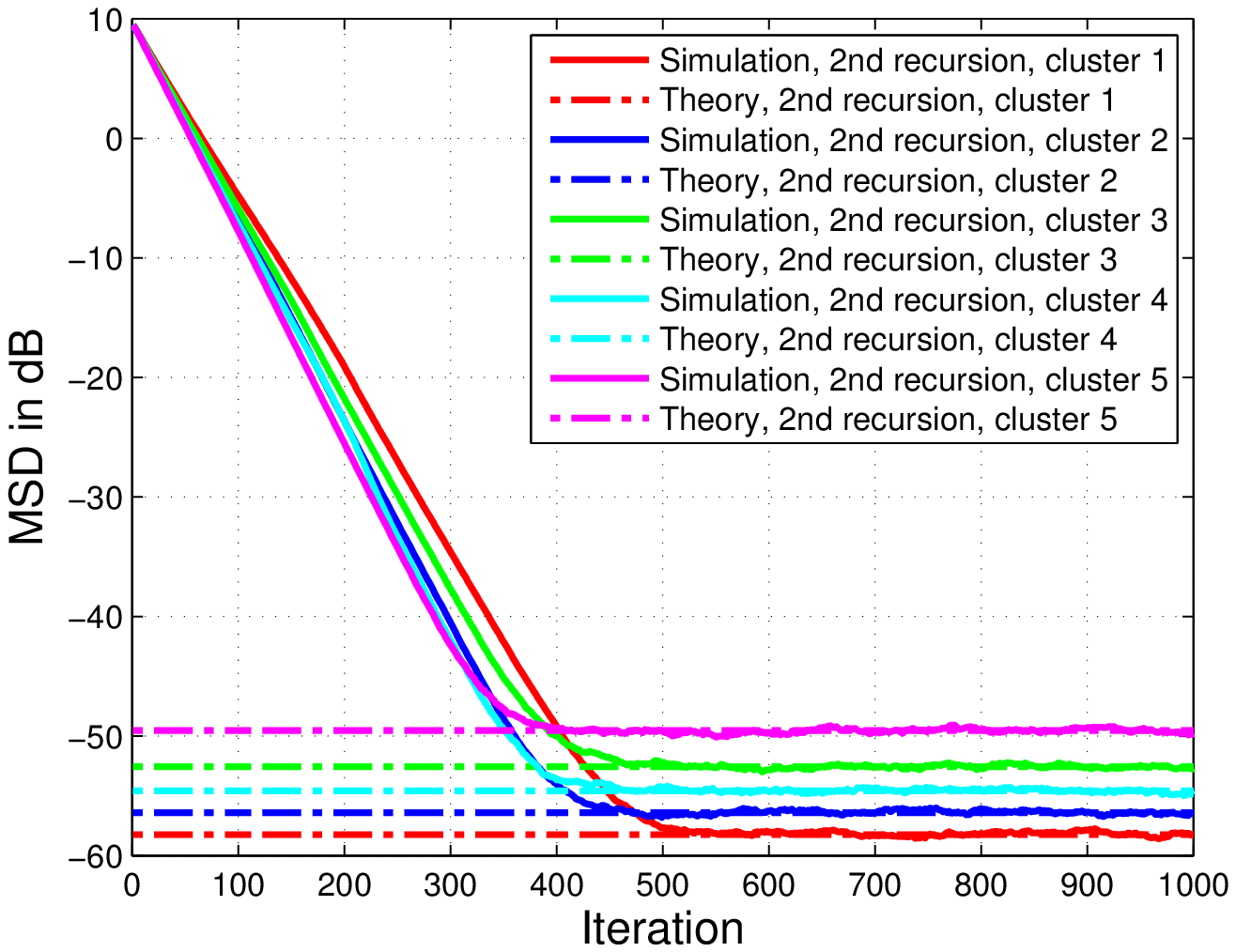}
\label{fig:50MSD2}}}
\caption{The MSD learning curves for the proposed distributed clustering and learning algorithm.}
\label{fig:secondsimMSD}
\vspace{-1\baselineskip}
\end{figure}

\appendices

\section{Proof of Lemma \ref{lemma:lowrankapprox}}
\label{app:lowrankapprox}
Since both models, \eqref{eqn:lowdimensionerrorrecursion} and \eqref{eqn:longtermerrorrecursion}, can be decoupled into $G$ separate recursions one for each group, it is sufficient to show that for sufficiently small step-sizes, and for any group $\Gcal_m$, it holds that
\be
\label{eqn:lowrankmodelgapgroup}
\limsup_{i\rightarrow \infty} \E \| \wt{\swb}_{m,i}^\longterm - \bar{\swb}_{m,i}^\lowdim \|^2 = O(\mumax^2)
\ee
where $\bar{\swb}_{m,i}^\lowdim$ is given by \eqref{eqn:zqidef}. We adopt a technique similar to the one used in the proof of Theorem 10.2 \cite[p. 557]{Sayed14NOW} to establish \eqref{eqn:lowrankmodelgapgroup} in the sequel. We introduce the Jordan decomposition of each $A_m$ \cite{Horn85, Sayed14NOW}:
\be
\label{eqn:Aeigendecompdef}
A_m = V_m J_m V_m^{-1} \defeq \begin{bmatrix}
p_m^g & V_{m,R}
\end{bmatrix} \begin{bmatrix}
1 & \\
& J_{m,\epsilon} \\
\end{bmatrix} \begin{bmatrix}
\one_{N_m^g} & V_{m,L}
\end{bmatrix}^\T
\ee
where $J_{m,\epsilon} \in \mbbC^{(N_m^g-1)\times(N_m^g-1)}$ consists of all stable Jordan blocks with $\epsilon$'s on the first lower off-diagonal, and $V_m$ is a non-singular complex matrix. Let
\begin{align}
\label{eqn:bigVmdef}
\Vcal_m & \defeq V_m \kron I_M \\
\label{eqn:bigJmdef}
\Jcal_m & \defeq J_m \kron I_M 
\end{align}
Multiplying $\Vcal_m^\T$ to both sides of \eqref{eqn:longtermerrorrecursiongroup} yields:
\be
\label{eqn:Umodifiederrorrecursiongroup}
\Vcal_m^\T \wt{\swb}_{m,i}^\longterm = \bar{\Bcal}_m \Vcal_m^\T \wt{\swb}_{m,i-1}^\longterm + \Jcal_m^\T \Vcal_m^\T \Mcal_m \ssb_{m,i}(\swb_{m,i-1})
\ee
where
\be
\label{eqn:barbigBmdef}
\bar{\Bcal}_m \defeq \Vcal_m^\T \Bcal_m (\Vcal_m^\T)^{-1} = \Jcal_m^\T - \Jcal_m^\T \Vcal_m^\T \Mcal_m \Hcal_m (\Vcal_m^\T)^{-1}
\ee
By \eqref{eqn:Aeigendecompdef} and \eqref{eqn:bigVmdef}, we have
\be
\label{eqn:wbarandwcheckdef}
\Vcal_m^\T \wt{\swb}_{m,i}^\longterm = \begin{bmatrix}
(p_m^g \kron I_M)^\T \wt{\swb}_{m,i}^\longterm  \\
(V_{m,R} \kron I_M)^\T \wt{\swb}_{m,i}^\longterm  \\
\end{bmatrix} \defeq \begin{bmatrix}
\bar{\bm{w}}_{m,i}^\longterm  \\
\check{\swb}_{m,i}^\longterm  \\
\end{bmatrix}
\ee
where $\bar{\bm{w}}_{m,i}^\longterm$ is an $M\times 1$ vector, $\check{\swb}_{m,i}^\longterm$ is an $(N_m^g-1)M \times1$ vector. It follows from \eqref{eqn:bigVmdef} and \eqref{eqn:zqidef} that
\be
\label{eqn:Uonewcent}
\Vcal_m^\T \bar{\swb}_{m,i}^\lowdim = (V_m^\T \one_{N_m^g} ) \kron \wt{\bm{w}}_{m,i}^\lowdim = 
\begin{bmatrix}
\wt{\bm{w}}_{m,i}^\lowdim  \\
0
\end{bmatrix}
\ee
since $\one_{N_m^g}$ is the first column of $(V_m^\T)^{-1}$ in \eqref{eqn:Aeigendecompdef}. Using \eqref{eqn:wbarandwcheckdef} and \eqref{eqn:Uonewcent}, we find that
\be
\E \| \wt{\swb}_{m,i}^\longterm - \bar{\swb}_{m,i}^\lowdim \|_{\Sigma_m}^2 = 
\E \| \bar{\bm{w}}_{m,i}^\longterm - \wt{\bm{w}}_{m,i}^\lowdim \|^2 + \E \| \check{\swb}_{m,i}^\longterm \|^2
\ee
where $\Sigma_m \defeq \Vcal_m \Vcal_m^\T$ is a positive-definite weighting matrix. Since $\| \Sigma_m \|$ is independent of $\mumax$, result \eqref{eqn:lowrankmodelgapgroup} holds if the following condition holds:
\be
\label{eqn:steadystateerrorapproxlowdimgroup}
\limsup_{i\rightarrow\infty} \E \| \bar{\bm{w}}_{m,i}^\longterm - \wt{\bm{w}}_{m,i}^\lowdim \|^2 + \E \| \check{\swb}_{m,i}^\longterm \|^2 = O(\mumax^2)
\ee
Using Eq. (10.78) in \cite[p. 563]{Sayed14NOW}, we know that
\be
\label{eqn:wcheckboundzero}
\limsup_{i\rightarrow\infty} \E \| \check{\swb}_{m,i}^\longterm \|^2 = O(\mumax^2)
\ee
From \eqref{eqn:Umodifiederrorrecursiongroup} and \eqref{eqn:wbarandwcheckdef}, the evolution of $\bar{\bm{w}}_{m,i}^\longterm$ is given by (see Eq. (9.61) from \cite[p. 514]{Sayed14NOW} for a similar derivation):
\be
\label{eqn:barwmigrouprecursion}
\bar{\bm{w}}_{m,i}^\longterm = D_m \bar{\bm{w}}_{m,i-1}^\longterm - D_{21}^\T \check{\swb}_{m,i-1}^\longterm + (p_m^g \kron I_M)^\T \Mcal_m \ssb_{m,i}(\swb_{m,i-1})
\ee
where $D_{21}^\T \defeq (p_m^g \kron I_M)^\T \Mcal_m \Hcal_m (V_{m,L} \kron I_M)$. Using \eqref{eqn:barwmigrouprecursion} and \eqref{eqn:lowdimensionerrorrecursiongroup}, we obtain
\be
\label{eqn:lowrankblock1recursion}
\bar{\bm{w}}_{m,i}^\longterm  - \wt{\bm{w}}_{m,i}^\lowdim = D_m (\bar{\bm{w}}_{m,i-1}^\longterm  - \wt{\bm{w}}_{m,i-1}^\lowdim ) - D_{21}^\T \check{\swb}_{m,i-1}^\longterm
\ee
We recognize that recursion \eqref{eqn:lowrankblock1recursion} has a form that is similar to the recursion for $\bar{\bm{b}}_i$ in Eq. (10.64) of \cite[p. 561]{Sayed14NOW} except that here in \eqref{eqn:lowrankblock1recursion} the driving noise term is absent. Therefore, we immediately get from Eq. (10.66) of \cite[p. 562]{Sayed14NOW} that
\be
\label{eqn:lowrankblock1recursion2}
\E \| \bar{\bm{w}}_{m,i}^\longterm - \wt{\bm{w}}_{m,i}^\lowdim \|^2 \le (1 - \sigma_{11} \mumax) \E \| \bar{\bm{w}}_{m,i-1}^\longterm - \wt{\bm{w}}_{m,i-1}^\lowdim \|^2 + \frac{\sigma_{21}^2 \mumax}{\sigma_{11}} \E \| \check{\swb}_{m,i-1}^\longterm \|^2
\ee
for some constants $\sigma_{11} > 0$ and $\sigma_{21} > 0$. Substituting \eqref{eqn:wcheckboundzero} into \eqref{eqn:lowrankblock1recursion2} yields
\be
\label{eqn:lowrankblock1recursion3}
\E \| \bar{\bm{w}}_{m,i}^\longterm - \wt{\bm{w}}_{m,i}^\lowdim \|^2 \le (1 - \sigma_{11} \mumax)  \E \| \bar{\bm{w}}_{m,i-1}^\longterm - \wt{\bm{w}}_{m,i-1}^\lowdim \|^2 + O(\mumax^3)
\ee
for large enough $i$. Therefore, it follows from \eqref{eqn:lowrankblock1recursion3} that
\be
\label{eqn:wbarboundzero}
\limsup_{i\rightarrow\infty} \E \| \bar{\bm{w}}_{m,i}^\longterm - \wt{\bm{w}}_{m,i}^\lowdim \|^2 = O(\mumax^2)
\ee
Combining \eqref{eqn:wcheckboundzero} and \eqref{eqn:wbarboundzero} proves \eqref{eqn:steadystateerrorapproxlowdimgroup}.

\section{Proof of Lemma \ref{lemma:lowrankerrorcov}}
\label{app:lowrankerrorcov}
Let us examine the evolution of the covariance matrix of $\wt{\swb}_i^\lowdim$, which is defined by
\be
\Theta_i \defeq \E [\wt{\swb}_i^\lowdim (\wt{\swb}_i^\lowdim)^\T ]
\ee
Using \eqref{eqn:bigRsidef} and \eqref{eqn:martingaledifference}, we get from \eqref{eqn:lowdimensionerrorrecursion} that
\be
\label{eqn:Phirecursiondef}
\Theta_i = \Dcal \Theta_{i-1} \Dcal + \Pcal^\T \Mcal [\E \Rcal_{s,i}(\swb_{i-1} )] \Mcal \Pcal
\ee
We next introduce the fixed-point covariance recursion
\be
\label{eqn:Phiorecursiondef}
\Theta_i^{\ss} = \Dcal \Theta_{i-1}^{\ss} \Dcal + \Pcal^\T \Mcal \Rcal_{s,i}( \sw^o ) \Mcal \Pcal
\ee
Let 
\be
\label{eqn:DeltaThetaandRsidef}
\Delta \Theta_i \defeq \Theta_i - \Theta_i^{\ss}, \;\;
\Delta \Rcal_{s,i} \defeq \E \Rcal_{s,i}( \swb_{i-1}) - \Rcal_{s,i}(\sw^o)
\ee
The difference matrix $\Delta \Theta_i$ evolves by the following recursion: 
\be
\label{eqn:deltaPhiirecursion}
\Delta \Theta_i = \Dcal \Delta \Theta_{i-1} \Dcal + \Pcal^\T \Mcal \Delta \Rcal_{s,i} \Mcal \Pcal
\ee
We bound the difference matrix $\Delta \Rcal_{s,i}$ by 
\begin{align}
\label{eqn:bounddeltaR}
\| \Delta \Rcal_{s,i} \| & \stackrel{(a)}{\le} \E \| \Rcal_{s,i}(\swb_{i-1}) - \Rcal_{s,i}(\sw^o) \| \nn \\
& \stackrel{(b)}{\le} \kappa_s \E \| \wt{\swb}_{i-1} \|^{\gamma_s} \nn \\
& \stackrel{(c)}{\le} \kappa_s \left( \E \| \wt{\swb}_{i-1} \|^4 \right)^{\gamma_s/4} 
\end{align}
where step (a) is by using Jensen's inequality; step (b) is by using \eqref{eqn:lipschitzcovariance} from Assumption \ref{ass:gradienterrors}; and step (c) is by applying Jensen's inequality again to the concave function $x^{\gamma_s/4}$ for $\gamma_s \le 4$ and $x\ge0$. As $i \rightarrow \infty$, we get from \eqref{eqn:bounddeltaR} that
\be
\label{eqn:bounddeltaR2}
\limsup_{i\rightarrow\infty} \| \Delta \Rcal_{s,i} \| = O(\mumax^{\gamma_s/2})
\ee
by using \eqref{eqn:4thorderasymptoticbound}. From Eq. (9.286) in \cite[p. 548]{Sayed14NOW}, we have
\be
\label{eqn:boundIQMmumaxH}
\| \Dcal \| = \max_m \| D_m \| \le 1 - \sigma \mumax
\ee
for some $\sigma > 0$. Using the triangle inequality and the sub-multiplicativity property of norms, we have from \eqref{eqn:deltaPhiirecursion} that
\begin{align}
\label{eqn:bounddiffphiiandphio}
\!\! \| \Delta \Theta_i \| & \le \| \Dcal \Delta \Theta_{i-1} \Dcal \| + \| \Pcal^\T \Mcal \Delta \Rcal_{s,i} \Mcal \Pcal \| \nn \\
\!\! & \le \| \Dcal \|^2 \| \Delta \Theta_{i-1} \| + \mumax^2 \| \Pcal \|^2 \| \Delta \Rcal_{s,i} \| \nn \\
\!\! & \le (1 - \sigma \mumax) \| \Delta \Theta_{i-1} \| + \mumax^2 \| \Pcal \|^2 \| \Delta \Rcal_{s,i} \| \!\!
\end{align}
where in the last step we used \eqref{eqn:boundIQMmumaxH} and the fact that $0< 1-\sigma \mumax < 1$. Then, as $i \rightarrow \infty$, we get from \eqref{eqn:bounddeltaR2} and \eqref{eqn:bounddiffphiiandphio} that
\be
\label{eqn:bounddiffphiiandphio2}
\limsup_{i\rightarrow\infty} \| \Delta \Theta_i \| \le \sigma^{-1} \mumax \| \Pcal \|^2 ( \limsup_{i\rightarrow\infty} \| \Delta \Rcal_{s,i} \| ) = O(\mumax^{1+ \gamma_s/2})
\ee
Now, since $\Dcal$ is stable and in view of \eqref{eqn:convergentcovariance}, the fixed-point recursion \eqref{eqn:Phiorecursiondef} converges as $i \rightarrow \infty$. At steady-state, the limit $\Theta_{\infty}^{\ss} \defeq \lim_{i \rightarrow \infty} \Theta_i^{\ss}$ of \eqref{eqn:Phiorecursiondef} satisfies the discrete Lyapunov equation \eqref{eqn:ThetaDTLEdef} by identifying $\Theta \equiv \Theta_{\infty}^{\ss}$.

\section{Proof of Theorem \ref{theorem:blockstructure}}
\label{app:block}
From Lemmas \ref{lemma:approxerrorrecursion} and \ref{lemma:lowrankapprox},
\begin{align}
\label{eqn:winetclosetozi2}
& \lim_{\mumax \rightarrow 0} \limsup_{i\rightarrow \infty} \mumax^{-1} \E \| \wt{\swb}_i - \bar{\swb}_i^\lowdim \|^2 \nn \\
& \le \lim_{\mumax \rightarrow 0} \limsup_{i\rightarrow \infty} \mumax^{-1} \E \| \wt{\swb}_i - \wt{\swb}_i^\longterm + \wt{\swb}_i^\longterm - \bar{\swb}_i^\lowdim \|^2 \nn \\
& \le \lim_{\mumax \rightarrow 0} \limsup_{i\rightarrow \infty} 2 \mumax^{-1} \E \| \wt{\swb}_i - \wt{\swb}_i^\longterm \|^2 + \lim_{\mumax \rightarrow 0} \limsup_{i\rightarrow \infty} 2 \mumax^{-1} \E \| \wt{\swb}_i^\longterm - \bar{\swb}_i^\lowdim \|^2 \nn \\
& = 0
\end{align}
Let 
\be
\label{eqn:Piilowdef}
\Pi_i^\lowdim \defeq \mumax^{-1} \E \bar{\swb}_i^\lowdim (\bar{\swb}_i^\lowdim)^\T
\ee
Then, by Jensen's inequality,
\begin{align}
\label{eqn:boundcovgap1}
\mumax \| \Pi_i - \Pi_i^\lowdim \| & \le \E \| \wt{\swb}_i \wt{\swb}_i^\T - \bar{\swb}_i^\lowdim (\bar{\swb}_i^\lowdim)^\T \| \nn \\
& = \E \| \wt{\swb}_i \wt{\swb}_i^\T - \bar{\swb}_i^\lowdim \wt{\swb}_i^\T + \bar{\swb}_i^\lowdim \wt{\swb}_i^\T - \bar{\swb}_i^\lowdim (\bar{\swb}_i^\lowdim)^\T \| \nn \\
& \le \E \| (\wt{\swb}_i - \bar{\swb}_i^\lowdim) \wt{\swb}_i^\T \| + \E \| \bar{\swb}_i^\lowdim (\wt{\swb}_i - \bar{\swb}_i^\lowdim)^\T \|
\end{align}
The second term on the RHS of \eqref{eqn:boundcovgap1} can be bounded by
\begin{align}
\label{eqn:boundcovgap2}
\E \| \bar{\swb}_i^\lowdim (\wt{\swb}_i - \bar{\swb}_i^\lowdim)^\T \| & = \E \| (\bar{\swb}_i^\lowdim - \wt{\swb}_i + \wt{\swb}_i ) (\wt{\swb}_i - \bar{\swb}_i^\lowdim)^\T \| \nn \\
& \le \E \| (\bar{\swb}_i^\lowdim - \wt{\swb}_i) (\wt{\swb}_i - \bar{\swb}_i^\lowdim)^\T \| + \E \| \wt{\swb}_i (\wt{\swb}_i - \bar{\swb}_i^\lowdim)^\T \| \nn \\
& = \E \| \bar{\swb}_i^\lowdim - \wt{\swb}_i \|^2 + \E \| \wt{\swb}_i (\wt{\swb}_i - \bar{\swb}_i^\lowdim)^\T \|
\end{align}
Substituting \eqref{eqn:boundcovgap2} into \eqref{eqn:boundcovgap1} yields:
\be
\label{eqn:boundcovgap3}
\mumax \| \Pi_i - \Pi_i^\lowdim \| \le 2 \E \| (\wt{\swb}_i - \bar{\swb}_i^\lowdim) \wt{\swb}_i^\T \| + \E \| \bar{\swb}_i^\lowdim - \wt{\swb}_i \|^2
\ee
The first term on the RHS of \eqref{eqn:boundcovgap3} can be bounded by
\begin{align}
\label{eqn:boundcovgap4}
\E \| (\wt{\swb}_i - \bar{\swb}_i^\lowdim) \wt{\swb}_i^\T \| & \le \E ( \| \wt{\swb}_i - \bar{\swb}_i^\lowdim \| \| \wt{\swb}_i \| ) \nn \\
& \le \sqrt{ \E \| \wt{\swb}_i - \bar{\swb}_i^\lowdim \|^2 \E \| \wt{\swb}_i \|^2 }
\end{align}
by using the Cauchy-Schwarz inequality. Substituting \eqref{eqn:boundcovgap4} into \eqref{eqn:boundcovgap3} yields:
\be
\label{eqn:boundcovgap5}
\| \Pi_i - \Pi_i^\lowdim \| \le 2 \sqrt{ \mumax^{-1} \E \| \wt{\swb}_i - \bar{\swb}_i^\lowdim \|^2 } \cdot \sqrt{ \mumax^{-1} \E \| \wt{\swb}_i \|^2 } + \mumax^{-1} \E \| \bar{\swb}_i^\lowdim - \wt{\swb}_i \|^2
\ee
Using \eqref{eqn:winetclosetozi2} and Theorem \ref{theorem:stability}, it follows from \eqref{eqn:boundcovgap5} that
\be
\label{eqn:boundcovgap6}
\lim_{\mumax \rightarrow 0} \limsup_{i \rightarrow \infty} \| \Pi_i - \Pi_i^\lowdim \| = 0
\ee
Noting that $\bar{\swb}_i^\lowdim$ is obtained by extending $\wt{\swb}_i^\lowdim$ via \eqref{eqn:zidef} and \eqref{eqn:zqidef}, we have
\be
\label{eqn:boundcovgap7}
\E \bar{\swb}_{m,i}^\lowdim (\bar{\swb}_{n,i}^\lowdim)^\T = ( \one_{N_m^g} \one_{N_n^g}^\T ) \kron \E \wt{\bm{w}}_{m,i}^\lowdim (\wt{\bm{w}}_{n,i}^\lowdim)^\T
\ee
for any $m$ and $n$. From \eqref{eqn:PiinftyequalPhi}, we know that
\be
\label{eqn:boundcovgap8}
\lim_{\mumax \rightarrow 0} \limsup_{i\rightarrow \infty} \| \mumax^{-1} \E \wt{\bm{w}}_{m,i}^\lowdim (\wt{\bm{w}}_{n,i}^\lowdim)^\T - \Phi_{m,n} \| = 0
\ee
where $\Phi_{m,n}$ denotes the $(m,n)$-th block of $\Phi$ with block size $M \times M$. It follows from \eqref{eqn:boundcovgap7} and \eqref{eqn:boundcovgap8} that
\be
\label{eqn:boundcovgap9}
\lim_{\mumax \rightarrow 0} \! \limsup_{i\rightarrow \infty} \! \| \mumax^{-1} \E \bar{\swb}_{m,i}^\lowdim (\bar{\swb}_{n,i}^\lowdim)^\T \!-\! ( \one_{N_m^g} \! \one_{N_n^g}^\T ) \kron \Phi_{m,n} \| \!=\! 0
\ee
Using \eqref{eqn:zidef}, \eqref{eqn:Pinetworkdef}, and \eqref{eqn:Piilowdef}, we get from \eqref{eqn:boundcovgap9} that
\be
\label{eqn:boundcovgap10}
\lim_{\mumax \rightarrow 0} \limsup_{i \rightarrow \infty} \| \Pi_i^\lowdim - \Pi \| = 0
\ee
Combining \eqref{eqn:boundcovgap6} and \eqref{eqn:boundcovgap10}, we arrive at \eqref{eqn:Piinftyoblockstructure}.

\section{Proof of Lemma \ref{lemma:normallowdimensional}}
\label{app:normal}

We establish this result by calling upon Theorem 1.1 from \cite[p. 319]{Kushner03}, which considers a stochastic recursion of the following form:
\be
\label{eqn:SGDdef}
\bm{x}_i = \bm{x}_{i-1} + \mu g(\bm{x}_{i-1}) + \mu \bm{v}_i
\ee
with step-size $\mu > 0$, update vector $g(\bm{x}_{i-1})$, and noise $\bm{v}_i$, satisfying the conditions:
\begin{enumerate}
\item The function $g(\cdot)$ is continuously differentiable and can be expanded as
\be
g(x) = g(x^o) + [\nabla g(x^o)]^\T (x - x^o) + o(\| x - x^o \|)
\ee
around a point $x^o$, where $\nabla g(\cdot)$ denotes the Jacobian of $g(\cdot)$, and $o(\cdot)$ is the ``small-$o$'' notation that represents higher order terms.

\item It holds that $x^o$ is the unique point that satisfies:
\be
\label{eqn:gradientbezero}
g(x^o) = 0
\ee

\item The Jacobian $A \defeq \nabla g(x^o)$ is a Hurwitz matrix (i.e., the real parts of the eigenvalues of $A$ are negative).

\item The noise process $\{\bm{v}_i; i\ge0\}$ is a martingale difference, i.e.,
\be
\label{eqn:martingale}
\E (\bm{v}_i | \F_{i-1} ) = 0
\ee
where $\F_{i-1}$ is the filtration defined by $\{ \bm{x}_i; i\ge 0\}$.

\item The noise $\bm{v}_i$ has an asymptotically bounded moment of order higher than 2, namely,
\be
\label{eqn:boundedvar}
\lim_{\mu \rightarrow 0} \limsup_{i \rightarrow \infty} \E \| \bm{v}_i \|^{2+p} < \infty
\ee
for some $p>0$.

\item The covariance matrices of the noise process $\{ \bm{v}_i; i\ge 0 \}$ converge to a positive semi-definite matrix $\Sigma \ge 0$:
\be
\lim_{\mu \rightarrow 0} \limsup_{i \rightarrow \infty} \| \E \bm{v}_i \bm{v}_i^\T - \Sigma \| = 0 
\ee
\end{enumerate}
Under these conditions, it holds that as $i\rightarrow\infty$ and $\mu \rightarrow 0$ asymptotically, the sequence $\{ \bm{x}_i/\sqrt{\mu} \}$ converges weakly to a Gaussian random distribution with mean $x^o$ and covariance matrix $C$, which is the unique solution to the continuous Lyapunov equation $A C + C A^\T = \Sigma$.

These conditions are satisfied by our recursion \eqref{eqn:lowdimensionerrorrecursionnew} by identifying $\wt{\swb}_i^\lowdim \equiv \bm{x}_i$, $\mumax \equiv \mu$, $-\bar{\Hcal} \wt{\swb}_{i-1}^\lowdim \equiv g(\bm{x}_{i-1})$, $\bm{v}_i \equiv \bar{\bm{s}}_i$. First, since $\bar{\Hcal}$ is positive-definite by \eqref{eqn:barbigHdef} and \eqref{eqn:barHmdef}, it is obvious that $x^o = 0$ is the unique point satisfying \eqref{eqn:gradientbezero}. Second, since $g(x) = -\bar{\Hcal} x$ and $x^o = 0$, condition 1) holds automatically with $[\nabla g(x^o)]^\T = -\bar{\Hcal}$. Third, it is easy to recognize that $A \equiv -\bar{\Hcal}$ is Hurwitz since $\bar{\Hcal}$ is positive-definite. Fourth, by \eqref{eqn:martingaledifference} from Assumption \ref{ass:gradienterrors}, condition \eqref{eqn:martingale} holds. Fifth, by \eqref{eqn:bounded4thorder} from Assumption \ref{ass:gradienterrors}, we have
\begin{align}
\label{eqn:boundbarsi4thorder}
\E \| \bar{\bm{s}}_i \|^4 & \le \| \Pcal \|^4 \E \| \ssb_i( \swb_{i-1} ) \|^4 \nn \\
& \le \| \Pcal \|^4 ( \alpha^2 \E \| \wt{\swb}_{i-1} \|^4 + \sigma_s^4 )
\end{align}
Using Theorem \ref{theorem:stability}, we get from \eqref{eqn:boundbarsi4thorder} that
\be
\lim_{\mumax \rightarrow 0} \limsup_{i \rightarrow \infty} \E \| \bar{\bm{s}}_i \|^4 \le \| \Pcal \|^4 ( O(\mumax^2) + \sigma_s^4 ) < \infty
\ee
which satisfies condition \eqref{eqn:boundedvar}. Sixth, we have from \eqref{eqn:barsidef} and \eqref{eqn:bigRsidef} that
\be
\label{eqn:covbarsi}
\E \bar{\bm{s}}_i \bar{\bm{s}}_i^\T = \mumax^{-2} \Pcal^\T \Mcal \E \Rcal_{s,i}( \swb_{i-1} ) \Mcal \Pcal
\ee
Let
\be
\Sigma_i \defeq \mumax^{-2} \Pcal^\T \Mcal \Rcal_{s,i}(\sw^o) \Mcal \Pcal
\ee
Then, using Jensen's inequality and \eqref{eqn:lipschitzcovariance} from Assumption \ref{ass:gradienterrors}, we have from \eqref{eqn:covbarsi} that
\be
\label{eqn:EbarsicovSigmaigap}
\| \E \bar{\bm{s}}_i \bar{\bm{s}}_i^\T - \Sigma_i \| \le \| \Pcal \|^2 \| \Delta \Rcal_{s,i} \|
\ee
where $\Delta \Rcal_{s,i}$ is from \eqref{eqn:DeltaThetaandRsidef}. Using \eqref{eqn:bounddeltaR2},
we further get
\be
\label{eqn:EbarsicovSigmaigap2}
\lim_{\mumax \rightarrow 0} \limsup_{i \rightarrow \infty} \| \E \bar{\bm{s}}_i \bar{\bm{s}}_i^\T - \Sigma_i \| = 0
\ee
Using \eqref{eqn:convergentcovariance}, we have
\be
\label{eqn:EbarsicovSigmaigap3}
\lim_{i\rightarrow\infty} \Sigma_i = \mumax^{-2} \Pcal^\T \Mcal \Rcal_s \Mcal \Pcal = \bar{\Rcal} \ge 0
\ee
where $\bar{\Rcal}$ is from \eqref{eqn:Rsdef}. It follows from \eqref{eqn:EbarsicovSigmaigap2} and \eqref{eqn:EbarsicovSigmaigap3} that
\begin{align}
\label{eqn:EbarsicovSigmaigap4}
\lim_{\mumax \rightarrow 0} \limsup_{i \rightarrow \infty} \| \E \bar{\bm{s}}_i \bar{\bm{s}}_i^\T -  \bar{\Rcal} \| = 0
\end{align}
Therefore, we conclude that the sequence $\{\wt{\swb}_i^\lowdim /\sqrt{\mumax}; i \ge 0\}$ converges weakly to the Gaussian random variable with zero mean and covariance matrix $\Phi$ that satisfies \eqref{eqn:continuoustimeLyapunovEqn}.

\section{Proof of Lemma \ref{lemma:convergence}}
\label{app:convergence}
We follow an argument similar to the proof of Theorem 2 from \cite[p. 256]{Shiryaev80} (which proves the result that convergence in moments implies convergence in distribution). Let $| f(x) | \le c$, i.e., bounded. Because a continuous function $f(x)$ is also \emph{uniformly} continuous in any \emph{bounded} region \cite[p. 54]{Shiryaev80}, for \emph{any} constant $\epsilon > 0$ \emph{and} for \emph{any} constant $b > 0$, there exists some $\delta_{\epsilon, b} > 0$, which depends on the choices of $\epsilon$ \emph{and} $b$, such that $|f(x) - f(y)| < \epsilon$ for $\| x \| < b$ \emph{and} $\| x - y \| < \delta_{\epsilon, b}$. Now, setting $b \defeq \sqrt{ 2 c \sigma^2 / \epsilon } > 0$, where $\sigma^2$ is from \eqref{eqn:convergencesigma}, and using conditional expectations, we have
\begin{align}
\label{eqn:boundmsgapzetaandeta}
\E | f(\bm{\zeta}_i) - f(\bm{\eta}_i) | & = \E [ | f(\bm{\zeta}_i) - f(\bm{\eta}_i) | \; | \; \| \bm{\zeta}_i - \bm{\eta}_i \| < \delta_{\epsilon, b}, \| \bm{\zeta}_i \| < b ] \cdot \Pr[\| \bm{\zeta}_i - \bm{\eta}_i \| < \delta_{\epsilon, b}, \| \bm{\zeta}_i \| < b] \nn \\
& \;\; + \E [ | f(\bm{\zeta}_i) - f(\bm{\eta}_i) | \; | \; \| \bm{\zeta}_i - \bm{\eta}_i \| < \delta_{\epsilon, b}, \| \bm{\zeta}_i \| \ge b ] \cdot \Pr[\| \bm{\zeta}_i - \bm{\eta}_i \| < \delta_{\epsilon, b}, \| \bm{\zeta}_i \| \ge b] \nn \\
& \;\; + \E [ | f(\bm{\zeta}_i) - f(\bm{\eta}_i) | \; | \; \| \bm{\zeta}_i - \bm{\eta}_i \| \ge \delta_{\epsilon, b} ] \cdot \Pr[\| \bm{\zeta}_i - \bm{\eta}_i \| \ge \delta_{\epsilon, b} ] 
\end{align}
The first term on the RHS of \eqref{eqn:boundmsgapzetaandeta} is bounded by
\be
\label{eqn:1stterm}
\mbox{1st term} \le \E [ \epsilon \; | \; \| \bm{\zeta}_i - \bm{\eta}_i \| < \delta, \| \bm{\zeta}_i \| < b ] \times 1 = \epsilon
\ee
Using the fact that $|f(x) - f(y)| \le |f(x)| + |f(y)| \le 2c$, and also the fact that the joint probability is bounded by any one of the marginal probabilities, i.e., $\Pr[A \cap B] \le \Pr[A]$ for any two events $A$ and $B$, the second term on the RHS of \eqref{eqn:boundmsgapzetaandeta} is bounded by
\be
\label{eqn:2ndterm}
\mbox{2nd term} \le 2c \, \Pr[\| \bm{\zeta}_i \| \ge b] \le \frac{2c \, \E \| \bm{\zeta}_i \|^2}{b^2} =  \frac{\epsilon \, \E \| \bm{\zeta}_i \|^2}{\sigma^2}
\ee
where we used Chebyshev's inequality \cite[p. 47]{Shiryaev80}. Likewise, the third term on the RHS of \eqref{eqn:boundmsgapzetaandeta} is bounded by
\be
\label{eqn:3rdterm}
\mbox{3rd term} \le 2c \, \Pr[\| \bm{\zeta}_i - \bm{\eta}_i \| \ge \delta ] \le \frac{2c \, \E \| \bm{\zeta}_i - \bm{\eta}_i \|^2}{\delta^2}
\ee
Now, substituting \eqref{eqn:1stterm}--\eqref{eqn:3rdterm} into \eqref{eqn:boundmsgapzetaandeta}, we have
\be
\E | f(\bm{\zeta}_i) - f(\bm{\eta}_i) | \le \epsilon + \frac{\epsilon \, \E \| \bm{\zeta}_i \|^2}{\sigma^2} + \frac{2c \, \E \| \bm{\zeta}_i - \bm{\eta}_i \|^2}{\delta^2}
\ee
Using \eqref{eqn:convergencemoments} and \eqref{eqn:convergencesigma}, we end up with
\be
\label{eqn:boundmsgapzetaandeta2}
\lim_{\mumax \rightarrow 0} \limsup_{i \rightarrow \infty} \E | f(\bm{\zeta}_i) - f(\bm{\eta}_i) | \le 
2\epsilon
\ee
Since $\epsilon$ is arbitrary, result \eqref{eqn:convergenceweakly} follows from \eqref{eqn:boundmsgapzetaandeta2}.

\section{Proof of \eqref{eqn:chikl2testvar}}
\label{app:moments}
To simplify the notation, we drop the subscript of $\bm{d}_{k,\ell}$ and denote its mean by $\bar{d} \defeq \E \bm{d}$ and its covariance by $C \defeq \E (\bm{d} - \bar{d})(\bm{d} - \bar{d})^\T$. Since $\bm{d}$ is Gaussian, it holds that
\begin{align}
\label{eqn:Ed4}
\E \| \bm{d} \|^4 & = \E \| \bm{d} - \bar{d} + \bar{d} \|^4 \nn \\
& = \E [ \| \bm{d} - \bar{d} \|^2 + 2 (\bm{d} - \bar{d})^\T \bar{d} + \| \bar{d} \|^2 ]^2 \nn \\
& = \E \| \bm{d} - \bar{d} \|^4 + 2 \E \| \bm{d} - \bar{d} \|^2 \| \bar{d} \|^2 + \| \bar{d} \|^4  + 4 \bar{d}^\T \E [(\bm{d} - \bar{d}) (\bm{d} - \bar{d})^\T] \bar{d} \nn \\
& = \E \| \bm{d} - \bar{d} \|^4 + 2 \Tr(C) \| \bar{d} \|^2 + \| \bar{d} \|^4 + 4 \| \bar{d} \|_C^2 
\end{align}
where we used the fact that the odd order moments of $\bm{d} - \bar{d}$ is zero. Likewise,
\begin{align}
\label{eqn:ed22}
( \E \| \bm{d} \|^2 )^2 & = ( \E \| \bm{d} - \bar{d} + \bar{d} \|^2 )^2 \nn \\
& = ( \E \| \bm{d} - \bar{d} \|^2 + \| \bar{d} \|^2 )^2 \nn \\
& = [\Tr(C)]^2 + 2 \Tr(C) \| \bar{d} \|^2 + \| \bar{d} \|^4
\end{align}
From \eqref{eqn:Ed4} and \eqref{eqn:ed22}, we have
\be
\label{eqn:eded2}
\E \| \bm{d} \|^4 - ( \E \| \bm{d} \|^2 )^2 = \E \| \bm{d} - \bar{d} \|^4 - [\Tr(C)]^2 + 4 \| \bar{d} \|_C^2 
\ee
From Lemma A.2 of \cite[p. 11]{Sayed08}, it can be verified that
\be
\label{eqn:edd4}
\E \| \bm{d} - \bar{d} \|^4 = [\Tr(C)]^2 + 2 \Tr(C^2)
\ee
Substituting \eqref{eqn:edd4} into \eqref{eqn:eded2} yields:
\be
\E \| \bm{d} \|^4 - ( \E \| \bm{d} \|^2 )^2 = 2 \Tr(C^2) + 4 \| \bar{d} \|_C^2 
\ee

\bibliographystyle{IEEEbib}
\bibliography{MyAbrv,clustering}

\end{document}

%% file: MyDefinitions.tex
\usepackage{hyperref}
\usepackage{xr}
\usepackage[table]{xcolor}
\usepackage{array}
\usepackage{bbding}
\usepackage{cite}
\usepackage{graphicx}
\usepackage{amssymb}
\usepackage{amsmath}
\usepackage{amsfonts}
\usepackage{algorithmic}
\usepackage{algorithm}
\usepackage{mdwmath}
\usepackage{mdwtab}
\usepackage{eqparbox}
\usepackage{fixltx2e}
\usepackage{stfloats}
\usepackage{dsfont}
\usepackage{mathrsfs}
\usepackage{bbm}
\usepackage{bm}
\usepackage{threeparttable}
\usepackage{multirow}
\usepackage[caption=false,font=footnotesize]{subfig}
\usepackage{centernot}
\usepackage{accents}
\usepackage{cleveref}
\usepackage{lscape}

\newtheorem{theorem}{Theorem}
\newtheorem{assumption}{Assumption}
\newtheorem{lemma}{Lemma}

\newtheorem{definition}{Definition}

\def\defeq{\triangleq}

\DeclareMathOperator*{\minimize}{\mathrm{minimize}}

\def\ds{\mathds}
\def\wt{\widetilde}
\def\wh{\widehat}

\def\col{{\mathrm{col}}}

\def\diag{{\mathrm{diag}}}

\def\Pr{{\mathbb{P}}}

\def\Tr{{\mathrm{Tr}}}

\def\kron{\otimes}

\def\clst{\textrm{c}}

\def\ss{{\textrm{fix}}}

\def\Var{\mathrm{Var}}

\def\mumax{\mu_{\max}}
\def\lowdim{\textrm{low}}
\def\longterm{\textrm{long}}
\def\sam{\textrm{sam}}

\def\one{\ds{1}}

\def\E{\mathbb{E}}
\def\F{\mathbb{F}}

\def\T{\mathsf{T}}

\def\nn{\nonumber}

\newcommand{\be}{\begin{equation}}
\newcommand{\ee}{\end{equation}}
\newcommand{\bea}{\begin{eqnarray*}}
\newcommand{\eea}{\end{eqnarray*}}

\newcommand{\mbbC}{\mathbb{C}}

\newcommand{\mbbH}{\mathbb{H}}

\newcommand{\mbbN}{\mathbb{N}}

\newcommand{\mbbR}{\mathbb{R}}

\newcommand{\Acal}{\mathcal{A}}
\newcommand{\Bcal}{\mathcal{B}}
\newcommand{\Ccal}{\mathcal{C}}
\newcommand{\Dcal}{\mathcal{D}}

\newcommand{\Gcal}{\mathcal{G}}
\newcommand{\Hcal}{\mathcal{H}}

\newcommand{\Jcal}{\mathcal{J}}

\newcommand{\Mcal}{\mathcal{M}}
\newcommand{\Ncal}{\mathcal{N}}
\newcommand{\Pcal}{\mathcal{P}}

\newcommand{\Rcal}{\mathcal{R}}

\newcommand{\Vcal}{\mathcal{V}}

\newcommand{\sw}{{\scriptstyle{\mathcal{W}}}}
\newcommand{\swb}{{\scriptstyle{\boldsymbol{\mathcal{W}}}}}
\newcommand{\ssb}{{\scriptstyle{\boldsymbol{\mathcal{S}}}}}